\documentclass[10pt,reqno]{amsart}
\usepackage{amssymb} 
\usepackage{latexsym}

\usepackage[T1]{fontenc}
\usepackage{fullpage}

\usepackage{enumerate}
\usepackage{blindtext}
\usepackage[pdftex]{graphicx} \usepackage{tikz-cd}
\usetikzlibrary{patterns}
\usepackage{bm}
\usepackage{commath}

 \renewcommand{\epsilon}{\varepsilon}

\usepackage{amsmath}
\usepackage{amsfonts}
\usepackage{amsthm}
\usepackage{amssymb, setspace}
\usepackage{color,graphicx,epsfig,geometry,hyperref,fancyhdr}
\usepackage[T1]{fontenc}
\usepackage{float}
\usepackage{commath}
\usepackage{bbm}
\usepackage{mathrsfs}
\usepackage{mathtools}
\usepackage{comment}
\usepackage{esint}

\usepackage{algorithm} 
\usepackage{algpseudocode} 

\newtheorem{theorem}{Theorem}[section]

\newtheorem{lemma}{Lemma}[section]
\newtheorem{definition}{Definition}[section]

\graphicspath{{paper-plots/}}

\newcommand{\bs}[1]{\boldsymbol{#1}}

\newcommand{\cl}[1]{\mathcal{#1}}

\newcommand{\al}{\alpha}

\newcommand{\eps}{\epsilon}

\newcommand{\la}{\lambda}

\newcommand{\Om}{\Omega}

\newcommand{\p}{\partial}

\usepackage[multiple]{footmisc}
\numberwithin{equation}{section}

\newtheorem{remark}{Remark}[section]

\theoremstyle{remark}

\newcommand{\tr}{\mathrm{tr}\,}

\newcommand{\rar}{\rightarrow}

\newcommand{\bbN}{\mathbb N}

\newcommand{\bbR}{\mathbb R}

\newcommand{\tens}{\otimes}

\author[G. Granados]{Govanni Granados*} \address[G. Granados]{Department of Mathematics, University of North Carolina, Chapel Hill, NC, 27599}
\thanks{*Author to whom any correspondence should be addressed.}
\email{ggranad@unc.edu}

\author[J.\,L. Marzuola]{Jeremy L. Marzuola} \address[J.\,L. Marzuola]{Department of Mathematics, University of North Carolina, Chapel Hill, NC, 27599}
\email{marzuola@math.unc.edu}

\author[C. Rodriguez]{Casey Rodriguez} \address[C. Rodriguez]{Department of Mathematics, University of North Carolina, Chapel Hill, NC, 27599}
\email{crodrig@email.unc.edu}

\begin{document}
\title{Recovering elastic subdomains with strain-gradient elastic interfaces from force measurements: the antiplane shear setting}
\begin{abstract}
We introduce and study a new inverse problem for antiplane shear in elastic bodies with strain-gradient interfaces. The setting is a homogeneous isotropic elastic body containing an inclusion separated by a thin interface endowed with higher-order surface energy. Using displacement–stress measurements on the exterior boundary, expressed through a certain Dirichlet-to-Neumann map, we show uniqueness in recovering both the shear and interface parameters, as well as the shape of the inclusion. 

{\color{black}
To address the inverse shape problem, we adapt the factorization method to account for the complications introduced by the higher-order boundary operator and its nontrivial null space. The resulting characterization relies on pairs of sampling points rather than a single-point test used in classical factorization methods. After fixing an interior reference point, the reconstruction procedure reduces to a single-point sampling algorithm. Focusing on the latter stage, numerical experiments illustrate the feasibility of the proposed reconstruction method and suggest that the framework has potential for the nondestructive detection of interior inhomogeneities, including damaged subvolumes.}
\end{abstract}
\maketitle

\section{\bf Introduction}
In this work, we initiate the study of a novel inverse problem with fundamental applications to nondestructive testing:
\begin{itemize}
	\item \textbf{Problem}: Given a three-dimensional homogeneous, isotropic, linear elastic body $\mathcal O$ with known material properties that contains an unknown, three-dimensional homogeneous, isotropic, linear elastic subregion $\mathcal D \neq \varnothing$, separated from the surrounding medium by a thin interface $\mathcal I$, determine the shape of the inclusion $\mathcal D$ as well as the mechanical properties of both the inclusion and the interface.
\end{itemize}

The available data will consist of boundary measurements in the form of displacement-stress pairs, encoded mathematically by the Dirichlet-to-Neumann map associated with a modification of the standard equations of linearized elasticity, first derived in this work. This system involves a new set of governing partial differential equations and boundary conditions that are of independent mathematical interest. 

As in much of classical physics (see, e.g.,  \cite{christodoulou2000action}), the novel governing system of field equations are derived from a variational perspective wherein we prescribe a Lagrangian action (or total potential energy) and apply Hamilton's principle. Thus, the primitives of the theory are the forms of the stored energy for the domains $\mathcal O \backslash \mathcal D$ and $\mathcal D$ and the interface $\mathcal I$. In this work, we consider the idealized setting of antiplane shear where the energies simplify considerably in complexity. We now describe our mathematical framework in more detail including the form that the energies take. 

\subsection{Stored energy of a three-dimensional body}
Let $\mathcal{B} \subseteq \bbR^3$ be a bounded domain modeling the reference configuration of a homogeneous, isotropic, linear elastic body. Let $\boldsymbol u : \mathcal B \rightarrow \bbR^3$ be a displacement field of $\mathcal B$. It is well-known (see, e.g.,  \cite{GurtinBook81}) that the total stored energy associated to the displacement field is given by 
\begin{align}
	E_{\mathcal B}(\boldsymbol u) = \int_{\mathcal B} \frac{\la}{2} (\mathrm{tr}\, \boldsymbol \epsilon(\boldsymbol u))^2 + \mu |\boldsymbol \epsilon(\boldsymbol u)|^2 \, dV,   \label{eq:3dstored}
\end{align}
where $\lambda, \mu \in \bbR$ are the Lam\'e parameters of the body satisfying the strong ellipticity conditions $\lambda + 2\mu > 0, \mu > 0$ and $\boldsymbol \epsilon$ is the infinitesimal strain tensor 
\begin{gather}
	\boldsymbol \epsilon(\boldsymbol u) = \epsilon_{ij}(\boldsymbol u) \boldsymbol e^i \tens \boldsymbol e^j = \frac{1}{2}\bigl [\nabla \boldsymbol u + (\nabla \boldsymbol u)^T \bigr ], \quad \boldsymbol u = u_i \boldsymbol e^i, \quad \nabla \boldsymbol u = \p_j u_i \boldsymbol e^i \tens \boldsymbol e^j, \\ \mathrm{tr}\, \boldsymbol \epsilon(\boldsymbol u) = \sum_{i=1}^3 \epsilon_{ii}(\boldsymbol u),  \quad |\boldsymbol \epsilon(\boldsymbol u)|^2 = \sum_{i,j=1}^3|\epsilon_{ij}(\boldsymbol u)|^2.   
\end{gather}
Here, repeated upper and lower indices implies summation and $\{ \boldsymbol e_i = \boldsymbol e^i \}_{i = 1}^3$ is a fixed orthonormal basis of $\bbR^3$. 

We now note the equivalence of the classical equations of linearized elasticity and Hamilton's variational principle applied to \eqref{eq:3dstored}. Indeed, suppose that a displacement field $\bs u_0 : \p \cl B \rar \bbR^3$ is given for the boundary of the body. Via integration by parts and the fundamental lemma of the calculus of variations, it follows that the classical placement problem of linearized elasticity  
\begin{align}
\begin{cases}
	\mathrm{div}\, \bs \sigma = \bs 0, \quad \mbox{on } \cl B,  \\
	\bs u = \bs u_0, \quad \mbox{on } \p \cl B,
\end{cases}
\end{align}
where $\bs \sigma = \la (\tr \eps(\bs u))\bs I + 2\mu \bs \eps(\bs u)$ is the stress tensor and $\mathrm{div}\, \bs \sigma = \sum_{i,j = 1}^3 \p_j \sigma_{ij} \bs e^i$ is equivalent to the following form of \textit{Hamilton's principle}: find $\bs u : \cl B \rar \bbR^3$ satisfying $\bs u = \bs u_0$ on $\p \cl B$ such that
\begin{align}
 \forall \bs \varphi \in C^\infty_c(\cl B), \quad	\frac{d}{d{\color{black}\tau}} E_{\cl B}(\bs u + {\color{black}\tau} \bs \varphi) \Big |_{{\color{black}\tau} = 0} = 0. 
\end{align} 

\subsection{Stored energy of the interface}
Returning to the setting of this work, we model the thin interface separating the inclusion $\cl D$ from the rest of the body $\cl O \backslash \cl D$ by a material surface $\mathcal I$ bonded to the boundary of the inclusion rather than a third three-dimensional domain, i.e., $\mathcal I$ and $\mathcal D$ form a \textit{surface-substrate system}. Mathematically, we prescribe a two dimensional stored energy $U$ depending on infinitesimal geometric strains associated to the surface $\mathcal I = \p \mathcal D$. The general theory of surface-substrate systems and interactions has been a source of intensive research over the past 50 years since the seminal works by Gurtin and Murdoch \cite{GurtinMurd75, GurtinMurd78} and Steigmann and Ogden \cite{SteigOgden97a, SteigOgden99}, and a review of the theoretical advances and applications following these works would not be possible here.    

We now describe the form of the stored energy for the interface $\p \cl D$ that we consider in this study. Let $\boldsymbol r: U \rightarrow \p \mathcal D$ be a local parameterization of $\p \mathcal D$. Here $U \subseteq \mathbb R^2$ is an open set with coordinates $(\theta^\al)$. We denote the corresponding natural basis vectors by $\boldsymbol A_\al := \boldsymbol r_{,\al}$ and the dual basis vectors by $\boldsymbol A^\al$ respectively. Here $\mbox{}_{,\al} = \p_{\theta^\al}$. The components of the metric are given by 
\begin{align}
	A_{\al \beta} := \boldsymbol A_{\al} \cdot \boldsymbol A_{\beta},
\end{align}  
the determinant of $(A_{\al \beta})$ is denoted by $A$, and the dual components are denoted by $(A^{\al \beta})$, so $A_{\al \gamma}A^{\gamma \beta} = \delta_\al^\beta$. The components of surface tensors are raised and lowered using $(A^{\al\beta})$ and $(A_{\al \beta})$ respectively. For a vector field $\boldsymbol u$ on $\p \mathcal D$ we define  
\begin{align}
\boldsymbol u_{\al|\beta} := \boldsymbol u_{,\al\beta} - \Gamma^{\gamma}_{\al \beta} \boldsymbol u_{,\gamma},
\end{align} 
where $(\Gamma^\gamma_{\beta\al})$ are the Christoffel symbols associated to $(A_{\al\beta})$. For a scalar valued function $\varphi$ on $\p \mathcal D$, its surface gradient is defined via $$\nabla_s \varphi = \varphi_{,\al} \boldsymbol A^\al,$$ and for a vector or tensor field $\boldsymbol E$ on $\p \mathcal D$, its surface gradient is defined via  
\begin{align}
	\nabla_s \boldsymbol E := \boldsymbol E_{,\al} \tens \boldsymbol A^\al. 
\end{align} 

For a displacement vector field $\boldsymbol u$ on $\p \mathcal D$, we define the \textit{infinitesimal surface strain tensor} 
\begin{align}
		\boldsymbol \epsilon_s(\boldsymbol u) := \frac{1}{2}\bigl (\boldsymbol A_\al \cdot \boldsymbol u_{,\beta} + \boldsymbol u_{,\al} \cdot \boldsymbol A_{,\beta} \bigr ) \boldsymbol A^\al \tens \boldsymbol A^\beta, 
\end{align}
and the \textit{infinitesimal relative normal curvature tensor} 
\begin{align}
	\boldsymbol \kappa_s(\boldsymbol u) := \boldsymbol N \cdot \boldsymbol u_{,\al|\beta} \boldsymbol A^\al \tens \boldsymbol A^\beta,
\end{align}
where $\bs N$ is the unit outward normal along $\p \cl D$.
We note that both $\boldsymbol \epsilon_s(\boldsymbol u)$ and $\boldsymbol \kappa_s(\boldsymbol u)$ are linear in $\boldsymbol u$. Physically, $\boldsymbol \epsilon_s(\boldsymbol u)$ is an infinitesimal measure of local stretching of $\p \mathcal D$ and $\boldsymbol \kappa_s(\boldsymbol u)$ is an infinitesimal measure of the change of curvature for $\p \mathcal D$ (see, e.g., \cite{steigmann2023lecture}). 

For the interface $\p \mathcal I = \p \mathcal D$, we prescribe a uniform hemitropic\footnote{For a discussion of hemitropy and the general theory of material symmetry, see, e.g. \cite{MurdochCohen80, MurdCohenAdd81, SteigOgden99, rodriguez2024elastic}.} energy $E_{\p \mathcal D}$, 
\begin{gather}
	E_{\p \mathcal D}(\boldsymbol u) = \int_{\p \mathcal D}  \left[  \frac{\la_s}{2}(\mathrm{tr}\, \boldsymbol \epsilon_s(\boldsymbol u))^2 + \mu_s |\boldsymbol \epsilon_s(\boldsymbol u)|^2 + \ell_s^2 \Bigl [ \frac{\la_s}{2}|\nabla_s \mathrm{tr}\, \boldsymbol \epsilon_s(\boldsymbol u)|^2 + \mu_s |\nabla_s \boldsymbol \epsilon_s(\boldsymbol u)|^2 \Bigr ] \right. \\
	\left. \quad + \frac{\zeta}{2}(\mathrm{tr}\, \boldsymbol \kappa_s(\boldsymbol u))^2 + \eta |\boldsymbol \kappa_s(\boldsymbol u)|^2 \right] \, dA,   \label{eq:surfacestoredenergy}
\end{gather}
where the surface constants satisfy the strong ellipticity conditions: 
\begin{gather}
	\la_s + 2\mu_s > 0, \mu_s > 0, \quad \ell_s > 0, \quad \zeta + 2\eta > 0, \eta > 0.  \label{eq:strongellipt}
\end{gather}

The surface-stored energy \eqref{eq:surfacestoredenergy} provides a versatile framework for modeling interface phenomena. When $\ell_s = \zeta = \eta = 0$, the resulting energy reduces to that of the Gurtin–Murdoch theory of surface-substrate interactions \cite{GurtinMurd75, GurtinMurd78}, in which the surface behaves as a membrane-like interface along the bulk's boundary. Setting $\ell_s = 0$ and assuming $\zeta + 2\eta > 0$ and $\eta > 0$, we recover the quadratic Steigmann–Ogden surface energy \cite{SteigOgden97a, SteigOgden99}, which models a shell- or plate-like interface similarly bonded to the boundary of $\mathcal D$. A more general theory, in which the surface energy also depends on surface strain gradients, was developed by Rodriguez in \cite{rodriguez2024elastic}, where it was shown that these higher-order effects confer resistance to geodesic distortion, that is, the deformation of geodesics into non-geodesics. Furthermore, in \cite{rodriguez2024elastic, rodriguez2024midsurface}, it was demonstrated that such surface-substrate models regularize singular behavior that arises in classical linearized elasticity, notably in the setting of mode-III antiplane shear loading for finite-length cracks, something neither the Gurtin–Murdoch nor Steigmann–Ogden models achieve. Finally, it was shown in \cite{researchplayground25} that for infinitesimal displacement gradients, \eqref{eq:surfacestoredenergy} arises as the leading cubic order-in-thickness contribution to the integrated-through-thickness stored energy for certain homogenous, isotropic, three-dimensional strain-gradient elastic layers. These results support the use of \eqref{eq:surfacestoredenergy} as a well-justified and physically meaningful model for interfaces in bodies that include failure related mechanisms.

It is worth emphasizing that both the Steigmann–Ogden and Rodriguez surface-substrate models incorporate energies involving second-order derivatives of the displacement field, distinguishing them from the classical Gurtin–Murdoch model and standard practice in classical continuum mechanics. The inclusion of such higher-order terms places these models within the broader framework of second-gradient and generalized continuum theories, a field with a deep and evolving history. The origins trace back to the work of Piola in 1846 \cite{Piola1846Book, dellIsola15} and were significantly developed by the Cosserat brothers \cite{Cosserats1909}, who introduced additional rotational degrees of freedom at each material point. The second half of the 20th century witnessed major advances through the seminal contributions of Toupin \cite{Toupin62, Toupin64}, Green and Rivlin \cite{GreenRivlin64a, GreenRivlin64b}, Mindlin \cite{Mindlin64a, Mindlin1965}, Mindlin and Eshel \cite{MindlinEshel1968}, and Germain \cite{Germain73a, Germain73b}. For modern perspectives and overviews of the field, see, for example, \cite{Askes2011, dellIsola17, Maugin17Book, dellIsola2020higher}. These developments form the foundation for contemporary higher-order theories, including the surface models used in this work.

\subsection{The governing field equations for antiplane shear and the associated inverse problem}
The field equations governing the equilibrium configuration of the body $\cl O$ with inclusion $\cl D$, interface $\p \cl D$, and prescribed boundary displacement $\bs u_0 : \p \cl D \rar \bbR^3$ are derived via Hamilton's principle: find $\bs u : \cl O \rar \bbR^3$ such that $\bs u = \bs u_0$ on $\p \cl O$ and
\begin{align}
	\forall \bs \varphi \in C^\infty_c(\cl B), \quad	\frac{d}{d{\color{black}\tau}} \Bigl [ E_{\cl O \backslash \cl D}(\bs u + {\color{black}\tau} \bs \varphi) +  E_{\cl D}(\bs u + {\color{black}\tau} \bs \varphi) + E_{\p \cl D}(\bs u + {\color{black}\tau} \bs \varphi) \Bigr ] \Big |_{{\color{black}\tau} = 0} = 0. \label{eq:Hamilton}
\end{align}
 
In this work, we consider the simpler idealized setting of antiplane shear wherein there exist $h > 0$, and smooth domains
\begin{align}
	\cl O = \Omega \times (-h,h), \quad \cl D = D \times (-h,h)
\end{align} 
and all displacement fields take the form 
\begin{align}
	\bs u = u(x_1, x_2) \bs e_3. 
\end{align}
We assume that $\p D$ is parameterized by arclength, $s$, with unit outward normal $\bs \nu$. Then $\p \cl D = \p D \times (-h,h)$ is parameterized by $(s,x_3)$ with outward unit normal $\bs N = \bs \nu$. It follows that  
\begin{gather}
	\boldsymbol \eps(\bs u) = \frac{1}{2}\p_1 u (\bs e_1 \tens \bs e_3 + \bs e_3 \tens \bs e_1) + \frac{1}{2}\p_2 u (\bs e_2 \tens \bs e_3 + \bs e_3 \tens \bs e_2), \\
	\boldsymbol \eps_s(\bs u) = \frac{1}{2}\p_s u (\bs t \tens \bs e_3 + \bs e_3 \tens \bs t). 
\end{gather} 

We insert antiplane shear displacements into \eqref{eq:Hamilton} to obtain the field equations governing equilibrium configurations for antiplane shears of $\cl O$. In particular, \textit{Hamilton's principle for antiplane shear displacements} is given by: find $u : \Omega \rar \bbR$ such that $u = f$ on $\p \Omega$ and  
\begin{equation} 
\begin{gathered} 
	\forall \varphi \in C^\infty_c(\Omega), \quad	\frac{d}{d{\color{black}\tau}} \left [ \int_{\Omega \backslash D} \frac{\mu_o}{2}|\nabla u + {\color{black}\tau} \nabla \varphi|^2 \, dx_1 \, dx_2 +  \int_{D} \frac{\mu_i}{2}|\nabla u + {\color{black}\tau} \nabla \varphi|^2 \, dx_1dx_2 \right. \\ 
	\left. + \int_{\p D} \bigl ( \frac{\mu_s}{2}|\p_s u + {\color{black}\tau} \p_s \varphi|^2 + \frac{\ell_s^2 \mu_s}{2} |\p_s^2 u + {\color{black}\tau} \p_s^2 \varphi|^2 \bigr) \, ds \right ] \bigg |_{{\color{black}\tau} = 0} = 0. \label{eq:Hamiltonantiplane}
\end{gathered} 
\end{equation}
Here $\mu_o$ and $\mu_i$ are the shear moduli of the known material parameterized by $\Omega \backslash D$ and $D$, respectively. By an appropriate nondimensionalization, we may set $\mu_o = 1$ and we denote $\mu_i$ by $\mu > 0$. In particular, it is important to note the following physical interpretations:
\begin{itemize}
\item If $\mu > 1$, then the shear modulus of the inclusion is larger than that of the bulk and $\cl D$ is stiffer than $\cl O \backslash \cl D$. Thus, $\cl D$ is composed of a ``stronger'' material than $\cl O \backslash \cl D$. 
\item If $\mu < 1$, then the shear modulus of the inclusion is smaller than that of the bulk and $\cl D$ is more pliable than $\cl O \backslash \cl D$. Thus, $\cl D$ is composed of a ``weaker'' material than $\cl O \backslash \cl D$. For example, $\cl D$ could model a \textit{damaged} portion of $\cl O$. 
\end{itemize}

{\color{black}For what follows we define 
\begin{equation}
[\![\partial_\nu u ]\!] := (\partial_{\nu} u^{+} - \mu\partial_{\nu} u^{-})\big\rvert_{\partial D}
\end{equation}
where the `$+$' symbol represents the trace on $\partial D$ from $\Omega \backslash \overline{D}$ and `$-$' represents the trace on $\partial D$ from $D$.} Via straightforward integration by parts and the fundamental lemma of the calculus of variations, we see that \eqref{eq:Hamiltonantiplane} holds if and only if
\begin{equation}\label{bvp}
	\begin{cases} 
		- \Delta u = 0 & \text{in} \enspace \Omega \backslash D, \\
		-\mu \Delta u = 0 & \text{in} \enspace D ,\\
		[\![\partial_\nu u ]\!] = \mathscr{B}(u)  & \text{on} \enspace \partial D,\\
		u = f & \text{on} \enspace \partial \Omega,
	\end{cases}
\end{equation}
where the boundary operator $\mathscr{B}$ is defined as
\begin{equation} \label{b-op}
	\mathscr{B}(u) =  \partial^{2}_{s} (\mu_s \ell_s^2 \partial^{2}_{s} u) - \partial_{s}(\mu_s \partial_{s} u).
\end{equation}
We abuse notation slightly and denote the outward unit normal on $\p \Omega$ also by $\bs \nu = \nu^1 \bs e_1 + \nu^2 \bs e_2$. Then we have that the stress along $\p \Omega$ necessary to support the boundary displacement $f$ is given by 
\begin{align}
	\bigl [
	\la_o (\tr \bs \eps(\bs u)) \bs I + 2 \bs \eps(\bs u)
	\bigr ] \bs \nu = [\nu^1 \p_1 u + \nu^2 \p_2 u]\bs e_3 = \p_\nu u \bs e_3.
\end{align}
Thus, the Dirichlet-to-Neumann map (DtN) $\Lambda(f) = \p_\nu u$ maps boundary displacement to the stress (magnitude) necessary to support it. The problem \eqref{bvp} represents the direct or forward problem: given a boundary displacement $f$, determine the resulting displacement of the body containing the inclusion and interface. The inverse problem we consider in this work can then be stated as: 
\begin{itemize}
    \item \textbf{Problem}: Determine the shape and mechanical parameters $\mu$, $\mu_s,$ and $\ell_s$ for $D$ using boundary displacement-stress information provided by the Dirichlet-to-Neumann map $\Lambda$.
\end{itemize}
Our shape reconstruction results will focus on the comparison between force measurements when an inclusion exists versus when it does not, the latter set of measurements considered known a priori. Hence, we also introduce the standard, unperturbed Dirichlet problem
\begin{equation}\label{lifting}
	\begin{cases} 
		- \Delta u_0 = 0 & \text{in} \enspace \Omega, \\
		u_0 = f & \text{on} \enspace \partial \Omega,
	\end{cases}
\end{equation} 
and the corresponding DtN operator defined as $\Lambda_0(f) = \p_\nu u_0$.

The main contributions are the resolution of the previously stated problem and are summarized as follows. 
\begin{theorem} Consider the elliptic boundary value problem \eqref{bvp} on a bounded smooth domain $\Omega \subset\mathbb{R}^2$ with a smooth subdomain (inclusion) $D \subset \Omega$ such that $\partial D \cap \partial \Omega = \varnothing$.  Then, the following statements hold. 
\begin{enumerate} 
    \item  The DtN map $\Lambda$ uniquely determines the coefficients $\mu, \mu_s$, and $\ell_s$.

    \item The difference of the DtN maps, $\Lambda - \Lambda_0$, uniquely reconstructs the inclusion $D$.
\end{enumerate}
\end{theorem}
The first part of the above theorem shows that the measured shear force on the known boundary $\partial \Omega$ uniquely determines the system parameters; a more detailed version is presented in Theorem \ref{thm:mechprecise}. The second part shows that a specific comparison of shear forces on $\partial \Omega$ uniquely recovers the location and shape of the inclusion with a more precise version given in Theorem \ref{main-thm}. To this end, we will derive a qualitative sampling algorithm to determine $D$ without the knowledge of the system parameters $\mu, \mu_s$, and $\ell_s$. 

\begin{remark}
 Our assumptions on the smoothness of the boundary of $D$ can likely be weakened. Our arguments require only the validity of Green's identities and the well-definedness of the Dirichlet-to-Neumann maps, so substantially lower boundary regularity should suffice. We adopt the present smoothness assumptions only to simplify the exposition.  
\end{remark}

\subsection{Qualitative, shape reconstruction method} To solve the inverse shape problem for extended inclusions, we consider a qualitative reconstruction method. This family of methods is advantageous in the sense that they require little a priori knowledge of the unknown regions. In contrast, iterative methods require ``good'' initial estimates for the unknown region and/or parameters to insure that the iterative process will converge to the unique solution. The specific qualitative method that we will derive is a variant of the so called {\it regularized factorization method}. This method is based on the analysis of the linear sampling method {\color{black}(see e.g. \cite{cakoni-colton-2003, cakoni-colton-haddar, colton-haddar-monk, colton-haddar-piana})} using measurements involving the Dirichlet-to-Neumann map as in \cite{cheney2001thelinear,kirsch2002music}. Factorization-type methods have been extensively studied for inverse shape problems arising in electrostatics and inverse scattering; see, for example, \cite{harris0,harris1,kirsch2005factorization,kirsch2007factorization} and the references therein. {\color{black} A central feature of these methods is the characterization of the unknown inclusion through singular solutions, each associated with a single sampling point.} In this work, we adapt this philosophy to a fundamentally different setting: an inverse problem in continuum mechanics governed by elastic equilibrium equations coupled to a higher-order strain-gradient interface model.

{\color{black} However, the present problem differs from the classical setting in two fundamental respects. First, the data operator possesses a nontrivial null space, inherited from the fourth-order boundary operator $\mathscr{B}$ defined in \eqref{b-op}, whose kernel consists of the constant functions. Second, the strain-gradient surface energy introduces higher-order boundary terms, requiring substantially greater elliptic regularity than in the classical theory. These features fundamentally alter the structure of the factorization method. In particular, the one-dimensional null space prevents the standard characterization of the inclusion using singular solutions, each centered at a single sampling point. Instead, we obtain a characterization in terms of differences of singular solutions, leading naturally to a two-point sampling procedure for reconstructing the inclusion.}


The regularized factorization method relies on connecting the inclusion to the range of an operator derived from measured data. This will be accomplished by characterizing the unknown inclusion $D$ by the spectral decomposition of the compact, data operator $\Lambda - \Lambda_0$. This makes the numerical implementation of these methods computationally simple since one only needs to compute the singular value decomposition of the discretized operator. From a computational perspective, this approach is more cost effective compared to deriving an iterative method, which may require solving one or more adjoint problems at each iteration.

\subsection{Outline}

We first establish the existence, uniqueness and regularity of solutions to the direct problem \eqref{bvp} in Section \ref{sec:direct}.  Then, we provide the key factorization results of $\Lambda-\Lambda_0$ in Section \ref{sec:factorization}. This is required to prove the main results. The parameter determination for a given domain is established in Section \ref{inverse-parameter-problem}. The domain determination result given fixed parameters is then given in Section \ref{inverse-shape-problem}.  Some supporting numerics are given in Section \ref{numerics}, and concluding remarks are given in Section \ref{sec:conclusion}. Some of the more technical aspects of Section \ref{sec:factorization} are given in Appendix \ref{app:a}, and some technical analysis related to Section \ref{numerics} are given in Appendix \ref{appendix-kernel}. 

\subsection*{Acknowledgments} G.G. was supported by NSF RTG DMS-2135998.  J.L.M. acknowledges support from NSF grant DMS-2307384. C. R. acknowledges support provided by NSF DMS-2307562.

\section{\bf The Direct Problem}
\label{sec:direct}

In this section we study the well-posedness of the forward problem \eqref{bvp}. Given a bounded, smooth domain $\Omega \subset \mathbb{R}^2$ and an smooth interior region $D$, we assume that dist$(\partial \Omega , \overline{D}) > 0$ and denote any coordinate pair as a single variable, e.g. $x$. Due to the fourth-order boundary operator \eqref{b-op} on $\partial D$, we consider finding the solution $u \in \mathcal{H}$ to \eqref{bvp} for a given $f \in H^{1/2}(\partial \Omega)$. Here the solution space is the space defined as
$$\mathcal{H} = \{u \in H^{1}(\Omega) \mid u \big \rvert_{\partial D} \in H^{2}(\partial D)\},$$
equipped with the seminorm
$$\norm{u}^{2}_{\mathcal{H}} :=  \norm{\nabla u}^{2}_{L^{2}(\Omega)} + \|\partial^{2}_{s} u \|^{2}_{L^{2}(\partial D)} +\norm{\partial_{s} u}^{2}_{L^{2}(\partial D)} .$$
We note that if $v \in \cl H \cap H^1_0(\Omega)$, then by the Poincar\'e inequality, there exists a constant $C>0$ such that 
\begin{align}
 \| v \|_{H^1(\Omega)}^2 + \|\partial^{2}_{s} u \|^{2}_{L^{2}(\partial D)} +\norm{\partial_{s} u}^{2}_{L^{2}(\partial D)} \leq C \| v \|_{\cl H}^2, 
\end{align}
i.e. $\| \cdot \|_{\cl H}$ is a norm on $\cl H \cap H^1_0(\Omega)$, turning $\cl H \cap H^1_0(\Omega)$ into a Hilbert space. 
To show well-posedness of \eqref{bvp}, we rederive its variational formulation in detail here.
{\color{black}\begin{lemma}\label{vf-bvp}
The variational formulation for \eqref{bvp} is: find $u \in \mathcal{H}$ such that 
\begin{equation}\label{var-form}
     \int_{\Omega \setminus D} \nabla u \cdot \nabla {\varphi} \, \text{d}x + \int_{D} \mu \nabla u \cdot \nabla {\varphi} \, \text{d}x + \int_{\partial D} \mu_s \partial_{s} u \partial_{s} {\varphi} \, \text{d}s + \int_{\partial D}  \mu_s \ell_s^2 \partial^{2}_{s} u  \partial^{2}_{s}  {\varphi} \, \text{d}s = 0
\end{equation}
for all $\varphi \in \mathcal{H} \cap H^{1}_{0}(\Omega)$.

\end{lemma}
\begin{proof}

Note that
$$\int_{\Omega \setminus D} - \Delta u \, {\varphi} \, \text{d}x + \int_{D} - \mu \Delta u \, {\varphi} \, \text{d}x = 0$$
for any $\varphi \in C^\infty_c(\Omega)$, where $u$ satisfies \eqref{bvp}. We will use Green's first identity for integrals defined on $\Omega \setminus D$ and $D$. In $\Omega \setminus {D}$,
$$-\int_{\Omega \setminus {D}} \Delta u \, {\varphi} \, \text{d}x = \int_{\Omega \setminus {D}} \nabla u \cdot \nabla {\varphi} \, \text{d}x - \int_{\partial \Omega} {\varphi} \partial_{\nu} u \, \text{d}s + \int_{\partial D} {\varphi} \partial_{\nu} u^{+} \, \text{d}s.$$
Since $\varphi \in \mathcal{H} \cap H^{1}_{0}(\Omega)$, then we get that
$$-\int_{\Omega \setminus {D}} \Delta u\, {\varphi} \, \text{d}x = \int_{\Omega \setminus {D}} \nabla u \cdot \nabla {\varphi} \, \text{d}x + \int_{\partial D} {\varphi} \partial_{\nu} u^{+} \, \text{d}s.$$
Similarly, in $D$ we have that
$$-\int_{D} \mu \Delta u {\varphi} \, \text{d}x = \int_{D} \mu \nabla u \cdot \nabla {\varphi} \, \text{d}x - \int_{\partial D} {\varphi} \mu \partial_{\nu}u^{{-}} \, \text{d}s.$$
Recall that  $-\Delta u = 0$ in $\Omega \setminus D$ and $-\mu \Delta u = 0$ in $D$. Thus, in the entire region $\Omega$ we have that
\begin{align*}
    0 &= \int_{\Omega \setminus D} \nabla u \cdot \nabla {\varphi} \, \text{d}x + \int_{D} \mu \nabla u \cdot \nabla {\varphi} \, \text{d}x + \int_{\partial D} (\partial_{\nu} u^{+} - \mu \partial_{\nu} u^{-} ){\varphi} \, \text{d}s \\
    &= \int_{\Omega \setminus D} \nabla u \cdot \nabla {\varphi} \, \text{d}x + \int_{D} \mu \nabla u \cdot \nabla {\varphi} \, \text{d}x + \int_{\partial D}  [\![\partial_\nu u ]\!] {\varphi} \, \text{d}s.
\end{align*}

\noindent Using the boundary condition on $\partial D$, we get that
\begin{align*}
    0 &= \int_{\Omega \setminus D} \nabla u \cdot \nabla {\varphi} \, \text{d}x + \int_{D} \mu \nabla u \cdot \nabla {\varphi} \, \text{d}x  + \int_{\partial D} \mathscr{B}(u) {\varphi} \, \text{d}s \\
    &= \int_{\Omega \setminus D} \nabla u \cdot \nabla {\varphi} \, \text{d}x + \int_{D} \mu \nabla u \cdot \nabla {\varphi} \, \text{d}x + \int_{\partial D}  \mu_s \ell_s^2 \partial^{2}_{s} u  \partial^{2}_{s}  {\varphi} \, \text{d}s + \int_{\partial D} \mu_s \partial_{s} u \partial_{s} {\varphi} \, \text{d}s ,
\end{align*} 
where we used the fact that $\partial D$ is a closed curve in using integration by parts. Thus, the variational formulation for \eqref{bvp} is: find $u \in \mathcal{H} \cap H^{1}_{0}(\Omega)$ such that 
     \[\int_{\Omega \setminus D} \nabla u \cdot \nabla {\varphi} \, \text{d}x + \int_{D} \mu \nabla u \cdot \nabla {\varphi} \, \text{d}x + \int_{\partial D} \mu_s \partial_{s} u \partial_{s} {\varphi} \, \text{d}s + \int_{\partial D}  \mu_s \ell_s^2 \partial^{2}_{s} u  \partial^{2}_{s}  {\varphi} \, \text{d}s = 0 \]
     for all $\varphi \in \mathcal{H} \cap H^{1}_{0}(\Omega)$.
\end{proof}}
We now consider the extension of $f$, $u_0 \in H^{1}(\Omega)$, satisfying the Dirichlet problem \eqref{lifting}. Thus, $u_0$ exists, is unique, and continuously depends on $f$, i.e. $\norm{u_0}_{H^{1}(\Omega)} \leq C \norm{f}_{H^{1/2}(\partial \Omega)}$.  Moreover, by elliptic regularity, we can conclude  
\begin{equation}\label{u_0-H-norm}
    \norm{u_0}_{\mathcal{H}} \leq C \norm{f}_{H^{1/2}(\partial \Omega)},
\end{equation}
where $C>0$ is an absolute constant. 

We define $v \in \mathcal{H} \cap H^{1}_{0}(\Omega)$ {\color{black} where} $u = v + u_0$. Then the variational formulation \eqref{var-form} becomes: 
\begin{multline}\label{var-form-v}
     \int_{\Omega \setminus D} \nabla v \cdot \nabla {\varphi} \, \text{d}x + \int_{D} \mu \nabla v \cdot \nabla {\varphi} \, \text{d}x + \int_{\partial D} \mu_s \partial_{s} v \partial_{s} {\varphi} \, \text{d}s + \int_{\partial D}  \mu_s \ell_s^2 \partial^{2}_{s} v  \partial^{2}_{s}  {\varphi} \, \text{d}s \\ = - \int_{\Omega \setminus D} \nabla u_0 \cdot \nabla {\varphi} \, \text{d}x - \int_{D} \mu \nabla u_0 \cdot \nabla {\varphi} \, \text{d}x  - \int_{\partial D} \mu_s \partial_{s} u_0 \partial_{s} {\varphi} \, \text{d}s - \int_{\partial D}  \mu_s \ell_s^2 \partial^{2}_{s} u_0  \partial^{2}_{s}  {\varphi} \, \text{d}s
\end{multline}
for all $\varphi \in \mathcal{H} \cap H^{1}_{0}(\Omega)$. We define $A(\cdot ,\cdot) : \big( \mathcal{H} \cap H^{1}_{0}(\Omega) \big ) \times \big ( \mathcal{H} \cap H^{1}_{0}(\Omega) \big ) \rightarrow \mathbb{R}$ as 
\begin{equation}
\label{eqn:Adef}
A(v , \varphi) =  \int_{\Omega \setminus D} \nabla v \cdot \nabla {\varphi} \, \text{d}x + \int_{D} \mu \nabla v \cdot \nabla {\varphi} \, \text{d}x + \int_{\partial D} \mu_s \partial_{s} v \partial_{s} {\varphi} \, \text{d}s + \int_{\partial D}  \mu_s \ell_s^2 \partial^{2}_{s} v  \partial^{2}_{s}  {\varphi} \, \text{d}s.
\end{equation}
It is clear that $A(\cdot , \cdot)$ is a sesquilinear form. We also define $L: \mathcal{H} \cap H^{1}_{0}(\Omega) \rightarrow \mathbb{R}$ as 
\begin{equation}
\label{eqn:Ldef}
L(\varphi) = - \int_{\Omega \setminus D} \nabla u_0 \cdot \nabla {\varphi} \, \text{d}x - \int_{D} \mu \nabla u_0 \cdot \nabla {\varphi} \, \text{d}x  - \int_{\partial D} \mu_s \partial_{s} u_0 \partial_{s} {\varphi} \, \text{d}s - \int_{\partial D}  \mu_s \ell_s^2 \partial^{2}_{s} u_0  \partial^{2}_{s}  {\varphi} \, \text{d}s,
\end{equation}
where it is clear that $L(\cdot )$ is linear in $\varphi$. Thus, the variational formulation with respect to $v \in \mathcal{H} \cap H^{1}_{0}(\Omega)$ as written in \eqref{var-form-v} can be expressed as $$A(v , \varphi) = L(\varphi).$$

\noindent The following lemmas are necessary to use the Lax-Milgram Lemma in order to show the well-posedness of \eqref{var-form-v}.

\begin{lemma}\label{A-bdd-coercive}
    The sesquilinear form $A(\cdot , \cdot )$ defined in \eqref{eqn:Adef} is bounded and coercive.
\end{lemma}
\begin{proof}
For boundedness, consider 
\begin{align*}
    \rvert A (v , \varphi ) \rvert &\leq \text{max}(1,\mu) \|\nabla v\|_{L^{2}(\Omega)} \norm{\nabla \varphi}_{L^{2}(\Omega)} + \mu_s \norm{\partial_s v}_{L^{2}(\partial D)} \norm{\partial_s \varphi}_{L^{2}(\partial D)} \\
    & \hspace{2cm} + \mu_s \ell_s^2 \|\partial^2_s v\|_{L^{2}(\partial D)} \|\partial^2_s \varphi\|_{L^{2}(\partial D)} \\
    &\leq \big ( \text{max}(1,\mu)  \norm{\nabla \varphi}_{L^{2}(\Omega)} + \mu_s \norm{\partial_s \varphi}_{L^{2}(\partial D)} + \mu_s \ell_s^2 \|\partial^2_s \varphi\|_{L^{2}(\partial D)} \big ) \norm{v}_{\mathcal{H}} \\
    &\leq \big( \text{max}(1,\mu) + \mu_s + \mu_s \ell_s^2 \big) \norm{v}_{\mathcal{H}} \norm{\varphi}_{\mathcal{H}}.
\end{align*}
Thus, $A(\cdot , \cdot )$ is bounded. For coercivity, consider
\begin{align*}
\rvert A(v,v) \rvert &\geq \text{min}(1,\mu) \norm{\nabla v}^{2}_{L^{2}(\Omega)} + \mu_s \|\partial_s v\|^{2}_{L^{2}(\partial D)} + \mu_s \ell_s^2 \|\partial^2_s v\|^{2}_{L^{2}(\partial D)} \\
    &\geq \text{min} \big(( \text{min}(1,\mu), \mu_s , \mu_s \ell_s^2 \big) \norm{v}^{2}_{\mathcal{H}}.
\end{align*}
Therefore, $A(\cdot , \cdot)$ is coercive. 
\end{proof}

\begin{lemma}\label{L-bdd}
    The linear map $L$ defined in \eqref{eqn:Ldef} is bounded, i.e. there exists a constant $C>0$ such that $\rvert L (\varphi ) \rvert \leq C \norm{f}_{H^{1/2}(\partial \Omega)} \norm{\varphi}_{\mathcal{H}}$ {\color{black}for all $\varphi \in \mathcal{H}$}.
\end{lemma}

\begin{proof}
    Consider
    \begin{align*}
        \rvert L (\varphi ) \rvert &\leq \text{max}(1,\mu) \norm{\nabla u_0}_{L^{2}(\Omega)} \norm{\nabla \varphi}_{L^{2}(\Omega)} \\
        & \hspace{1cm} + \mu_s \norm{\partial_s u_0}_{L^{2}(\partial D)} \norm{\partial_s \varphi}_{L^{2}(\partial D)} + \mu_s \ell_s^2 \|\partial^2_s u_0\|_{L^{2}(\partial D)} \|\partial^2_s \varphi\|_{L^{2}(\partial D)} \\
        &\leq \big ( \text{max}(1,\mu) \norm{\nabla u_0}_{L^{2}(\Omega)} + \mu_s \|\partial_s u_0\|_{L^{2}(\partial D)} + \mu_s \ell_s^2 \|\partial^2_s u_0\|_{L^{2}(\partial D)} \big ) \norm{\varphi}_{\mathcal{H}} \\
        &\leq \big( \text{max}(1,\mu) + \mu_s + \mu_s \ell_s^2 \big) \norm{u_0}_{\mathcal{H}} \norm{\varphi}_{\mathcal{H}}.
    \end{align*}
    By \ref{u_0-H-norm}, we have that
    $$ \norm{u_0}_{\mathcal{H}} \leq C \norm{f}_{H^{1/2}(\partial \Omega)}.$$
Therefore, 
$$|L (\varphi)| \leq C \norm{f}_{H^{1/2}(\partial \Omega)} \norm{\varphi}_{\mathcal{H}}$$
where $C>0$ is an absolute constant. 
\end{proof}

Since $A(\cdot , \cdot)$ is a sesquilinear form that is bounded and coercive and $L(\cdot)$ is linear and bounded, it follows that by the Lax-Milgram Lemma, there exists a unique $v \in \mathcal{H} \cap H^{1}_{0}(\Omega)$ that satisfies \eqref{var-form-v} for all $\varphi \in \mathcal{H} \cap H^{1}_{0}(\Omega)$ and $\norm{v}_{\mathcal{H}} \leq C \norm{f}_{H^{1/2}(\partial \Omega)}$ for some absolute constant $C > 0$. Thus, $u = v + u_0$ satisfies the weak form of \eqref{bvp}. Furthermore, the previous implies that $u$ depends continuously on the data.
\begin{lemma}\label{cont-unique}
    The solution $u$ of \eqref{bvp} continuously depends on the Dirichlet data $f \in H^{1/2}(\partial \Omega)$ and is unique.
\end{lemma}
\begin{proof}
Continuity follows from the fact that $\norm{v}_{\mathcal{H}} \leq C \norm{f}_{H^{1/2}(\partial \Omega)}$ and \eqref{u_0-H-norm}. Thus, there exists a positive constant $C$ such that $\norm{u}_{\mathcal{H}} \leq C \norm{f}_{H^{1/2}(\partial \Omega)}$. To show uniqueness, suppose $u_1$ and $u_2$ are solutions to \eqref{bvp}. Let $w = u_1 - u_2 \in \mathcal{H} \cap H^{1}_{0}(\Omega)$. Then $w$ also satisfies $A(w,\varphi) = 0$ for all $\varphi \in \mathcal{H} \cap H^{1}_{0}(\Omega)$. Thus, $A(w,w) = 0$ which by coercivity of $A$ implies that $\| w \|_{\cl H} =0$. The Poincar\'e inequality and the fact that $w \in H^1_0(\Omega)$ then imply that $w = 0$, proving uniqueness.  
\end{proof}

\begin{theorem}\label{thm:bvp-well-posed}
    The boundary value problem \eqref{bvp} is well-posed.
\end{theorem}
\begin{proof}
    We have that \eqref{var-form} is the weak formulation of \eqref{bvp}. Lemma \ref{A-bdd-coercive} and Lemma \ref{L-bdd} imply that there exists $u$ that satisfies \eqref{var-form}. Lastly, Lemma \ref{cont-unique} shows that the solution $u$ continuously depends on the data and is unique. Therefore, \eqref{bvp} is well-posed.
\end{proof}

\section{\bf Factorization of the Data Operator}
\label{sec:factorization}

In this section, we derive a symmetric factorization of the data operator, $\Lambda - \Lambda_0$. The decomposition will set the foundation for the main results of Section \ref{inverse-parameter-problem} on parameter recovery and Section \ref{inverse-shape-problem} on shape reconstruction. In our physical system \eqref{bvp}, we assume that a displacement $f$ has been applied to the boundary $\partial \Omega$. The measured data is given by the induced, shear force $\partial_{\nu}u$.  We similarly consider the unperturbed system \eqref{lifting}, and the resulting shear force $\partial_{\nu}u_0$.
By the well-posedness and linearity of the physical systems \eqref{bvp} and \eqref{lifting}, we have that the prescribed boundary- to body-displacement mappings
$$f \mapsto u \quad \text{and} \quad f \mapsto u_0$$
are bounded linear operators from $H^{1/2}(\partial \Omega)$ to $\mathcal{H}$. Furthermore, by additionally appealing to the trace theorem, it holds that the DtN mappings
$$\Lambda \enspace \text{and} \enspace \Lambda_0: H^{1/2}(\partial \Omega) \longrightarrow H^{-1/2}(\partial \Omega)$$
where 
$$\Lambda f = \partial_{\nu} u \big\rvert_{\partial \Omega} \quad \text{and} \quad \Lambda_0 f = \partial_{\nu} u_0 \big\rvert_{\partial \Omega}$$
are also bounded linear operators. The analysis is based on the factorization of the data operator $\Lambda - \Lambda_0$, where the imaging functional to be derived will use its spectral decomposition (or singular value decomposition). The theory used here was developed in \cite{harris1} and will ultimately allow us to derive an imaging functional for extended regions. We begin by considering the eigensystem of the fourth-order boundary operator $\mathscr{B}$ as defined in \eqref{b-op}. As shown in \cite{lions}, for smooth $\partial D$, there exists
\begin{equation}\label{eig-sys}
    \{\lambda_n , \phi_n\}_{n\in \mathbb{N} \cup \{0\}} \quad \text{such that} \quad -\partial^{2}_{s}\phi_n = \lambda^2_n \phi_n\enspace \text{for all} \enspace n \in \mathbb{N} \cup \{0\}
\end{equation}
where $\lambda_n \geq 0$ for all $n$ and $\phi_n$ are $C^{\infty}$ functions that form an orthonormal basis of $L^{2}(\partial D)$. Here, $\lambda_0 = 0$ is a simple eigenvalue with associated eigenspace spanned by $\phi_0 := |\p D|^{-1/2}$, and $\lambda_n > 1$ for all $n \in \mathbb N$. Moreover, for any $p \geq 0$, 
$$q \in H^{p}(\partial D) \quad \text{if and only if} \quad \| q \|^{2}_{H^p (\partial D)} := |q_0|^2 + \sum_{n \geq 1} \rvert \lambda^{p}_n q_n\rvert^2 < \infty,$$
where $q_n = (q,\phi_n)_{L^{2}(\partial D)}$ for all $n \in \mathbb{N}$. For all $p \geq 0$, we identify the dual space of $H^p(\partial D)$ with $H^{-p}(\partial D)$. That is, for each $p \geq 0$, there is a Gelfand triple
\[H^p(\partial D) \subset L^{2}(\partial D) \subset H^{-p}(\partial D)\]
with dense embedding, with $L^2(\partial D)$ being the pivot space. We refer to Chapter 3 and Appendix A of \cite{mclean2000strongly} and/or \cite{taylor1996partial1}, Chapter 4 for precise discussions. The following regularity result will be useful in our factorization analysis. 

\begin{lemma}\label{add-reg}
    There exists a constant $C > 0$ with the following property. {\color{black}For any $f \in H^{1/2}(\partial \Omega)$}, if $u \in \mathcal{H}$ is the unique solution to \eqref{bvp}, then $ u \rvert_{\partial D} \in H^{7/2}(\partial D)$ and $\| u \|_{H^{7/2}(\p D)} \leq C \| f \|_{H^{1/2}(\p \Omega)}$.
\end{lemma}

\begin{proof}
By the wellposedness of \eqref{bvp}, we have $[\![\partial_\nu u ]\!] \big \rvert_{\partial D} \in H^{-1/2}(\partial D)$ with 
\begin{align}
|g_0|^2 + \sum_{n \geq 1}|g_n|^2 \lambda^{-1}_n \leq C \| f \|_{H^{1/2}(\p \Omega)}^2 \enspace \mbox{where} \enspace [\![\partial_\nu u ]\!] \big\rvert_{\partial D}= \sum g_n \phi_n.\label{eq:normal}
\end{align}
By \eqref{bvp}, $\mathscr{B}(u) \in H^{-1/2}(\partial D)$. We write $$u - \fint_{\partial D} u \, \text{d}s = \sum_{n \geq 1} u_n \phi_n \quad \text{on} \enspace \partial D.$$
Therefore, $$\mathscr{B}(u - \fint_{\partial D} u \, \text{d}s) = \sum_{n \geq 1} u_n (\mu_s \lambda^2_n +  \mu_s \ell_s^2 \lambda^4_n) \phi_n \quad \text{on} \enspace \partial D.$$
By the boundary condition on $\partial D$, for all $n \geq 1$, 
$$u_n = \frac{g_n}{\mu_s \ell_s^2 \lambda^4_n + \mu_s \lambda^2_n}. $$
Thus, 
\begin{align}
    \sum_{n \geq 1} \big|\lambda^{7/2}_n u_n\big|^2 = \sum_{n \geq 1} \lambda^{7}_n \big| u_n \big|^2 = \sum_{n \geq 1} \frac{\lambda^{7}_n \big| g_n \big|^2}{(\mu_s \ell_s^2 \lambda^4_n + \mu_s \lambda^2_n)^2} = \sum_{n \geq 1} \frac{\lambda^{8}_n }{(\mu_s \ell_s^2 \lambda^4_n + \mu_s \lambda^2_n)^2} \big| g_n \big|^2 \lambda^{-1}_n. \label{eq:high}
\end{align}
The identifies \eqref{eq:high}, \eqref{eq:normal}, and the fact that 
$$\Bigl |\int_{\p D} \phi_0 u \, \text{d}s\Bigr | \leq \| u \|_{L^2(\p D)} \leq \| u \|_{\cl H} \leq C \| f \|_{H^{1/2}(\p \Omega)}$$ complete the proof.   
\end{proof}

This result will facilitate a symmetric factorization of the data operator $(\Lambda - \Lambda_0)$. Influenced by the difference of these DtN operators, we note that $u - u_0 \in H^{1}_{0} (\Omega)$ solves 
$$    \begin{cases} 
      - \Delta (u-u_0) = 0 & \text{in} \enspace \Omega \setminus D, \\
      -\mu \Delta (u-u_0) = 0 & \text{in} \enspace D , \\
      [\![\partial_\nu (u-u_0 )]\!] = \mathscr{B}(u) - (1-\mu) \partial_{\nu} u_0 & \text{on} \enspace \partial D.
   \end{cases}
$$
Inspired by this, we define $w \in H^{1}_{0} (\Omega)$ to be the unique solution of 
\begin{equation} \label{aux-prob}
\begin{cases} 
      - \Delta w = 0 & \text{in} \enspace \Omega \setminus D , \\
      -\mu \Delta w = 0 & \text{in} \enspace D, \\
      [\![\partial_\nu w]\!] = \mathscr{B}(h) - (1-\mu) \Lambda_D h + (1-\mu) \partial_{\nu} w^{-}& \text{on} \enspace \partial D
   \end{cases}
\end{equation}
for a given $h \in H^{7/2}(\partial D)$ where $\Lambda_D $ is the interior DtN map defined on $\partial D$. One can show that \eqref{aux-prob} is well-posed by appealing to a variational formulation and using Green's first identity. From the well-posedness of this auxiliary problem and Lemma \ref{add-reg}, we have the following:
\begin{align} \label{eqn:wbd}
  \begin{array}{ll} 
  \| h\|_{H^{7/2}(\partial D)}^2 = \| \mathscr{B}(h) \|_{H^{-1/2}(\partial D)}^2 &= \| [\![\partial_\nu w]\!] + (1- \mu) \Lambda_D h - (1-\mu) \partial_{\nu} w^{-} \|_{H^{-1/2}(\partial D)}^2\\
    &\leq C \big ( \| w \|^{2}_{H^{1}_{0}(\Omega)} + \| \Lambda_D h\|^{2}_{H^{-1/2}(\partial D)} \big )  \\
    &\leq C \big ( \| w \|^{2}_{H^{1}_{0}(\Omega)} + \| h\|^{2}_{H^{1/2}(\partial D)} \big )  \\
    &\leq C \big ( \| w\|_{H^{1}_{0}(\Omega)}^2 + \| h\|^{2}_{H^{2}(\partial D)} \big ) ,
\end{array}
\end{align}
where the second to last inequality follows from the boundedness of the DtN map on $\p D$. Therefore, we can define the bounded linear operator
$$G:H^{7/2}(\partial D) \rightarrow H^{-1/2} (\partial \Omega) \quad \text{given by} \quad Gh = \partial_{\nu} w \big \rvert_{\partial \Omega}$$
where $w$ is the unique solution to \eqref{aux-prob}. Notice that since $\eqref{aux-prob}$ is well-posed, then $\partial_{\nu} w \rvert_{\partial \Omega} = (\Lambda - \Lambda_0) f$ provided that $h = u \rvert_{\partial D}$. Thus, we define the solution operator for \eqref{bvp} as
\begin{equation}\label{s-operator}
    S:H^{1/2}(\partial \Omega) \rightarrow H^{7/2} (\partial D) \quad 
\text{given by} \quad Sf = u \big \rvert_{\partial D}.
\end{equation}
From this, one sees that $(\Lambda - \Lambda_0) f = GSf $ for any $f \in H^{1/2} (\partial \Omega)$. From this preliminary factorization, we will further decompose the operator $G$ to provide a symmetric factorization for $(\Lambda - \Lambda_0)$. However, we first provide some useful properties of the solution operator $S$.
\begin{lemma}\label{s-compact-injective}
    The operator $S$ as defined in \eqref{s-operator} is {\color{black} bounded} and injective.
\end{lemma}
\begin{proof}
{\color{black} The boundedness of $S$ is a direct consequence of Lemma \ref{add-reg}.} To prove injectivity, we let $f \in {\rm Null}(S)$, which implies that $u = 0$ in $D$. By our boundary condition, we have that $[\![\partial_\nu u ]\!] = \mathscr{B}(u) = 0$ on $\partial D$. Thus, $\partial_{\nu}u^{+} \big \rvert_{\partial D} = 0$. By Holmgren’s Theorem, we have that $u = 0$ in $\Omega$. Then by the trace theorem, it holds that $f = 0$ on $\partial \Omega$, proving that $S$ is injective.    
\end{proof}
Our proposed strategy to derive a symmetric factorization for $(\Lambda - \Lambda_0)$ is to decompose the operator $G$ involving the adjoint operator of $S$. To proceed, we define the sesqulinear dual-product on a closed curve $\Gamma$ as
\begin{equation}
    \langle \varphi , \psi \rangle_{\Gamma} = \int_{\Gamma} \varphi \psi \, \text{d}s \quad \text{for all} \quad \varphi \in H^{p}(\Gamma) \enspace \text{and} \enspace \psi \in H^{-p}(\Gamma) 
\end{equation}
between the Hilbert space $H^{p}(\Gamma)$ and its dual space $H^{-p}(\Gamma)$ for $p>0$ where $L^{2}(\Gamma)$ is the Hilbert pivot space. We are primarily interested in the cases where $\Gamma = \partial \Omega$ with $p=1/2$ and $\Gamma = \partial D$ with $p=7/2$, denoted as $\langle\cdot , \cdot \rangle_{\partial \Omega}$ and $\langle\cdot , \cdot \rangle_{\partial D}$, respectively. These dual-products will be instrumental in the rest of our factorization analysis. In particular, they are needed in defining the adjoint of $S$, denoted as $S^{*}: H^{-7/2}(\partial D) \rightarrow H^{-1/2}(\partial \Omega)$. However, in order to do so, we must discuss the elliptic equation
 \begin{equation} \label{adj-s}
\begin{cases} 
      - \Delta v = 0 & \text{in} \enspace \Omega \setminus D, \\
      -\mu \Delta v = 0 & \text{in} \enspace D, \\
      [\![\partial_\nu v]\!] = \mathscr{B}(v) + g & \text{on} \enspace \partial D \\
      v = 0 & \text{on} \enspace \partial \Omega.
   \end{cases}
    \end{equation}
The solvability of \eqref{adj-s} for $ g \in H^{-7/2}(\partial D)$ will need to be established. For $g \in H^{-2}(\partial D)$, the notion of solution is much simpler. In particular, as similarly shown for \eqref{bvp}, the solvability of \eqref{adj-s} is understood in the variational sense as described below.  
\begin{definition} \label{def-adj_H-2}
    If $g \in  H^{-2}(\partial D)$, then {\color{black} we say} $v\in \mathcal{H} \cap H^{1}_{0}(\Omega)$ satisfies \eqref{adj-s} if
\begin{equation}\label{var-form-aux}
     \int_{\Omega \setminus D} \nabla v \cdot \nabla {\varphi} \, \text{d}x + \int_{D} \mu \nabla v \cdot \nabla {\varphi} \, \text{d}x + \int_{\partial D} \mu_s \partial_{s} v \partial_{s} {\varphi} \, \text{d}s + \int_{\partial D}  \mu_s \ell_s^2 \partial^{2}_{s} v \partial^{2}_{s}  {\varphi} \, \text{d}s + \int_{\partial D}  g {\varphi} \, \text{d}s= 0
\end{equation}
for all $\varphi \in \mathcal{H} \cap H^{1}_{0}(\Omega)$.
\end{definition}

In particular, for $g$ given by a constant, there exists a unique solution to \eqref{adj-s} in the form of Definition \ref{def-adj_H-2}.  However, if $g \in H^{-7/2}(\p D) \backslash H^{-2}(\p D)$, then variational methods do not apply directly to \eqref{adj-s}. Thus, we will define in what sense an element $v \in H^1_0(\Omega)$ satisfies \eqref{adj-s} in the case that $g \in \mathrm{span}\{1\}^\perp \subset H^{-7/2}(\partial D)$. Note that ${\rm Null}(\mathscr{B}) = {\rm span} \{1\}$ on $\partial D$. Thus, in what follows, we define the pseudoinverse operator 
\begin{equation}\label{eqn:b-sharp}
    \mathscr{B}^{\#} : {\rm span} \{1\}^{\perp} \subset H^{q}(\partial D) \rightarrow {\rm span} \{1\}^{\perp} \subset H^{q+4}(\partial D)
\end{equation}
for any $q \in (-\infty , \infty)$, where for $\phi_n$ as in \eqref{eig-sys}, we have the convention
\begin{equation}
    \label{eqn:Bsharpdef}
   \mathscr{B}^{\#} \phi_0 = 0, \ \  \mathscr{B}^{\#} \phi_n = (\mu_s \lambda_n + \mu_s \ell_s \lambda_n^2)^{-1} \phi_n \ \text{for} \ n \geq 1 .
\end{equation}

We also define $\Lambda_{\Omega \setminus D}$ as the exterior DtN map on $\partial D$, i.e. $\Lambda_{\Omega\setminus D} f = \partial_{\nu} \omega^{+} \rvert_{\partial D}$, where $\omega$ satisfies  
      $$- \Delta \omega = 0  \enspace \text{in} \enspace \Omega \setminus D \quad  \text{with} \quad \omega = 0  \enspace \text{on} \enspace \partial \Omega , \quad \text{and} \quad \omega =  f  \enspace \text{on} \enspace \partial D.$$ 

\begin{definition} \label{def-adj}
    If $g \in {\rm span}\{1\}^{\perp} \subset H^{-7/2}(\partial D)$, then {\color{black} we say} $v\in H^{1}_{0}(\Omega)$ satisfies \eqref{adj-s} if
    \begin{enumerate}
        \item $f := v \rvert_{\partial D}$ satisfies
    \begin{align}
        \int_{\partial D}\Lambda_{\Omega \setminus D} f \, \text{d}s = 0, \enspace \text{and} \label{exterior-dtn-mean}\\
        {\color{black}-}\mathscr{B}^{\#} g = f - \fint_{\partial D} f \, \text{d}s - \mathscr{B}^{\#} (\Lambda_{\Omega \setminus D} - \mu \Lambda_D)f \label{pseudo-g}
    \end{align}
    and 
        \item  $- \Delta v = 0  \enspace \text{in} \enspace \Omega \setminus D \quad  \text{and} \quad -\mu\Delta v= 0  \enspace \text{in} \enspace D , \quad \text{with} \quad v =  0  \enspace \text{on} \enspace \partial \Omega .$
    \end{enumerate}
\end{definition}  

We show that Definition 3.1 and Definition 3.2 are equivalent when $g \in {\rm span} \{1\}^{\perp} \subset H^{-2}(\partial D)$. From Definition \ref{def-adj}, we observe that if $g \in {\rm span} \{1\}^{\perp} \subset H^{-2}(\partial D)$, then 
$\mathscr{B}^{\#}g \in H^{2}(\partial D)$ and $\mathscr{B}^{\#}(\Lambda_{\Omega \setminus D} - \mu \Lambda_D)f \in H^{7/2}(\partial D)$. Thus, \eqref{pseudo-g} implies that 
$$f = {\color{black}-}\mathscr{B}^{\#}g + \fint_{\partial D} f \, \text{d}s + \mathscr{B}^{\#}(\Lambda_{\Omega \setminus D} - \mu \Lambda_D)f \in H^{2}(\partial D)$$
and that 
\begin{align*}
    -g = \mathscr{B}(f) - (\Lambda_{\Omega \setminus D} - \mu\Lambda_D)f \enspace \text{on} \enspace \partial D \quad \iff \quad [\![\partial_\nu v ]\!] = \mathscr{B}(v) +g \enspace \text{on} \enspace \partial D.
\end{align*}
 Thus, if $v$ satisfies Definition \ref{def-adj}, then it also satisfies Definition \ref{def-adj_H-2}. Conversely, suppose that $v$ satisfies Definition \ref{def-adj_H-2} with $g \in \text{span}\{1\}^{\perp} \subset H^{-2}(\partial D)$. By integration by parts and the definition of $\mathscr{B}$ as given in \eqref{b-op}, we observe that \eqref{var-form-aux} can be written as 
 \[ - \int_{\partial D}(\Lambda_{\Omega \setminus D} v - \mu \Lambda_{D} v ) \varphi \, \text{d}s + \int_{\partial D}  \big( \mathscr{B}(v) + g \big) \varphi \, \text{d}s = 0 \]
 for all $\varphi \in \mathcal{H} \cap H^{1}_{0}(\Omega)$. Thus, it holds that 
 \[ {\color{black}-}g = \mathscr{B}(v) - (\Lambda_{\Omega \setminus D} - \mu \Lambda_{D}) v = \mathscr{B}(v - \fint_{\partial D} v \, \text{d}s ) - (\Lambda_{\Omega \setminus D} - \mu \Lambda_{D}) v,\]
 where the last equality is given by the fact that $\text{Null}(\mathscr{B})= \text{span}\{1\}$. Given that $g \in \text{span}\{1\}^{\perp} \subset H^{-2}(\partial D)$
 we have that 
 \[ \int_{\partial D} \Lambda_{\Omega \setminus D} v \, \text{d}s = \int_{\partial D} \big( \mathscr{B} (v - \fint_{\partial D} v \, \text{d}s) + \mu \Lambda_D v \big) \, \text{d}s . \]
 Using integration by parts, one can see that $\int_{\partial D} \big( \mathscr{B} (v - \fint_{\partial D} v \, \text{d}s) \, \text{d}s= 0 $. Furthermore, since $g, \Lambda_D v$, and $\mathscr{B}(v - \fint_{\partial D} v \, \text{d}s)$ all have mean $0$, so does $\Lambda_{\Omega \setminus D} v$ and hence \eqref{exterior-dtn-mean} holds. Moreover, considering that  $v\rvert_{\partial D} - \fint_{\partial D} v \, \text{d}s \in \text{span}\{1\}^{\perp}$, it follows that
 \[{\color{black}-}\mathscr{B}^{\#} g = v - \fint_{\partial D} v \, \text{d}s - \mathscr{B}^{\#} (\Lambda_{\Omega \setminus D} - \mu \Lambda_{D}) v,\]
satisfying \eqref{pseudo-g}. Therefore, Definitions \ref{def-adj} and \ref{def-adj_H-2} are equivalent for the case where $g \in \text{span}\{1\}^{\perp} \subset H^{-2}(\partial D)$.

We will address the solvability of \eqref{adj-s} of solutions for $g \in  H^{-2}(\partial D)$ and $g \in {\rm span}\{1\}^{\perp} \subset H^{-7/2} (\partial D)$ in the Lemmas that follow. However, with Definitions \ref{def-adj_H-2} and \ref{def-adj} in place as notions of solutions for \eqref{adj-s}, we will first state our main result regarding $S^*$. 

\begin{theorem}\label{adj-s-def}
The adjoint operator \( S^* : H^{-7/2}(\partial D) \to H^{-1/2}(\partial \Omega) \) is given by
\[
S^*g = \left.\partial_{\nu} v\,\right|_{\partial \Omega},
\]
where \( v \in H^1_0(\Omega) \) is given by the sum \( v = v_0 + v_1 \) such that:
\begin{itemize}
    \item \( v_0 \in H^1_0(\Omega) \) solves \eqref{adj-s} with mean-zero data
    \[
    g_0 := g - \fint_{\partial D} g \, ds \in H^{-7/2}(\partial D),
    \]
    in the sense of Definition~\ref{def-adj}, and
    \item \( v_1 \in H^1_0(\Omega) \) solves \eqref{adj-s} with constant data
    \[
    g_1 := \fint_{\partial D} g \, ds \in H^{-2}(\partial D),
    \]
    in the sense of Definition~\ref{def-adj_H-2}.
\end{itemize}
\end{theorem}

\begin{proof}

We briefly assume the solvability of \eqref{adj-s} to demonstrate the derivation of the adjoint $S^{*}$. From \eqref{adj-s}, we apply a similar technique used to derive \eqref{var-form-aux} and invoke Green’s second identity to obtain
    $$ 0 = \int_{\partial \Omega} v \partial_{\nu} u - f \partial_{\nu} v \, \text{d}s + \int_{\partial D} v (\mu\partial_{\nu}u^{-} - \partial_{\nu} u^{+} ) \, \text{d}s  +  \int_{\partial D} u (\partial_{\nu} v^{+} - \mu\partial_{\nu} v^{-} ) \, \text{d}s.$$
   By the boundary condition on $\partial D$ for $u$, this reduces to
    $$\int_{\partial \Omega} f \partial_{\nu} v \, \text{d}s = \int_{\partial D} \big ( [\![\partial_\nu v ]\!] - \mathscr{B} (v) \big ) u \, \text{d}s. $$
    Thus, we have that
    \begin{equation}
    \label{eqn:adjSrel}
    \langle Sf , g \rangle_{\partial D} = \int_{\partial D}u {g} \, \text{d}s = \int_{\partial \Omega} f \partial_{\nu} v \, \text{d}s =  \langle f , S^{*} g \rangle_{\partial \Omega}
        \end{equation}
    for all $f \in H^{1/2}(\partial \Omega)$ and $g \in H^{-7/2}(\partial D)$, which implies that $S^{*} g = \partial_{\nu} v \rvert_{\partial \Omega}$.
    \end{proof} 

Note that $S^{*}g$ also agrees with Definition \ref{def-adj} using \eqref{pseudo-g} for the case when $g \in H^{-2}(\partial D)$. We now discuss the solvability of \eqref{adj-s} in accordance to Definitions \ref{def-adj_H-2} and \ref{def-adj}.

\begin{lemma}
    \label{lem:H-2existence}
    There exists a unique solution $v \in \cl H \cap H_0^1 (\Omega)$ of \eqref{adj-s} in the sense of Definition \ref{def-adj_H-2} for any $g \in H^{-2} (\partial D)$.
\end{lemma}
\begin{proof}
We note that if $g \in H^{-2}(\partial D)$, then there exists a unique solution to \eqref{adj-s} via a simple variational formulation for $v \in \cl H \cap H^1_0(\Omega)$ analogous to the methods employed to prove Theorem \ref{thm:bvp-well-posed}. 
\end{proof}

Hence, using Lemma \ref{lem:H-2existence}, we can solve for $v_1 \in H^1_0 (\Omega)$ using classical methods so that $\partial_{\nu} v_1\rvert_{\partial \Omega} \in H^{-1/2} (\partial \Omega)$ using trace theorems.  With this characterization, we move on to show that \eqref{adj-s} is solvable with mean zero data in $H^{-7/2}(\partial D)$.
\begin{lemma}\label{adj-s-bvp}
    For all $g \in {\rm span} \{1\}^{\perp} \subset H^{-7/2}(\partial D)$, there exists a unique solution $v \in H^{1}_{0}(\Omega)$ satisfying \eqref{adj-s} in the sense of Definition \ref{def-adj}.
\end{lemma}
\begin{proof}
    If suffices to prove that \eqref{pseudo-g} has a unique solution. We define the Hilbert space
    $$ X:= \Bigl \{ \varphi \in H^{1/2}(\partial D) \mid \int_{\partial D} \Lambda_{\Omega \setminus D} \varphi \, \text{d}s =0 \Bigr \} $$
    and the {\color{black} compact operator $K:H^{1/2}(\partial D) \rightarrow H^{1/2}(\partial D)$ as
    \begin{align}
    Kf := \fint_{\partial D} f \, \text{d}s+ \mathscr{B}^{\#}(\Lambda_{\Omega \setminus D} - \mu \Lambda_D) f. \label{eq:Kdefinition}
    \end{align}
    The operator $K$ is compact since the right-hand side of \eqref{eq:Kdefinition} defines a bounded operator from $H^{1/2}(\partial D) \rar H^{7/2}(\partial D) \Subset H^{1/2}(\partial D).$}
    
    We claim that ${\rm Null}(I-K) = \{0\}$ and that ${\rm Range}(I-K) = X$. Indeed, by the Fredholm Alternative, we have that
    $${\rm dim} \, {\rm Null}(I-K) = {\rm dim} \, {\rm Null} (I - K^*) = {\rm dim} \, {\rm Range}(I - K)^{\perp},$$ and that ${\rm Range}(I-K)$ is closed. Here, $I$ represents the identity operator on $\partial D$. We will show that ${\rm dim} \, {\rm Null} (I - K) = 0$. To this end, suppose $(I-K) f = 0$. Then, 
    $$(\Lambda_{\Omega \setminus D} - \mu \Lambda_D ) f = \mathscr{B}(f). $$
    Let $v \in H^{1}_{0}(\Omega)$ be the unique extension satisfying 
    $$ \begin{cases} 
        \Delta v = 0 & \text{in} \enspace \Omega \setminus D, \\
        \mu \Delta v = 0 & \text{in} \enspace D, \\
        v =  f & \text{on} \enspace \partial D.
    \end{cases}
    $$
    Then, $v \in H^{1}_{0}(\Omega)$ also satisfies
    $$ \begin{cases} 
        \Delta v = 0 & \text{in} \enspace \Omega \setminus D ,\\
        \mu \Delta v = 0 & \text{in} \enspace D ,\\
         [\![\partial_\nu v ]\!] = \mathscr{B}(v) & \text{on} \enspace \partial D ,\\
        v =  0 & \text{on} \enspace \partial \Omega.
    \end{cases}
    $$
    By well-posedness of the forward problem \eqref{bvp}, since $v=0$, then $f=0$. Thus, ${\rm dim} \, {\rm Null} (I-K) = 0$, proving the claim.
\end{proof}

 We now provide a useful property of the adjoint $S^{*}$, which finalizes the argument regarding the solvability of \eqref{adj-s}.

\begin{lemma}
\label{adj-s-lem}
The adjoint $S^{*}{\color{black}:H^{-7/2}(\partial D) \rightarrow H^{-1/2}(\partial \Omega)}$ is compact and injective.
\end{lemma} 
\begin{proof}
{\color{black}For compactness, let $\Omega^{*}$ be a region that satisfies $D \subset \Omega^{*} \subset \Omega$ such that $\text{dist}(\partial D, \Omega^{*}) > 0$ and $\text{dist}(\Omega^{*}, \partial \Omega) >0$. Since $v \rvert_{\partial \Omega} = 0$, then $v \in H^{2}(\Omega \setminus \Omega^{*})$. Then, by the Trace Theorem, $\partial_{\nu} v \rvert_{\partial \Omega} \in H^{1/2}(\partial \Omega)$. The result follows from the fact that $H^{1/2}(\partial \Omega) \Subset H^{-1/2}(\partial \Omega)$.} \\
To show injectivity, suppose $g \in {\rm Null}(S^{*})$. Then $v$ satisfies
    $$    \begin{cases} 
      - \Delta v = 0 & \text{in} \enspace \Omega \setminus D, \\
      v = 0 & \text{on} \enspace \partial \Omega, \\
      \partial_{\nu}v =  0 & \text{on} \enspace \partial \Omega .
    \end{cases}
$$
By Holmgren's Theorem, $v = 0$ in $\Omega \setminus D$ which implies that $v \rvert_{\partial D} = 0$ and $\partial_{\nu} v^{+} \big \rvert_{\partial D} = 0$. Furthermore, $v$ satisfies the Dirichlet problem
$$ -\mu \Delta v = 0 \enspace \text{in} \enspace D \qquad \text{with} \qquad v = 0 \enspace \text{on} \enspace \partial D.$$
Thus, $v = 0$ in $D$, which implies that $\partial_{\nu}v^{-} \rvert_{\partial D} = 0$. We thus observe that $g = [\![\partial_\nu v]\!] - \mathscr{B}(v)  = 0$ as desired. 
\end{proof}

Combining \eqref{eqn:adjSrel} with the results from Lemmas \ref{lem:H-2existence}, \ref{adj-s-bvp} and \ref{adj-s-lem}, we observe that \eqref{adj-s} is solvable.  We return to our principal task at hand, which is to derive a symmetric factorization for the data operator $(\Lambda - \Lambda_0)$. Thus far, we have introduced and shown some properties of the well-defined maps $S:H^{1/2}(\partial \Omega) \rightarrow H^{7/2}(\partial D)$ and its adjoint $S^{*} : H^{-7/2}(\partial D) \rightarrow H^{-1/2}(\partial \Omega)$. In order to complete a symmetric factorization of $(\Lambda - \Lambda_0)$, we need to define a middle operator $T$ that connects $H^{7/2}(\partial D)$ with its dual space. Recall that $w \in H^{1}_0 (\Omega)$ is the unique solution to equation \eqref{aux-prob} where
\begin{align*}
[\![\partial_\nu w ]\!] \big \rvert_{\partial D} &= \mathscr{B}(h) - (1-\mu) \Lambda_D h + (1-\mu) \partial_{\nu} w^{-}\\ 
&=\mathscr{B}(w) + \mathscr{B}(h-w) - (1-\mu) \Lambda_D h + (1-\mu) \partial_{\nu} w^{-}.
\end{align*}
Thus, by the well-posedness of \eqref{adj-s} and Definition \ref{adj-s-def}, we have that
$$\partial_{\nu} w \big \rvert_{\partial \Omega} = Gh \qquad \text{as well as} \qquad \partial_{\nu} w \big \rvert_{\partial \Omega} = S^{*} \big (\mathscr{B}(h-w) - (1-\mu) \Lambda_D h + (1-\mu) \partial_{\nu} w^{-} \big).$$
Motivated by this, we define the operator
\begin{equation}\label{T}
    T:H^{7/2}(\partial D) \rightarrow H^{-7/2}(\partial D) \quad \text{given by} \quad Th = \mathscr{B}(h-w) - (1-\mu) \Lambda_D h + (1-\mu) \partial_\nu w^-.
\end{equation}
Here, $h \in H^{-7/2}(\partial D)$ and $w \in H^{1}_0 (\Omega)$ are as in \eqref{aux-prob}. Notice that if $h \in {\rm span} \{1\}$, then by \eqref{aux-prob}, we have that $w=0$, i.e. if $h \in {\rm span} \{1\}$, then $Th = 0$. By the well-posedness of \eqref{aux-prob}, $T$ is a bounded linear operator. In fact, for any $\xi \in H^{7/2}(\partial D)$
\begin{align*}
    \langle \xi , Th \rangle_{\partial D} &= \int_{\partial D} \xi \big [\mathscr{B}(h-w) - (1-\mu) \Lambda_D h + (1-\mu) \partial_\nu w^- \big ] \, \text{d}s \\
    &= \int_{\partial D} \xi \mathscr{B}(h) \, \text{d}s - \int_{\partial D} \xi \mathscr{B}(w) \, \text{d}s - (1-\mu) \int_{\partial D} \xi \Lambda_D h \, \text{d}s + (1-\mu) \int_{\partial D} \xi \partial_\nu w^- \, \text{d}s.
\end{align*}
For the first integral, we have that 
\begin{align*}
    \int_{\partial D} \xi \mathscr{B}(h) \, \text{d}s &\leq \| \xi \|_{H^{1/2}(\partial D)} \|\mathscr{B}(h) \|_{H^{-1/2}(\partial D)}\\ 
    &\leq  \| \xi \|_{H^{7/2}(\partial D)} \| h\|_{H^{7/2}(\partial D)}.
\end{align*}
For the second integral, 
\begin{align*}
    \int_{\partial D} \xi \mathscr{B}(w) \, \text{d}s &\leq \| \xi \|_{H^{7/2}(\partial D)} \|\mathscr{B}(w) \|_{H^{-7/2}(\partial D)} \\
    &\leq \| \xi \|_{H^{7/2}(\partial D)} \|w \|_{H^{1/2}(\partial D)} \\
    &\leq C \| \xi \|_{H^{7/2}(\partial D)} \|w \|_{H^{1}_0(\Omega)} \\
    &\leq C  \| \xi \|_{H^{7/2}(\partial D)}  \| h \|_{H^{7/2}(\partial D)},
\end{align*}
where we used the trace theorem and well-posedness of \eqref{aux-prob}. For the third integral, 
\begin{align*}
    \int_{\partial D} \xi \Lambda_D h \, \text{d}s &\leq \|\xi \|_{H^{7/2}(\partial D)} \| \Lambda_D h\|_{H^{-7/2}(\partial D)} \\
    &\leq C \|\xi \|_{H^{7/2}(\partial D)} \| h\|_{H^{-5/2}(\partial D)} \\
    &\leq C\|\xi \|_{H^{7/2}(\partial D)} \| h\|_{H^{7/2}(\partial D)},
\end{align*}
where used the boundedness of $\Lambda_D$. Lastly, for the fourth integral,
\begin{align*}
    \int_{\partial D} \xi \partial_{\nu}w^- \, \text{d}s &\leq \| \xi \|_{H^{1/2}(\partial D)} \| \partial_{\nu}w^-  \|_{H^{-1/2}(\partial D)} \\ 
    &\leq  C \| \xi \|_{H^{7/2}(\partial D)} \| w \|_{H^{1}(\Omega)} \\
    &\leq C \| \xi \|_{H^{7/2}(\partial D)} \| h \|_{H^{7/2}(\partial D)},
\end{align*}
where we used the Neumann trace theorem  and the well-posedness of \eqref{aux-prob}. Therefore,
 $$\langle \xi , Th \rangle_{\partial D} \leq C \| \xi \|_{H^{7/2}(\partial D)} \| h \|_{H^{7/2}(\partial D)}$$
 for some positive constant $C$. Recall that we had already established that $(\Lambda - \Lambda_0) = GS$. Based on our results above, we have factorized the operator $G$ as $G = S^* T$. This gives the following result.

\begin{theorem}\label{initial-factorization}
    The difference of the $DtN$ operators $(\Lambda - \Lambda_0):H^{1/2}(\partial D) \rightarrow H^{-1/2}(\partial D)$ has the following symmetric factorization 
    \begin{equation}
     \label{eqn:factorization1}
     (\Lambda - \Lambda_0) = S^{*}TS.
    \end{equation}
\end{theorem}

The fact that ${\rm Null} (\mathscr{B}) = {\rm span} \{1\}$, i.e. the fourth-order, boundary operator $\mathscr{B}$ defined on $\partial D$ has a non-trivial null space, has direct consequences on the coercivity of $T$ and the injectivity of $(\Lambda - \Lambda_0)$. Namely, $T$ is not coercive on all of $H^{7/2}(\partial D)$ and $\text{Null} (\Lambda - \Lambda_0) \neq \{0\}$. This greatly differs from past works (see e.g. \cite{granados1, granados2, harris0}) that also use the regularized factorization method for shape reconstruction. As written, the factorization of our data operator \eqref{eqn:factorization1} is not sufficient to solve the inverse shape problem, but it sets the foundation to prove some useful properties. The rest of this section is dedicated to deriving a more detailed factorization that will ultimately allow us to reconstruct $D$ from the knowledge of $(\Lambda - \Lambda_0)$. We proceed by investigating a sufficient condition in which the operator $T$ is coercive. Consistent with the strategy used to show the solvability of \eqref{adj-s}, we prove the coercivity of the operator $T$ for mean zero functions. 

\begin{theorem}\label{t-coercive}
There exists $\mu_0 \in (0,1)$ such that for all $h \in H^{7/2}(\partial D)$ with mean zero, we have $\langle h , Th\rangle_{\partial D} \geq C \| h \|_{H^{7/2}(\partial D)}$ for all $\mu > \mu_0$, where $T$ is defined in \eqref{T}.
\end{theorem}
\begin{proof}
From the boundary condition on $\partial D$ in \eqref{aux-prob}, we have that
\begin{align*}
    \langle h, T(h) \rangle_{\partial D}  & = \int_{\partial D} h\mathscr{B}(h) \, \text{d}s - \int_{\partial D} w \mathscr{B}(h) \, \text{d}s - \int_{\partial D} (1 - \mu ) h \Lambda_D h \, \text{d}s + \int_{\partial D} (1- \mu) h \Lambda_D w\, \text{d}s.
\end{align*}
From the first integral, we have that
\begin{align*}
\int_{\partial D} h\mathscr{B}(h) \, \text{d}s &= \int_{\partial D} \mu_s \ell_s^2 |\partial^2_{s} h|^2 + \mu_s |\partial_s h|^2 \, \text{d}s\\
&\geq C \| h\|^2_{H^{2}(\partial D)},
\end{align*}
where the last inequality holds for some positive constant $C>0$ obtained from that fact that $h$ is mean zero. From the boundary condition on $\partial D$, the second integral gives
\begin{align*}
    -\int_{\partial D} w \mathscr{B}(h) \, \text{d}s &= \int_{\partial D} -w [\![\partial_\nu w ]\!] - (1-\mu) w \Lambda_D h + (1-\mu) w \partial_{\nu}w^{-} \, \text{d}s\\
    &=  \int_{\partial D} -w (\partial_{\nu}w^{+}-\partial_{\nu}w^{-}) - (1-\mu) w \Lambda_D h \\
    &= \| \nabla w\|^{2}_{L^{2}(\Omega)} - \int_{\partial D} (1-\mu) w \Lambda_D h \, \text{d}s
\end{align*}
applying Green's second identity to obtain the last equality. We note that since $\Lambda_D$ is self-adjoint, we obtain that 
\[  \langle h, T(h) \rangle_{\partial D} = \langle h , {\color{black}(}\mathscr{B} -(1-\mu) \Lambda_D {\color{black})h}\rangle_{\partial D}  +  \| \nabla w\|^{2}_{L^{2}(\Omega)}.\]
We claim that there exists $\mu_0 \in (0,1)$ such that the operator $\mathscr{B} -(1-\mu) \Lambda_D$ is positive for all $\mu > \mu_0$. To show the existence for such value $\mu_0$, we proceed as follows. Since $\partial D$ is smooth, there exists an order 0 operator $B:L^{2}(\partial D) \rightarrow L^{2}(\partial D)$ such that the interior DtN map has the decomposition 
    \begin{equation}
    \label{eqn:DtNrep}
         \Lambda_D = \sqrt{-\partial^2_s} + B.
          \end{equation}
    This result is contained in for instance \cite{polterovich2015heat}, Section 4 or \cite{taylor1996partial}, Chapter 12, Proposition $C1$ for proofs or \cite{girouard2022dirichlet,taylor2020dirichlet} for a nice overview.  Since $B$ is a bounded, order $0$ operator and $\sqrt{-\partial_s^2}$ is an order $1$ operator that can be simultaneously diagonalized with $\mathscr{B}$, we can see that $(\mathscr{B} -(1-\mu) \Lambda_D )  > 0$ for all $\mu$ such that 
    \[
    \mu_s \lambda_1 + \mu_s \ell_s \lambda_1^2 -(1 - \mu) \sqrt{\lambda_1} - |1-\mu|\| B \|_{L^2(\partial D) \to L^2(\partial D)} > 0,
    \]
where $\lambda_1$ is the first non-zero eigenvalue of $\sqrt{-\partial^2_s}$ as introduced \eqref{eig-sys} and $\| \cdot \|_{L^2(\partial D) \to L^2(\partial D)}$ is the operator norm from $L^2(\partial D) \to L^2(\partial D)$.  To see this, consider the function
\[
u_\mu (z) = \mu_s z^2 + \mu_s \ell_s z^4 -(1 - \mu) z - |1-\mu|\| B \|_{L^2(\partial D) \to L^2(\partial D)} .
\]
If $u_\mu (z) > 0$ for all $z \geq \sqrt{\lambda_1}$, then $\langle (\mathscr{B} -(1-\mu) \Lambda_D ) \phi_n, \phi_n \rangle_{\partial D} > 0$ for each eigenfunction $\phi_n$, $n \geq 1$ and hence it is a positive operator using that
\[
\langle - (1-\mu) B \phi, \phi \rangle_{\partial D} > - |1-\mu| \| B \|_{L^2(\partial D) \to L^2(\partial D)} \| \| \phi \|_{L^2(\partial D)}^2
\]
for any $\phi \in L^2 (\partial D)$. That is, we use that $\mathscr{B} -(1-\mu) \Lambda_D > \mathscr{B} -(1-\mu) \sqrt{-\partial_s^2} - |1-\mu| \| B \|_{L^2(\partial D) \to L^2(\partial D)}$ and that all the operators on the right are simultaneously diagonalizable using the basis $\{ \phi_n \}$.
We have that
\begin{align*}
    u_\mu'(z) & = 2 \mu_s z + 4 \mu_s \ell_s z^3 -(1 - \mu), \\
    u_\mu''(z) & = 2 \mu_s + 12 \mu_s \ell_2 z^2 > 0.
\end{align*}
So, provided
\begin{equation}
    \label{eqn:posconds}
    u_\mu (\sqrt{\lambda_1})>0 \quad \text{and} \quad u_\mu'(\sqrt{\lambda_1}) >0 ,
\end{equation}
we have that $(\mathscr{B} -(1-\mu) \Lambda_D ) > 0$.  Since $2 \mu_s \sqrt{\lambda_1} + 4 \mu_s \ell_s \sqrt{\lambda_1}^3 > 0$ and $\mu_s \lambda_1 + \mu_s \ell_s \lambda_1^2>0$, and $u_\mu(\sqrt{\lambda_1})$ and $ u_\mu'(\sqrt{\lambda_1})$ are continuous functions in $\mu$, we can find $0<\mu_0<1$ such that \eqref{eqn:posconds} is satisfied for all $1 \geq \mu > \mu_0$. Hence, $(\mathscr{B} -(1-\mu) \Lambda_D ) > 0$ for all $\mu > \mu_0$. We note that positivity for the case where $\mu \geq 1$ follows trivially from the fact that $\Lambda_D$ is a non-negative definite operator. Note, $\mu_0$ as described here is sufficient as a lower bound on the values of $\mu$ but does not describe the sharpest possible value.  The value $\mu_0$ is dependent upon the value of $\lambda_1$ as well as the parameters $\mu_s, \ell_s$, which is to be expected.


We now finalize the coercivity argument. By the above calculations we have that there exist $C_1, C_2 > 0$ such that
\begin{align*}
    \langle h, T(h) \rangle_{\partial D}  &= \langle {\color{black}h, (\mathscr{B} -(1-\mu) \Lambda_D ) h} \rangle + \| \nabla w\|^{2}_{L^{2}(\Omega)} \\
    &\geq C_1 \big( \|h \|^{2}_{H^{2}(\partial D)} + \| w\|^{2}_{H^{1}_0(\Omega)} \big) \\
    &\geq C_2 \|h\|^{2}_{H^{7/2} (\partial D)}
\end{align*}
by Poincar\'e's inequality and \eqref{eqn:wbd} for $\mu$ such that $(\mathscr{B} -(1-\mu) \Lambda_D )  > 0$.  
\end{proof}

For the remainder of this work, we assume that 
\begin{align}
\mu > \mu_0
\end{align}
where $\mu_0$ is given in Theorem \ref{t-coercive}. Given that the operator $T$ is coercive on mean zero functions, we define $P:H^{7/2}(\partial D) \rightarrow H^{7/2}(\partial D)$ as 
\[P \phi = \phi - \fint_{\partial D} \phi \, \text{ds},\]
i.e. the orthogonal projection onto ${\rm span} \{1\}^{\perp}$. We may identify ${\rm span} \{1\}$ with
$$ P^{\perp} \phi = \fint_{\partial D} \phi \, \text{d}s, $$
i.e. $P^{\perp}$ is the orthogonal projection onto constants defined on $\partial D$. We note that $P$ and $P^{\perp}$ are self--adjoint. Thus, we have that $P^{*}TP = PTP$. Furthermore, recall that if $h \in {\rm span} \{1\}$, then $Th=0$. Therefore, $T = PTP$. This yields a more detailed factorization for the difference of the DtN maps:
\begin{equation}\label{sym-pro-fac}
    (\Lambda - \Lambda_0) = S^{*}PTPS.
\end{equation}
With this new factorization, we are able to prove the following properties. 
\begin{theorem}\label{DtN-com-inv-ran}
    The difference of the DtN mappings $(\Lambda - \Lambda_0 ):H^{1/2} (\Omega) \rightarrow H^{-1/2}(\partial \Omega)$ is compact and positive. Furthermore, the range of $(\Lambda - \Lambda_0)$ is dense in ${\rm span} \{1\}^{\perp}$ and ${\rm Null}(\Lambda - \Lambda_0) = {\rm span} \{1\}$.
\end{theorem}
\begin{proof}
    {\color{black}Compactness follows from Lemma \ref{adj-s-lem}, since $S^{*}$ is compact} and $T$ is bounded. Note that from the factorization given in \eqref{sym-pro-fac} we have that
     \begin{equation}\label{DtN-t-coercive}
     \langle f , S^* PTPSf \rangle_{\partial \Omega} = \langle PSf , TPSf \rangle_{\partial D} \geq C \| PSf \|^2_{H^{7/2}(\partial D)},
     \end{equation}
    where we have used the coercivity of $T$ on mean zero functions as shown in Theorem \ref{t-coercive}. Thus, $(\Lambda - \Lambda_0)$ is positive. We now simultaneously show that the range of $(\Lambda - \Lambda_0)$ is dense in ${\rm span} \{1\}^{\perp}$ and ${\rm Null} (\Lambda - \Lambda_0)$ is ${\rm span}\{1\}$. To this end, suppose $f \in H^{1/2}(\partial \Omega)$ is an annihilator for ${\rm Range}(\Lambda - \Lambda_0)$ or $f\in {\rm Null} (\Lambda - \Lambda_0)$. In either case, \eqref{DtN-t-coercive} implies that
        $$0 \geq C \| PSf \|^2_{H^{7/2}(\partial D)}.$$
    Thus, $PSf = Pu\rvert_{\partial D} = 0$. Then there exists $\alpha\in {\rm span} \{1\}$ such that $u \rvert_{\partial D} = \alpha$. Note that $u$ satisfies 
    $$-\mu \Delta u = 0 \enspace \text{in} \enspace D \quad \text{with} \quad u = \alpha \enspace \text{on} \enspace \partial D. $$
    Thus, $u = \alpha$ in $D$. This implies that $\partial_{\nu}u^{-}\rvert_{\partial D} = 0$. Also, since $u \rvert_{\partial D}= \alpha$, then we have that $\mathscr{B}(u) = 0$ on $\partial D$. By the boundary condition of \eqref{bvp}, $\partial_{\nu}u^{+}\rvert_{\partial D} = 0$. Note that $u$ must satisfy
     $$-\Delta u = 0 \enspace \text{in} \enspace \Omega \setminus D \quad \text{with} \quad u = \alpha \enspace \text{on} \enspace \partial D  \quad \text{and} \quad \partial_{\nu}u^{+} = 0 \enspace \text{on} \enspace \partial D.$$ Therefore, $u-\alpha$ satisfies 
     $$-\Delta (u - \alpha) = 0 \enspace \text{in} \enspace \Omega \setminus D \quad \text{with} \quad u - \alpha =0\enspace \text{on} \enspace \partial D  \quad \text{and} \quad \partial_{\nu}(u-\alpha)^{+} = 0 \enspace \text{on} \enspace \partial D.$$ 
     By Holmgren's Theorem, $u = \alpha$ in $\Omega \setminus D$. That is, $u = \alpha$ in $\Omega$. Therefore, $f \in {\rm span}  \{1\}$. Thus, ${\rm Range}(\Lambda - \Lambda_0)$ is dense in ${\rm span} \{1\}^{\perp}$ and ${\rm Null}(\Lambda - \Lambda_0) = {\rm span} \{1\}$.
\end{proof}

The factorizations \eqref{eqn:factorization1} and \eqref{sym-pro-fac} provide the baseline to study the inverse parameter problem in Section \ref{inverse-parameter-problem} as well as the inverse shape problem in Section \ref{inverse-shape-problem}.

\section{\bf Uniqueness of the Inverse Parameter Problem} \label{inverse-parameter-problem}

In this section, we study the uniqueness of the inverse parameter problem of determining the ratio of the shear modulus $\mu$ and the boundary parameters $\mu_s$ and $\ell_s$ provided that the boundary $\partial D$ is given. Recovery of coefficients have been studied for problems in electrostatics \cite{chaabane} and inverse scattering \cite{cakoni-lee-monk}. We will establish the uniqueness of the aforementioned parameters from knowledge of the DtN operator $\Lambda :H^{1/2}(\partial \Omega) \rightarrow H^{-1/2}(\partial \Omega)$. To this end, we first consider the following density result.

\begin{lemma}\label{s-dense-range}
    The set 
    \[ \mathcal{U} = \big \{ u \big \rvert_{\partial D} \in H^{7/2}(\partial D) \mid u \in \mathcal{H} \enspace \text{solves} \enspace \eqref{bvp} \enspace \text{for any} \enspace f \in H^{1/2}(\partial \Omega)\big\}\]
    is dense in $H^{7/2}(\partial D)$.
\end{lemma}
The proof of the above Lemma is an immediate consequence of Lemma \ref{adj-s-lem}, where we proved that the operator $S^{*}$ is injective. By definition, $\mathcal{U} = {\rm Range}(S)$. Hence, $\mathcal{U}$ is dense in $H^{7/2}(\partial D)$.

\begin{theorem}
\label{thm:mechprecise}
    If $\partial D$ is smooth and the parameters $\mu$, $\mu_s$, and $\ell_s$ are positive, then the mapping $(\mu , \mu_s , \ell_s) \longmapsto \Lambda$ is injective.
\end{theorem}
\begin{proof}
    Given $f \in H^{1/2}(\partial \Omega)$, let $u_i$ be the solution to \eqref{bvp} with parameters $(\mu^i , \mu^i_s, \ell_s^i )$ and $\Lambda_i$ be the corresponding DtN operator for each $i = 1,2$. Assume that the DtN operators $\Lambda_1$ and $\Lambda_2$ coincide. Then, $\partial_{\nu} u_1 = \partial_{\nu} u_2$ and $u_1 = u_2$ on $\partial \Omega$, which implies that $u_1 = u_2$ in $\Omega \setminus D$ from Holmgren’s Theorem. By the trace theorem, we also have that $u_1 = u_2$ on $\partial D$. It follows that $u_1 - u_2$ satisfies the Dirichlet problem in $D$ with zero Dirichlet data. Thus, we conclude that $u_1 = u_2$ is $D$. Consequently, $u_1 = u_2$ in all of $\Omega$. Inspired by the boundary condition on $\partial D$, we proceed by defining the fourth order operator $\mathscr{L}:H^{2}(\partial D) \rightarrow H^{-2}(\partial D)$ as
    \[\mathscr{L} \xi= (\mu^{(1)} - \mu^{(2)})\Lambda_D \xi+  \partial^{2}_{s}(\mu_s^{(1)} \ell_s^{2,(1)} - \mu_s^{(2)} \ell_s^{2,(2)}) \partial^{2}_s\xi - \partial_s (\mu_s^{(1)} - \mu_s^{(2)} )\partial_s \xi. \]
    Using that $u_1 = u_2$ in $\Omega$, we obtain that  
    \[0 = \mathscr{L}u_1.\]
    For any $\phi \in H^{2}(\partial D)$, consider the function $\psi \in H^{-2}(\partial D)$ given by
    \[\psi := \mathscr{L} \phi. \]
    Then, 
    \[0 = \int_{\partial D} \phi \mathscr{L}u_1 \, \text{d}s = \int_{\partial D} u_1 \mathscr{L}\phi \, \text{d}s = \int_{\partial D} u_1 \psi \, \text{d}s.\]
    Therefore, $\psi \in \mathcal{U}^{\perp}$ and by Lemma \ref{s-dense-range}, we have that $\psi = 0$. That is, 
    \[ \mathscr{L} \phi = 0 \quad \text{for all} \enspace \phi \in H^{2}(\partial D).\]
    Let us recall the formula in \eqref{eqn:DtNrep},  $\Lambda_D = \sqrt{-\partial^2_s} + B$ for $B$ a bounded operator of order $0$.  
    We had previously established in \eqref{eig-sys} that we have the following eigensystem on $\partial D$
    \[\{\lambda_n , \phi_n\}_{n\in \mathbb{N} \cup \{0\}} \quad \text{such that} \quad -\partial^{2}_{s}\phi_n = \lambda^2_n \phi_n\enspace \text{for all} \enspace n \in \mathbb{N} \cup \{0\},\]
where $\lambda_n >0$ for all $n \in \bbN$, $\lambda_n \rightarrow \infty$ as $n \rightarrow \infty$, and $\phi_n$ is an orthonormal basis of $L^{2}(\partial D)$. Thus, we have that 
    \[\mathscr{L}= (\mu^{(1)} - \mu^{(2)}) \big (\sqrt{-\partial^2_s} + B \big ) + \partial^{2}_s \big(\mu_s^{(1)} \ell_s^{2,(1)} - \mu_s^{(2)} \ell_s^{2(2)} \big)\partial^{2}_s -\partial_s \big( \mu_s^{(1)} - \mu_s^{(2)} \big) \partial_s. \]
    By our orthogonality result from above, the eigensystem on $\partial D$, and the fact that all of these coefficients are positive, it holds that
    \begin{align*}
        0 &= \langle \phi_n , \mathscr{L} \phi_n \rangle_{\partial D} \\
        &= (\mu^{(1)} - \mu^{(2)}) \lambda_n + (\phi_n , \big(\mu^{(1)} - \mu^{(2)}) B \phi_n \big)_{L^{2}(\partial D)} + (\mu_s^{(1)} \ell_s^{2,(1)} - \mu_s^{(2)} \ell_s^{2(2)}) \lambda^4_n + (\mu_s^{(1)} - \mu_s^{(2)}) \lambda^2_n.
    \end{align*}
    Since $\lambda_n > 0$ for all $n \in \mathbb{N}$, we have that
    \begin{align*}
        0 &= \frac{\mu^{(1)} - \mu^{(2)}}{\lambda^3_n} + \frac{(\phi_n , \big(\mu^{(1)} - \mu^{(2)}) B \phi_n \big)_{L^{2}(\partial D)}}{\lambda^4_n} + \mu_s^{(1)} \ell_s^{2,(1)} - \mu_s^{(2)} \ell_s^{2(2)} + \frac{\mu_s^{(1)} - \mu_s^{(2)}}{\lambda^2_n}\\
        &\rightarrow \mu_s^{(1)} \ell_s^{2,(1)} - \mu_s^{(2)} \ell_s^{2(2)} \quad \text{as} \enspace n \rightarrow \infty
    \end{align*}
   using the fact that $B$ is an order zero operator. Thus, $\mu_s^{(1)} \ell_s^{2,(1)} - \mu_s^{(2)} \ell_s^{2(2)} =0$, i.e. $\mu_s^{(1)} \ell_s^{2,(1)} = \mu_s^{(2)} \ell_s^{2(2)}$. This yields that
    \[ 0 = (\mu^{(1)} - \mu^{(2)}) \lambda_n + (\phi_n , \big(\mu^{(1)} - \mu^{(2)}) B \phi_n \big)_{L^{2}(\partial D)} + (\mu_s^{(1)} \ell_s^{2,(1)} + (\mu_s^{(1)} - \mu_s^{(2)}) \lambda^2_n.\]
    By repeating this limiting argument two more times, we determine that $\mu_s^{(1)} = \mu_s^{(2)}$, which implies that $\ell_s^{2,(1)} = \ell_s^{2,(2)}$, and also that $\mu^{(1)} = \mu^{(2)}$, which proves the result.
\end{proof}

\section{\bf The Inverse Shape Problem} \label{inverse-shape-problem}

{\color{black}To recover the inclusion $D$ from the data operator $(\Lambda-\Lambda_0)$, we employ a qualitative reconstruction method based on the linear sampling and factorization methods (see, e.g., \cite{cakoni-colton-2003,cakoni-colton-haddar,colton-haddar-monk,colton-haddar-piana,harris1,kirsch2005factorization,kirsch2007factorization}), which have been widely used in inverse scattering and tomography. Our approach adapts the regularized factorization method to the present elasticity setting, following the general strategy of Theorem 2.3 in \cite{harris1}.

{\color{black}
A fundamental obstacle is that the data operator $(\Lambda-\Lambda_0)$ possesses a nontrivial, one-dimensional null space. Indeed, Theorem~\ref{DtN-com-inv-ran} shows that
$
{\rm Null}(\Lambda-\Lambda_0)= {\rm span}\{1\},
$
so the initial factorization
\[
(\Lambda-\Lambda_0)=S^{*}TS
\]
from Theorem~\ref{initial-factorization} is insufficient for recovering the inclusion. To overcome this difficulty, we derive the refined factorization
\[
(\Lambda-\Lambda_0)=(PS)^{*}TPS,
\]
given in \eqref{sym-pro-fac}, where $P$ denotes the orthogonal projection onto ${\rm span}\{1\}^{\perp}$. By removing the constant mode responsible for the null space, this factorization reduces the inverse problem to characterizing the range of $(PS)^{*}$. This leads to a two-point sampling criterion for recovering $D$. Once a single reference point inside the inclusion has been identified, the reconstruction reduces to the classical one-point sampling procedure.


To connect this range characterization to measured data on the known exterior boundary $\p \Om$, we use the fact that $(\Lambda - \Lambda_0)$ is non-negative and thus admits a non-negative square root. That is, there exists bounded operator $Q:H^{1/2}(\partial \Omega) \rightarrow L^{2}(\partial \Omega)$ such that $(\Lambda - \Lambda_0) = Q^{*}Q$. Given this new factorization for the difference of the DtN mappings, we have the following result which describes an important connection between the factorizations $(\Lambda - \Lambda_0) = Q^{*}Q$ and $(\Lambda - \Lambda_0) = (PS)^{*}TPS$.
} 

\begin{theorem}\label{range-adjs-q-s}
    Let $(\Lambda - \Lambda_0):H^{1/2}(\partial \Omega) \rightarrow H^{-1/2}(\partial \Omega)$ have factorizations $(\Lambda - \Lambda_0) = Q^{*}Q$ and $(\Lambda - \Lambda_0) = (PS)^{*}TPS$ such that 
    \[
    S:H^{1/2}(\partial \Omega) \rightarrow H^{7/2}(\partial D), \enspace T:H^{7/2}(\partial D) \rightarrow H^{-7/2}(\partial D), \quad \text{and} \quad Q:H^{1/2}(\partial D) \rightarrow L^{2}(\partial D)
    \]
    are bounded operators and $P$ is the orthogonal projection onto ${\rm span} \{1\}^{\perp}$. If $T$ is coercive on ${\rm Range}(PS)$, then 
    $${\rm Range}(Q^{*}) = {\rm Range}\big(S^{*}\rvert_{{\rm span} \{ 1\}^{\perp}} \big).$$
\end{theorem}
\begin{proof}
    Since $(\Lambda - \Lambda_0) = Q^{*}Q$ and $(\Lambda - \Lambda_0) = (PS)^{*}TPS$, then we have that
    $$\| Qf\|^{2}_{L^{2}(\partial \Omega)} = \langle f , (\Lambda - \Lambda_0)f \rangle_{\partial \Omega} = \langle PSf , TPSf \rangle_{\partial D} \quad \text{for all} \enspace f \in H^{1/2}(\partial \Omega).$$
    Since $T$ is coercive on ${\rm Range}(PS)$ from Theorem \ref{t-coercive}, there exists constant $C >0$ such that
    $$C \|PSf \|^2_{H^{7/2}(\partial D)}\leq \| Qf\|^{2}_{L^{2}(\partial \Omega)} \leq \| T\|_{\text{op}} \|PSf \|^2_{H^{7/2}(\partial D)},  $$
    where here $\| \cdot \|_{\text{op}}$ denotes the operator norm from $H^{7/2}(\partial D) \rightarrow H^{-7/2}(\partial D)$. Therefore, by Theorem 1 in \cite{embry}, ${\rm Range}(Q^{*}) = {\rm Range}((PS)^{*})$, which proves the claim.
\end{proof}

{\color{black} 
Sampling methods typically characterize the unknown region through singular solutions of an associated background problem. In this setting, the relevant singular solutions are generated by the Dirichlet Green's function corresponding to \eqref{lifting}, i.e., the problem where the bulk $\Omega$ has homogeneous unit shear modulus, defined by
\begin{equation}\label{greens}
    -\Delta \mathbb{G}(\cdot , z) = \delta(\cdot - z) \enspace \text{in} \enspace \Omega \quad \text{and} \quad \mathbb{G}(\cdot , z)\big \rvert_{\Omega} = 0.
\end{equation}

The following result connects the differences of boundary traces of Dirichlet Green's functions on $\p \Om$ to the inclusion $D$. This characterization is fundamental for the sampling method which connects the region of interest to an ill-posed equation involving the data operator.
}

{\color{black}
\begin{theorem}\label{thm:difference_G}
Let $z_1 \neq z_2 \in \Om$. Then we have that
\[\partial_{\nu} \mathbb{G}(\cdot , z_1) \big \rvert_{\partial \Omega} - \partial_{\nu} \mathbb{G}(\cdot , z_2) \big \rvert_{\partial \Omega}\in {\rm Range} \big( S^{*} \big \rvert_{{\rm span} \{1\}^{\perp}} \big)\quad \text{if and only if} \quad z_1 , z_2 \in D. \]
\end{theorem}
\begin{proof}
To prove the forward direction, assume at least one of $z_1, z_2 \in \Om \setminus D$ (i.e. at least one of $z_1 , z_2 \notin D$). Suppose by contradiction that there exists $g_{z_1 , z_2} \in \text{span}\{1\}^{\perp} \subset H^{-7/2}(\p D)$ such that $S^{*}g_{z_1 , z_2} = \partial_{\nu} \mathbb{G}(\cdot , z_1) \rvert_{\partial \Omega} - \partial_{\nu} \mathbb{G}(\cdot , z_2) \rvert_{\partial \Omega}$. This implies that there exists $v_{z_1, z_2} \in H^{1}_{0}(\Om)$ such that
 $$   \begin{cases} 
      - \Delta v_{z_1,z_2} = 0 & \text{in} \enspace \Omega \setminus D, \\
      -\mu \Delta v_{z_1,z_2} = 0 & \text{in} \enspace D, \\
      [\![\partial_\nu v_{z_1,z_2}]\!] = \mathscr{B}(v_{z_1,z_2}) + g_{z_1,z_2} & \text{on} \enspace \partial D
    \end{cases}
$$
Furthermore, $\p_{\nu} v_{z_1,z_2} \rvert_{\partial \Omega} = \partial_{\nu} \mathbb{G}(\cdot , z_1)  \rvert_{\partial \Omega} - \partial_{\nu} \mathbb{G}(\cdot , z_2) \rvert_{\partial \Omega}$. This implies that $W_{z_1 , z_2}:= v_{z_1,z_2} - \big(\mathbb{G}(\cdot , z_1) \rvert_{\partial \Omega} - \mathbb{G}(\cdot , z_2) \rvert_{\partial \Omega} \big) $ satisfies
$$    \begin{cases} 
      - \Delta W_{z_1,z_2} = 0 & \text{in} \enspace \Omega \setminus (D\cup\{z_1,z_2\}) ,\\
      W_{z_1,z_2} = 0 & \text{on} \enspace \partial \Omega ,\\
      \partial_{\nu}W_{z_1,z_2} =  0 & \text{on} \enspace \partial \Omega.
    \end{cases}
$$
By Holmgren's Theorem, $W_{z_1,z_2} = 0$ in $\Omega \setminus (D\cup\{z_1 , z_2\})$, i.e. $v_{z_1,z_2} = \mathbb{G}(\cdot , z_1) - \mathbb{G}(\cdot , z_2)$ in $\Omega \setminus (D\cup\{z_1 , z_2\})$. By interior elliptic regularity, $v_{z_1,z_2}$ is continuous in $\Omega$. However, $\mathbb{G}(\cdot , z_1) - \mathbb{G}(\cdot , z_2)$ has a singularity at $z_1$ or $z_2$. This proves the forward direction since $|v(z_j)| < \infty$ for $j = 1,2$ whereas
$$ |\mathbb{G}(x,z_1) - \mathbb{G}(x,z_2)| \rightarrow \infty \quad \text{as} \quad x \rightarrow z_1 \quad \text{(inclusive) or} \quad x \rightarrow z_2 .$$

Conversely, suppose that both $z_1,z_2 \in D$. Let $\xi \in H^{1}(D)$ be the solution to the following Dirichlet problem
$$-\mu \Delta \xi = 0 \quad \text{in} \enspace D \qquad \text{with} \qquad \xi \big \rvert^{-}_{\partial D} = \big[ \mathbb{G}(\cdot , z_1) - \mathbb{G}(\cdot , z_2)\big ]^{+}_{\partial D}. $$
We define $v_{z_1,z_2}$ as 
$$  v_{z_1,z_2} =    \begin{cases} 
      \mathbb{G}(\cdot , z_1) - \mathbb{G}(\cdot , z_2)& \text{in} \enspace \Omega \setminus D, \\
      \xi & \text{in} \enspace D
    \end{cases}
$$
and we will show that $v_{z_1 , z_2}$ solves \eqref{adj-s} with mean-zero data on $\p D$. By construction, $v_{z_1,z_2} \in H^{1}_0 (\Omega)$ such that it is harmonic in $\Omega \setminus D$ since $z_1,z_2 \in D$. Furthermore, $-\mu \Delta v_{z_1,z_2} = 0$ in $D$. We also have that $\partial_{\nu} v_{z_1,z_2} \rvert_{\partial \Omega} = \partial_{\nu} \mathbb{G}(\cdot , z_1) \rvert_{\partial \Omega} - \partial_{\nu} \mathbb{G}(\cdot , z_2) \rvert_{\partial \Omega}$. It remains to show that 
$$ g_{z_1,z_2} := [\![\partial_\nu v_{z_1,z_2}]\!] - \mathscr{B}(v_{z_1,z_2}) \in {\rm span}\{1\}^{\perp} \subset H^{-7/2}(\partial D).$$
We first show that $g_{z_1,z_2} \in {\rm span}\{1\}^{\perp}$. To prove this, we note that for any $z\in \Om$, 
\[\fint_{\p \Om} \p_{\nu}\mathbb{G}(\cdot , z) \, {\rm d}s = -\frac{1}{|\p \Om|} \]
by applying Green's first identity. We now consider
\begin{align*}
\int_{\p D} g_{z_1,z_2} \, {\rm d}s &= \int_{\p D} [\![\partial_\nu v_{z_1,z_2}]\!] - \mathscr{B}(v_{z_1,z_2})  \, {\rm d}s\\ 
&= \int_{\Om \setminus D} \Delta v_{z_1,z_2} \, {\rm d}x + \int_{\p D} \p_{\nu} v_{z_1, z_2}^{+} \, {\rm d}s \\
&= \int_{\p \Om} \p_{\nu}v_{z_1,z_2} \, {\rm d}s \\
&= 0
\end{align*}
where we once again used Green's first identity and the definition of $v_{z_1,z_2}$. Hence, $g_{z_1,z_2} \in {\rm span}\{1\}^{\perp}$. To show that $g_{z_1,z_2} \in H^{-7/2}(\p D)$, we have that by the Neumann trace theorem 
\[ [\![\partial_\nu v_{z_1,z_2}]\!] = \partial_{\nu}\mathbb{G}(\cdot , z_1) \big \rvert^{+}_{\partial D} - \partial_{\nu}\mathbb{G}(\cdot , z_2) \big \rvert^{+}_{\partial D} - \partial_{\nu} \xi \big \rvert^{-}_{\partial D} \in H^{-1/2}(\partial D) \]
since $\mathbb{G}(\cdot , z_j) \in H^{1}(\Omega \setminus D)$ for $j=1,2$ and $\xi \in H^{1}(D)$. Since $\mathscr{B}$ is a fourth-order operator on $\partial D$ and $v_{z_1,z_2}  \rvert_{\partial D} \in H^{1/2} (\partial D)$, then $\mathscr{B}(v_{z_1,z_2}) \in H^{-7/2}(\partial D)$. Hence, $v_{z_1,z_2}$ satisfies \eqref{adj-s} with $g_{z_1,z_2} \in {\rm span}\{1\}^{\perp} \subset H^{-7/2}(\p D)$. By Theorem \ref{adj-s-def}, it follows that $S^{*}g_{z_1,z_2} = \partial_{\nu} \mathbb{G}(\cdot , z_1) \rvert_{\partial \Omega} - \partial_{\nu} \mathbb{G}(\cdot , z_2) \rvert_{\partial \Omega}$, which proves the claim.
\end{proof}
}

{\color{black}
Recall from Theorem \ref{range-adjs-q-s} that ${\rm Range}(Q^{*}) = {\rm Range}((PS)^{*})$. Together with Theorem \ref{thm:difference_G}, this identifies precisely which differences of Green's function traces belong to the range of $(\Lambda - \Lambda_0)^{1/2} = Q^{*}$. Hence, fixing one point $z_0 \in D$ and varying a second sampling point $z \in \Om$, yields the following two-point reconstruction criterion.

\begin{theorem}\label{main-thm}
    Let $z_0 \in D$.  The difference of the DtN mappings $(\Lambda - \Lambda_0):H^{1/2}(\partial \Omega) \rightarrow H^{-1/2}(\partial \Omega)$ uniquely determines $D$ in that for any $z \in \Omega$
    $$z \in D \quad \text{if and only if} \quad \liminf_{\alpha \rightarrow 0} \langle f^{z}_{\alpha} , (\Lambda - \Lambda_0) f^{z}_{\alpha} \rangle_{\partial \Omega} < \infty$$
    for $f^{z}_{\alpha}$ the regularized solution to $(\Lambda - \Lambda_0) f^{z}_{\alpha} = \partial_{\nu} \mathbb{G}(\cdot , z) \big \rvert_{\partial \Omega} - \partial_{\nu} \mathbb{G}(\cdot , z_0) \big \rvert_{\partial \Omega}$. 
\end{theorem}

Thus, the reconstruction begins by identifying a pair of sampling points contained in $D$. By arbitrarily fixing one of these two points, we recover the rest of $D$ by single-point sampling via Theorem \ref{main-thm}. That is, once a reference point in $D$ has been identified, the reconstruction proceeds exactly as a standard (single-point) sampling method. We summarize this procedure after choosing grid points $\{z_k\} \subset \Omega$:}

\begin{algorithm}
	\caption{Reconstruction of $D$}
    \label{alg:pseudocode}

	\begin{algorithmic}[1]
	
	\State \textbf{Step 1 (identification of reference point):}
	Find a pair $(z_{k_0}, z_{k_1})$ with $k_0 \neq k_1$ such that the criterion of Theorem \ref{main-thm} is satisfied.  
	Then $z_{k_0}, z_{k_1} \in D$.
	
	\State \textbf{Step 2 (one-point sampling):}
	Fix $z_{k_0} \in D$. For each $k \neq k_0,k_1$, identify as belonging to $D$ all points $z_k$ such that $(z_{k_0},z_k)$ satisfies the criterion of Theorem \ref{main-thm}.
	\end{algorithmic} 
\end{algorithm}

This concludes the shape reconstruction problem for an extended region. In the following section, we provide some numerical experiments for reconstructing $D$.

\section{Numerical Experiments}
\label{numerics}
{\color{black}In this section, we present numerical examples for the regularized factorization method developed in Section \ref{inverse-shape-problem} for solving the inverse shape problem. As summarized in Algorithm \ref{alg:pseudocode}, the reconstruction procedure consists of two stages. The first stage identifies a pair of sampling points contained in the inclusion where we arbitrarily fix one of them, and refer to it as the anchor point $z_0 \in D$. The second stage reduces to a standard sampling procedure, where the anchor point $z_0 \in D$ is fixed and a second sampling point is varied throughout $\Omega$. Since the reconstruction of the inclusion is primarily determined by the second step, our numerical experiments focus on this stage. We begin with a preliminary experiment illustrating the significance of the location of the anchor point. For simplicity, all subsequent experiments fix $z_0=(0,0) \in D$, allowing us to investigate the effect of the regularization parameter, random noise, and material parameters $\mu, \mu_s$, and $\ell_s$ independently of the choice of anchor point.}

Our numerical experiments are done in \texttt{MATLAB} 2020a. For simplicity, we consider the problem where $\Omega$ is the unit disk in $\mathbb{R}^2$ and the inner domain $D$ is a smaller disk also centered at the origin. The trace spaces $H^{\pm 1/2} (\partial \Omega)$  in this case can be identified with $H_{\text{per}}^{\pm 1/2} [0 , 2 \pi ]$. To apply Theorem \ref{main-thm}, we need the normal derivative of the Dirichlet Green's function $\mathbb{G}(\cdot , z)$ on the unit disk. It is well known that for this case, it is given by the Poisson kernel which we express in polar coordinates
$$\partial_{\nu} \mathbb{G}\big ( \cdot \, , z \big ) \big|_{\partial \Omega} = \frac{1}{2 \pi} \bigg[ \frac{1 - |z|^2 }{|z|^2 + 1 - 2 |z| \text{cos}(\cdot \,- \theta_{z})}   \bigg ],$$ 
where $\theta_{z}$ is the polar angle of the sampling point $z \in \Omega$.

We let the matrix $\textbf{A} \in {\color{black}\mathbb{R}}^{N \times N}$ represent the discretized operator $(\Lambda - \Lambda_0)$ {\color{black}and given an anchor point $z_0 \in D$, we build the vector 
\begin{equation}\label{discretized-difference}
\textbf{b}_z = \big [ \partial_{\nu} \mathbb{G}\big (  \theta_j ,z \big ) -  \partial_{\nu} \mathbb{G}\big (  \theta_j ,z_0 \big )\big].
\end{equation}
}
In our numerical experiments, we add random noise to the discretized operator \textbf{A} such that 
\begin{equation} \label{delta-error}
\text{\textbf{A}}^{\delta} = \big [ \text{\textbf{A}}_{i,j} \big( 1 + \delta \text{\textbf{E}}_{i,j} \big) \big ]_{i,j=1}^{N} \quad \text{where} \quad \|{\text{\textbf{E}}}\|_{2} = 1.
\end{equation}
Here, the matrix \textbf{E} is taken to have random entries uniformly distributed between $[-1,1]$ and $\delta$ is the relative noise level added to the data in the sense that $\|{\text{\textbf{A}}^{\delta} - \text{\textbf{A}}}\|_{2} \leq \delta \|{\text{\textbf{A}}}\|_{2}$. 
It is important to note that the Poisson kernel on the unit disk is harmonic, which implies that it has the mean value property. Thus, for many of the sampling points $z \in \Omega$, we have that $\| \textbf{b}_z\| \ll 1$.  To avoid numerical instabilities, we normalize $\textbf{b}_z$ and define
\[\textbf{b}^{\rm nor}_z = \frac{\textbf{b}_z}{\big \| \textbf{b}_z \big\|}_{2}, \]
such that the quotient is understood to be component-wise division.
We use the `\texttt{norm}' command in \texttt{MATLAB} to compute the discretization of $ \|\partial_{\nu} \mathbb{G}(\cdot , z) \|_{L^{2}(\partial \Omega)}$.
Thus, to compute the indicator associated with Theorem \ref{main-thm}, we solve
\begin{equation} \label{discretized-main-thm}
    \text{\textbf{A}}^{\delta} \textbf{f}_z = \textbf{b}^{\rm nor}_z.
\end{equation}
By Theorem \ref{DtN-com-inv-ran}, the data operator $(\Lambda - \Lambda_0)$ is compact, which implies that \textbf{A} is ill-conditioned. Hence, one needs to employ a regularization technique to find an approximate solution to the discretized equation. In our experiments, we use the Spectral cut-off as the regularization scheme and follow a similar procedure demonstrated in \cite{harris0} where $\textbf{f}^{\alpha}_z$ represents the regularized solution to \eqref{discretized-main-thm} and $\alpha >0$ denotes the regularization parameter. To define the imagining functional, we follow \cite{harris1} to have the following
\[ ( \textbf{f}^{\alpha}_z , \text{\textbf{A}}^{\delta} \textbf{f}_z) = \sum_{j=1}^N \frac{\phi^2 (\sigma_j ; \alpha)}{\sigma_j} \big \rvert (\textbf{u}_j , \textbf{b}^{\rm nor}_z) \big\rvert^2. \]
Here, $\sigma_j$ and ${\bf u}_j$ {\color{black}denote} the singular values and left singular vectors of the matrix ${\textbf{A}}^{\delta}$, respectively. Also, $\phi (\sigma_j ; \alpha)$ corresponds to the filter function defined by the regularization scheme used to solve \eqref{discretized-main-thm}. In our examples, we use the filter function defined by
\begin{equation}
\phi (t , \alpha) =
     \begin{cases}
    
      1, & t^2 \geq \alpha ,\\
      0, & t^2 < \alpha,
    \end{cases}
\end{equation}
which corresponds to the aforementioned Spectral cut-off regularization scheme. With the above expressions, we can recover the unknown region $D$ by defining the imaging functional
 \begin{equation}\label{w-asis}
 W(z) = \big( \textbf{f}^{\alpha}_z , \textbf{A}^{\delta} \textbf{f}^{\alpha}_z \big)^{-1}.
 \end{equation}
Theorem \ref{main-thm} implies that {\color{black}$W(z) > 0$} provided that $z \in D$ as well as $W(z) \approx 0$  provided that $ z \notin D$. {\color{black} The upcoming numerical examples are based on this function to reconstruct the inclusion $D$.} 

We remark that a clear separation of scales justifies the smallness of the interfacial parameters used in the following numerical experiments. Recent work \cite{researchplayground25} shows that, when a surface energy is derived via a dimensional reduction for a thin, homogeneous, three-dimensional isotropic strain-gradient elastic layer, the coefficient $\mu_s$ scales with the layer thickness and the bulk shear modulus. Consequently, $\mu_s$ is expected to be small relative to the bulk shear modulus $\mu$. Likewise, the intrinsic length scale $\ell_s$ encodes the unresolved microstructure of the material (i.e., the inhomogeneous features that are averaged out in the homogeneous continuum description) and is therefore tied to the characteristic spacing between such inhomogeneities \cite{Askes2011}. Depending on the modeling regime, this may correspond to molecular spacing or, in a granular description, to the nearest-neighbor distance between the barycenters of adjacent grains in a heterogeneous solid \cite{RodriguezPlacidiMisraLaValle2026}.\\

\noindent\textbf{Numerical reconstruction of a circular region:}\\
In polar coordinates, we assume $\partial D$ is given by $\rho (\text{cos} ( \theta ), \text{sin} (\theta))$ for some constant $\rho \in (0,1)$. Using separation of variables, we determine that for all $\theta \in [0 , 2\pi)$
\begin{equation}\label{kernel-function}
(\Lambda - \Lambda_0) f(\theta) = \frac{1}{2 \pi} \int_0^{2\pi} K(\theta , \phi) f(\phi) \, \text{d}\phi \quad \text{where} \quad K(\theta , \phi) = \sum_{|n|=1}^{\infty}\kappa_ne^{\text{i}n(\theta - \phi)}.
\end{equation}
See Appendix \ref{appendix-kernel} for details on the calculation of the coefficients $\kappa_n$ for all $n \in \mathbb{Z}$.
In our examples, we approximate the kernel function $K(\theta , \phi)$ given above by truncating the series for $|n| = 0,\dots,100$. We then discretize the truncated integral operator by a 128 equally spaced grid on $[0,2\pi)$ using a collocation method. The dotted red lines represent the boundary of interest $\partial D$. \\

{\color{black} 
\noindent \textbf{Location of anchor point:} In our first experiment, we let $\rho=0.6$ for the reconstructions in Figure \ref{fig:anchorpoints}. Presented are contours plots of the imaging functional $W(z)$ (defined in \eqref{fig:anchorpoints}), where we vary the location of the anchor point $z_0 \in D$. We let $\delta = 0$ which corresponds to $0 \%$ relative random noise added to the data, and the Spectral cut-off regularization parameter $\alpha = 10^{-12}$. We also fix the coefficients $\mu_s = .01, \ell_s^2 = .001$, and $\mu = .2$. }

\begin{figure}[h!]
\centering 
\includegraphics[scale=0.15]{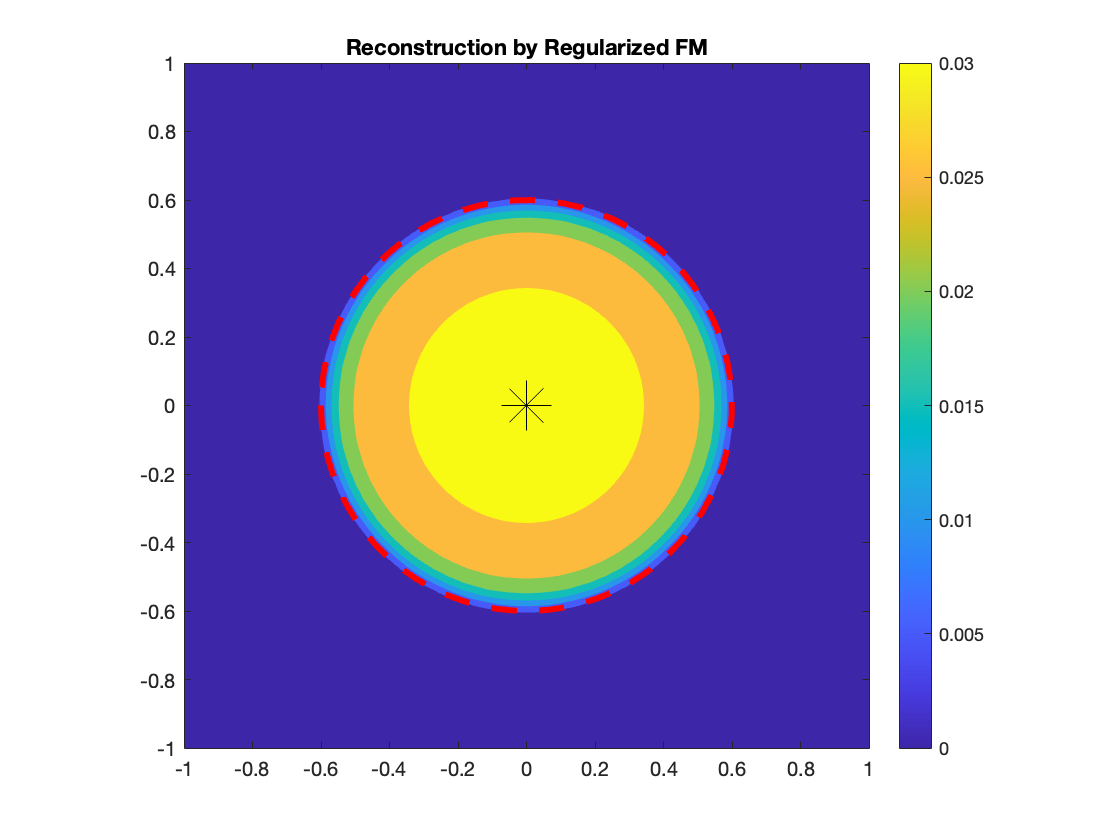}\hspace{-.25in} 
\includegraphics[scale=0.15]{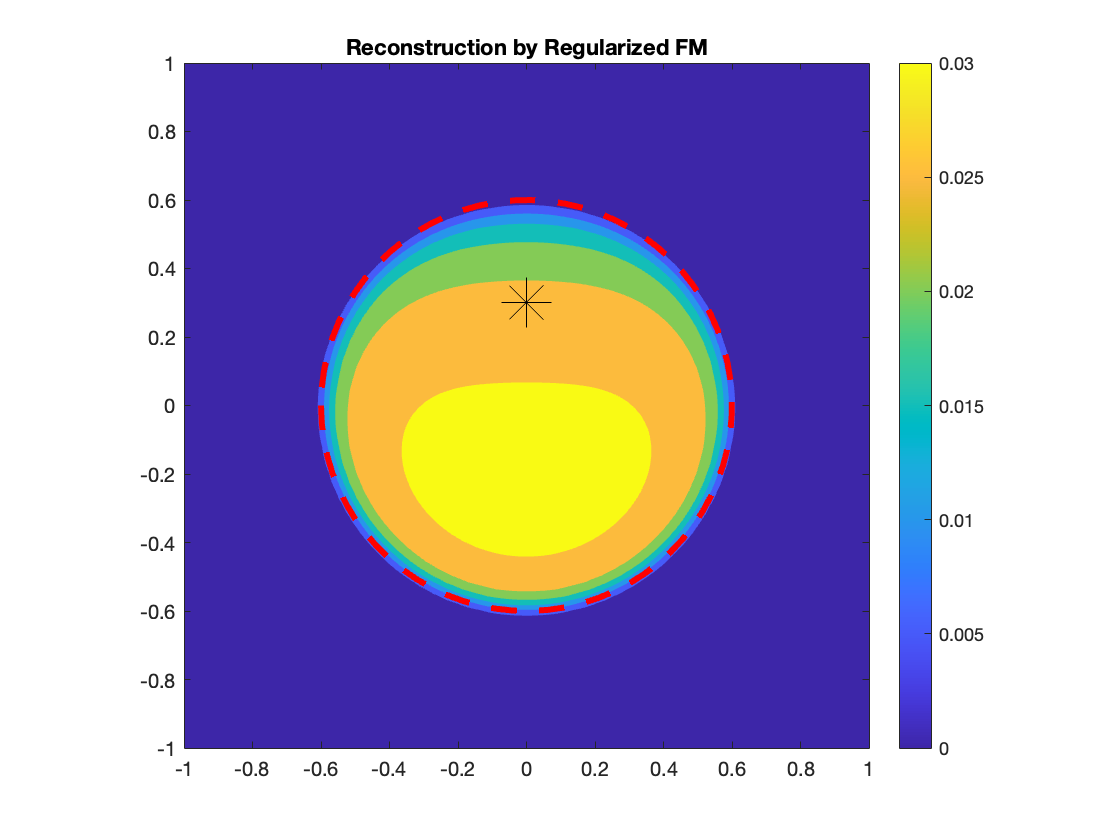} 
\includegraphics[scale=0.15]{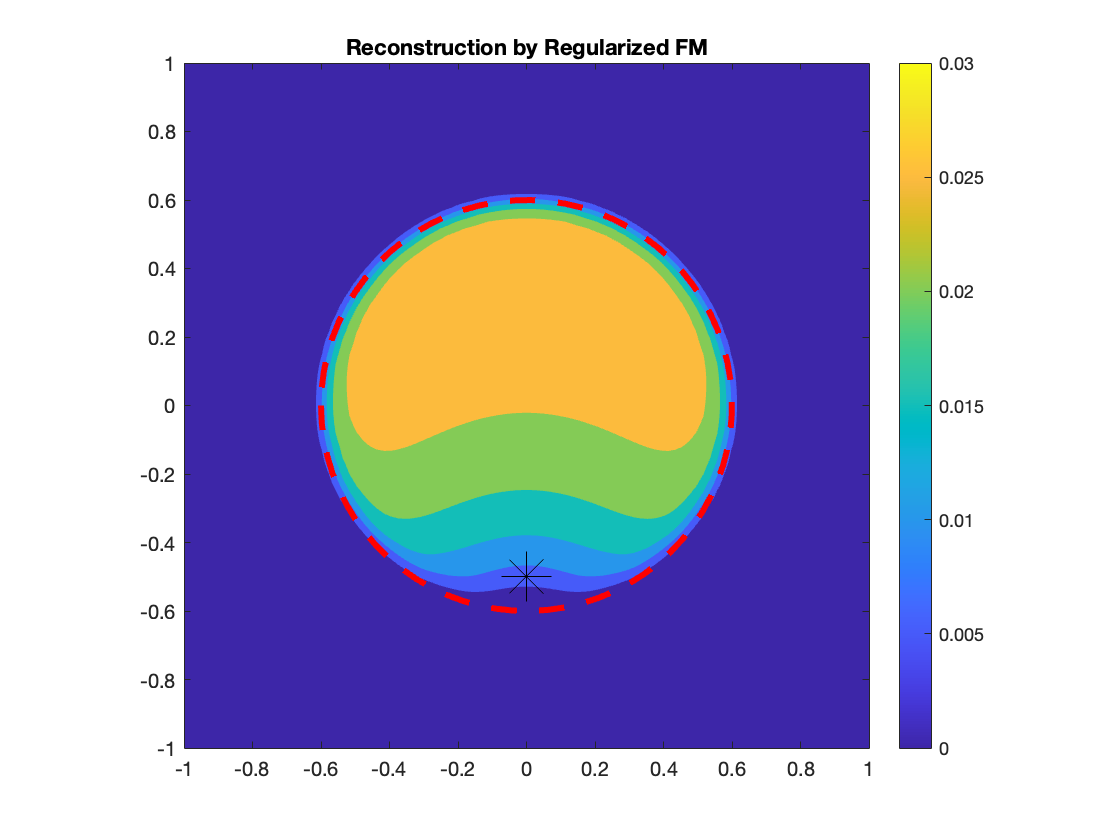}\hspace{-.25in}  
\includegraphics[scale=0.15]{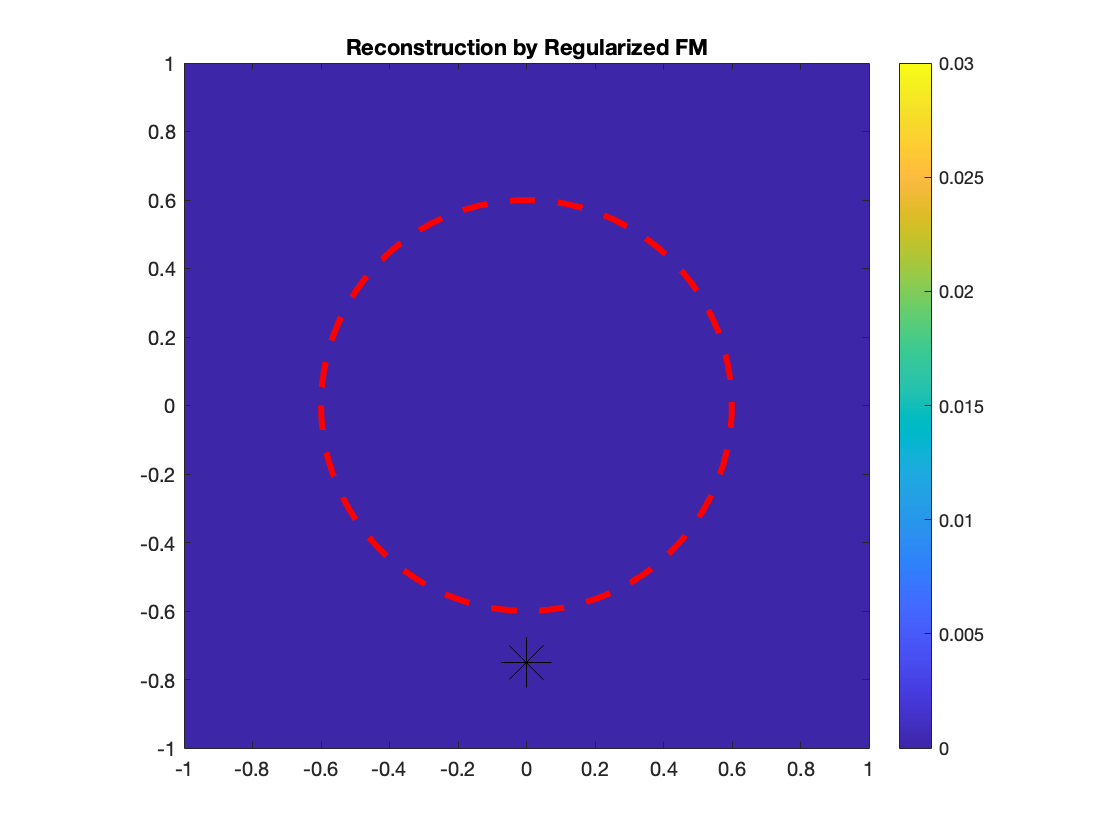}
\caption{We plot $W(z)$, the non-normalized reconstruction, of a circular region with $\rho=0.6$ via the regularized factorization method. Contour plot of $W(z)$ plotted with $\alpha = 10^{-12}$, $\mu = .2$,  $\mu_s = .01, \ell_s^2 = .001$ with $z_0 = (0,0)$ in the top-left, $z_0 = (0,.3)$ on the top-right, $z_0 = (0,-.5)$ on the bottom-left, and $z_0 = (0,-.75)$ on the bottom-right.  In each case, the anchor point is denoted by a black $*$.}
\label{fig:anchorpoints}
\end{figure}

{\color{black}
We note that this experiment is consistent with Theorems \ref{thm:difference_G} and \ref{main-thm} in that for anchor points located in $D$, corresponding to the top-left, top-right, and bottom-left plots of Figure \ref{fig:anchorpoints}, $W(z) > 0$ for $z \in D$, where the maximum value for each case is approximately $0.03$, and $W(z) \approx 0$ for $z \notin D$. Additionally, the bottom-right plot of Figure \ref{fig:anchorpoints} uses anchor point $z_0 = (0,-.75 ) \in \Om \setminus D$, and note that $W(z) \approx 0$ for all $z \in \Om$. Given that $z \notin D$, Theorem \ref{thm:difference_G} implies that for any $z \in \Om$, $z \neq z_0$
\[\partial_{\nu} \mathbb{G}(\cdot , z) \big \rvert_{\partial \Omega} - \partial_{\nu} \mathbb{G}(\cdot , z_0) \big \rvert_{\partial \Omega} \notin {\rm Range} \big( S^{*} \big \rvert_{{\rm span} \{1\}^{\perp}} \big),\] 
and that the hypothesis of Theorem \ref{main-thm} is not satisfied. As expected, this anchor point does not produce a qualitative reconstruction of the inclusion $D$. We refer the reader to Theorems \ref{thm:ps-range1} and \ref{thm:ps-range2} for additional details. In summary, Figure \ref{fig:anchorpoints} illustrates that anchor points within the inclusion lead to a qualitative reconstruction of the full subvolume $D$, whereas those outside, do not.

For all forthcoming examples, we fix the anchor point $z_{0} = 0 \in D$. Since $\partial_\nu \mathbb{G}(\cdot , 0) $ is constant, \eqref{discretized-difference} reduces to subtracting the average. That is,
\[
{\bf b}_z = \partial_\nu \mathbb{G}(\cdot , z) - \fint_{\partial \Omega} \partial_\nu \mathbb{G}(s , z) ds .\]
With this anchor point fixed, we will plot the normalized imaging functional
\begin{equation}
\quad W_{\text{nor}}(z) = \frac{W (z)}{\|{W (z)}\|_{\infty}},
\end{equation}
where $W(z)$ is defined in \eqref{w-asis}, in order to highlight the separation of scales. Consistent with Theorem \ref{main-thm}, $W_{\text{nor}}(z)$ is comparable to an indicator function defined on $D$ in that $W_{\text{nor}}(z) \approx 1$ provided that $z \in D$ and $W_{\text{nor}}(z) \approx 0$ provided that $z \notin D$. All subsequent examples use $W_{\text{nor}}(z)$ to recover $D$.} \\

\noindent
\textbf{Example 1:}
In our first example we let $\rho = 0.7$ for the reconstruction in Figure \ref{num1}. Presented is a contour plot of the imaging functional $W_{\text{nor}}(z)$, where we let $\delta= 0$ which corresponds to $0\%$ relative random noise added to the data. We fix the boundary parameters at $\mu_s = 0.1$ and $\ell_s^2 = 0.001$, as well as the ratio of the shear modulus $\mu = 2$. We vary the regularization parameter for the Spectral cut-off method.

\begin{figure}[h!]
\centering 
\includegraphics[scale=0.15]{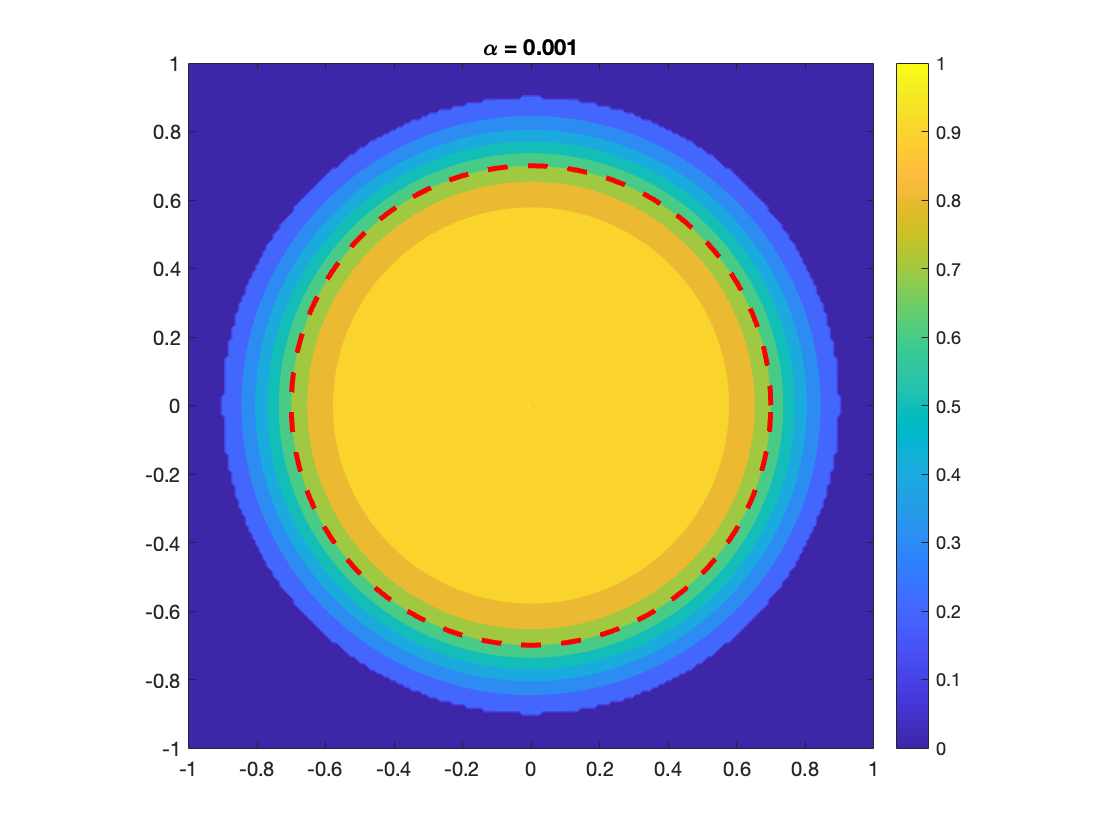}\hspace{-.25in} 
\includegraphics[scale=0.15]{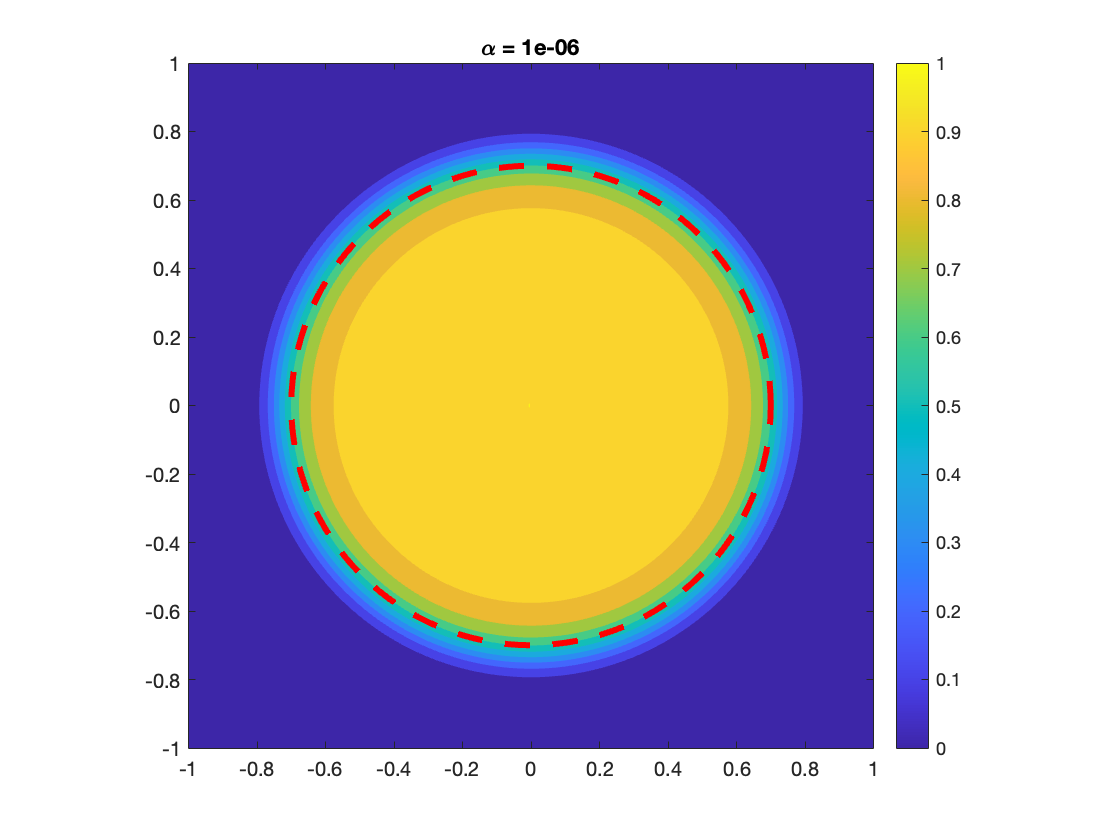} 
\includegraphics[scale=0.15]{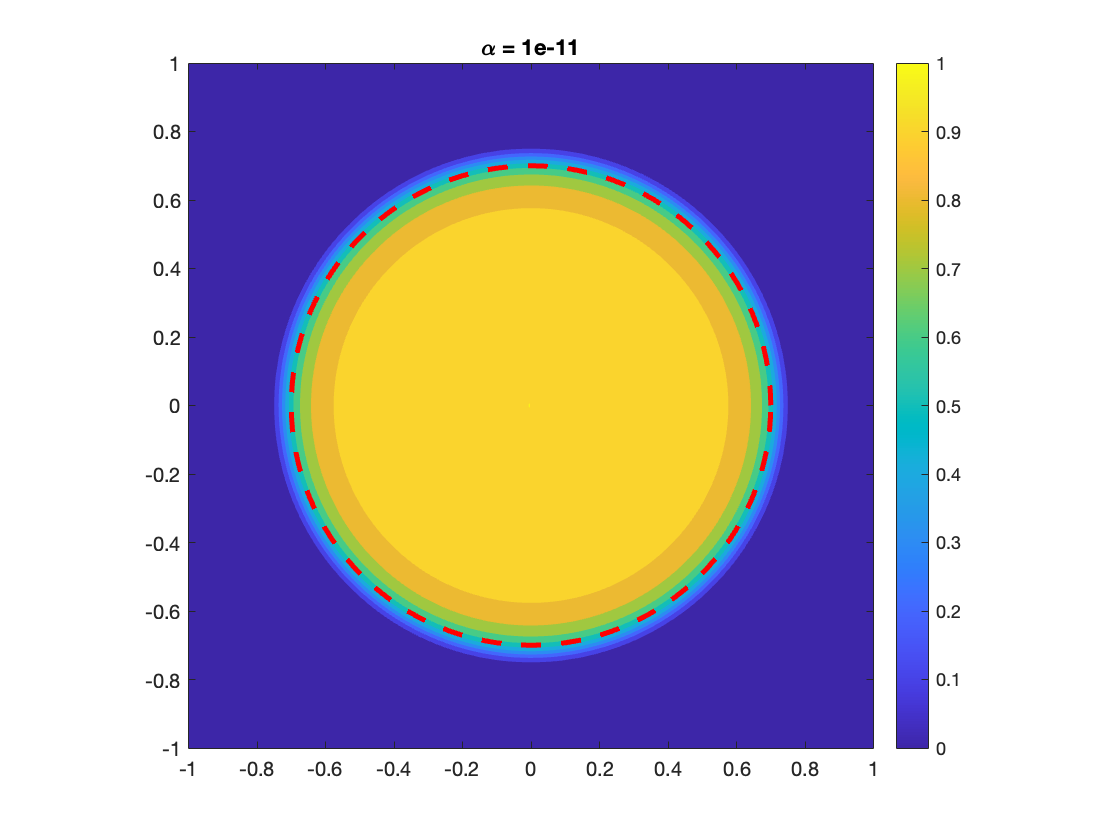}\hspace{-.25in}  
\includegraphics[scale=0.15]{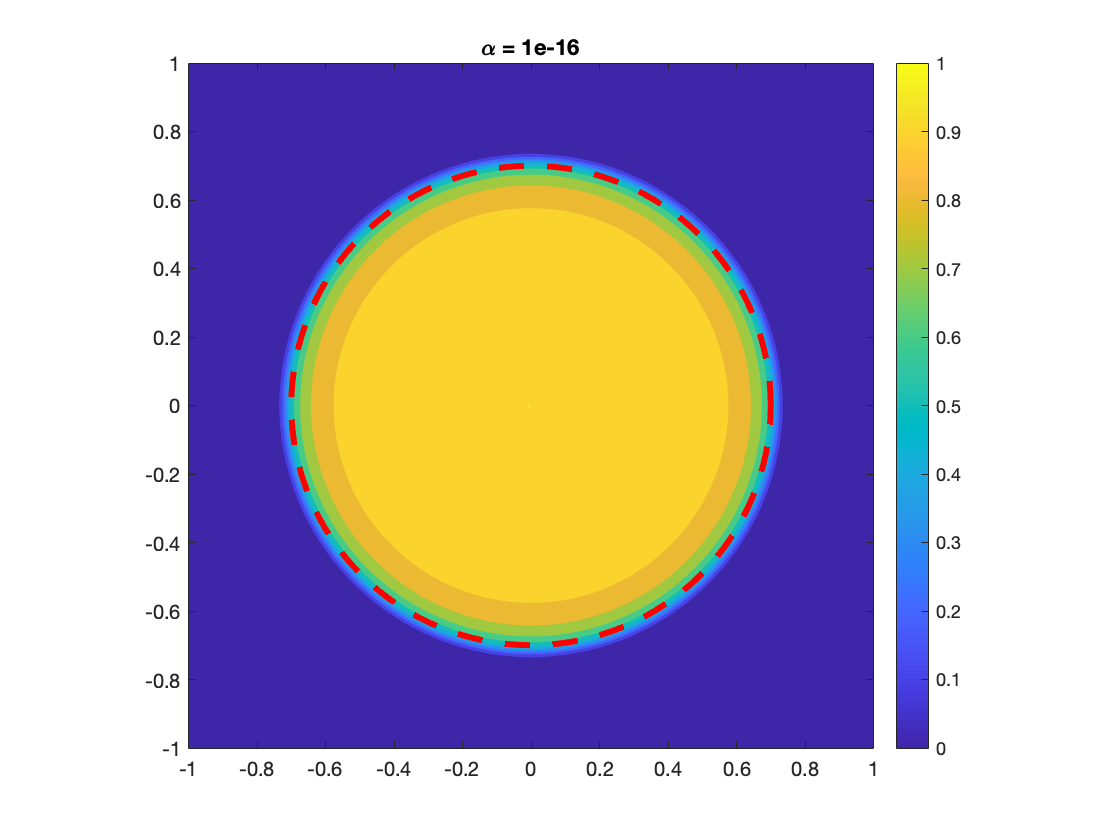}
\caption{Reconstruction of a circular region with $\rho=0.7$ via the regularized factorization method. Contour plot of $W_{\text{nor}}(z)$ plotted with $\alpha = 10^{-3}$ in the top-left, $\alpha = 10^{-6}$ on the top-right, $\alpha = 10^{-11}$ on the bottom-left, and $\alpha = 10^{-16}$ on the bottom-right.}
\label{num1}
\end{figure}
\noindent 
Figure \ref{num1} illustrates the importance of choosing a small value for regularization parameter $\alpha$. For the plot on the top-left, we chose $\alpha = 10^{-3}$, whereas for the plot on the bottom-right, we chose $\alpha = 10^{-16}$. For all plots, the higher values of the imaging functional are concentrated in the center of the region $D$. Across the boundary $\partial D$ (dotted in red), the values begin to decay to zero. Indeed, all plots exhibit the binary behavior of $W_{\text{nor}}(z) \approx 1$ for $z \in D$ and $W_{\text{nor}}(z) \approx 0$ for $z \notin D$. However, the plots on the bottom row illustrate how a much smaller value of $\alpha$ causes $W_{\text{nor}}(z)$ to decay to zero much faster for $z \notin D$. This creates a sharper reconstruction for the boundary $\partial D$. For most of our remaining examples, we heuristically choose very small values of $\alpha$.\\   

\noindent
\textbf{Example 2:}
In our second example we test the effect of including error in our data. We continue to let $\rho = 0.7$ and the shear modulus $\mu =2$ for the reconstruction in Figure \ref{num2}. The boundary parameters $\mu_s = 0.1$ and $\ell_s^2 = 0.001$ are also the same from Example 1. The regularization parameter for the Spectral cut-off method is taken to be $\alpha = 10^{-16}$. Here we vary the relative added noise, $\delta$, to the data as described in \eqref{delta-error}. \\

\begin{figure}[h!]
\centering 
\includegraphics[scale=0.15]{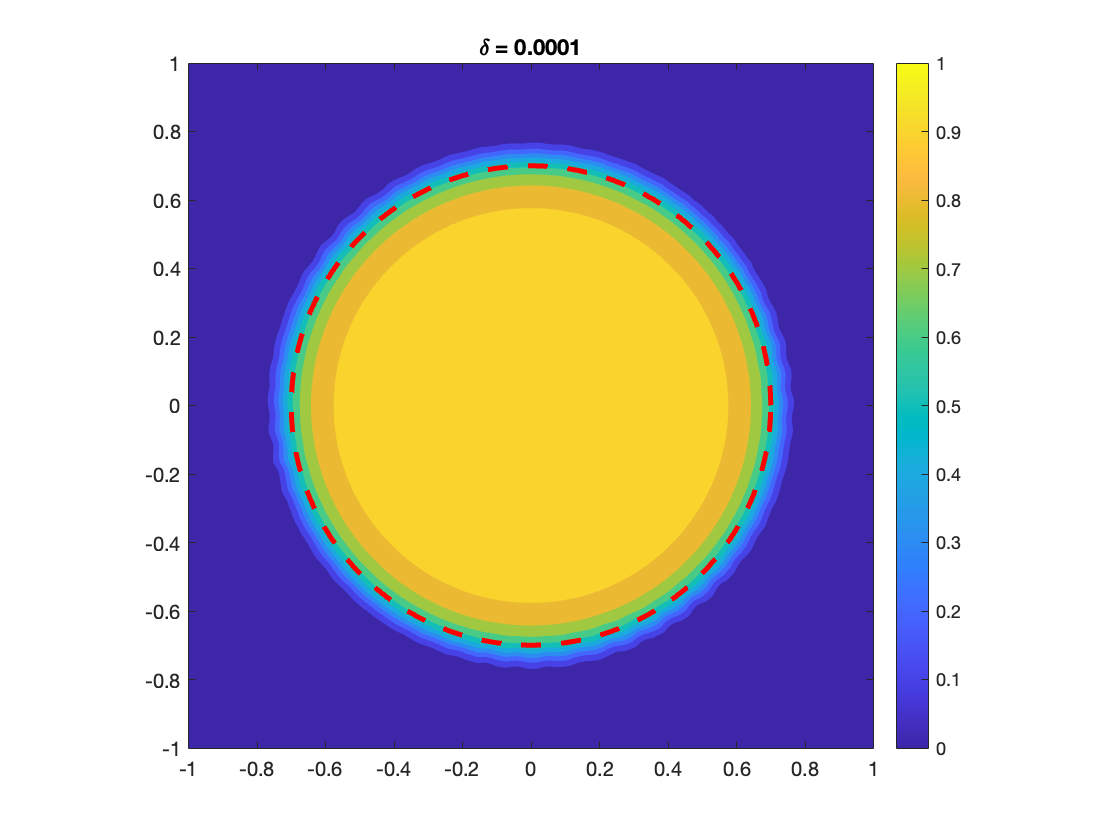} \hspace{-.25in}
\includegraphics[scale=0.15]{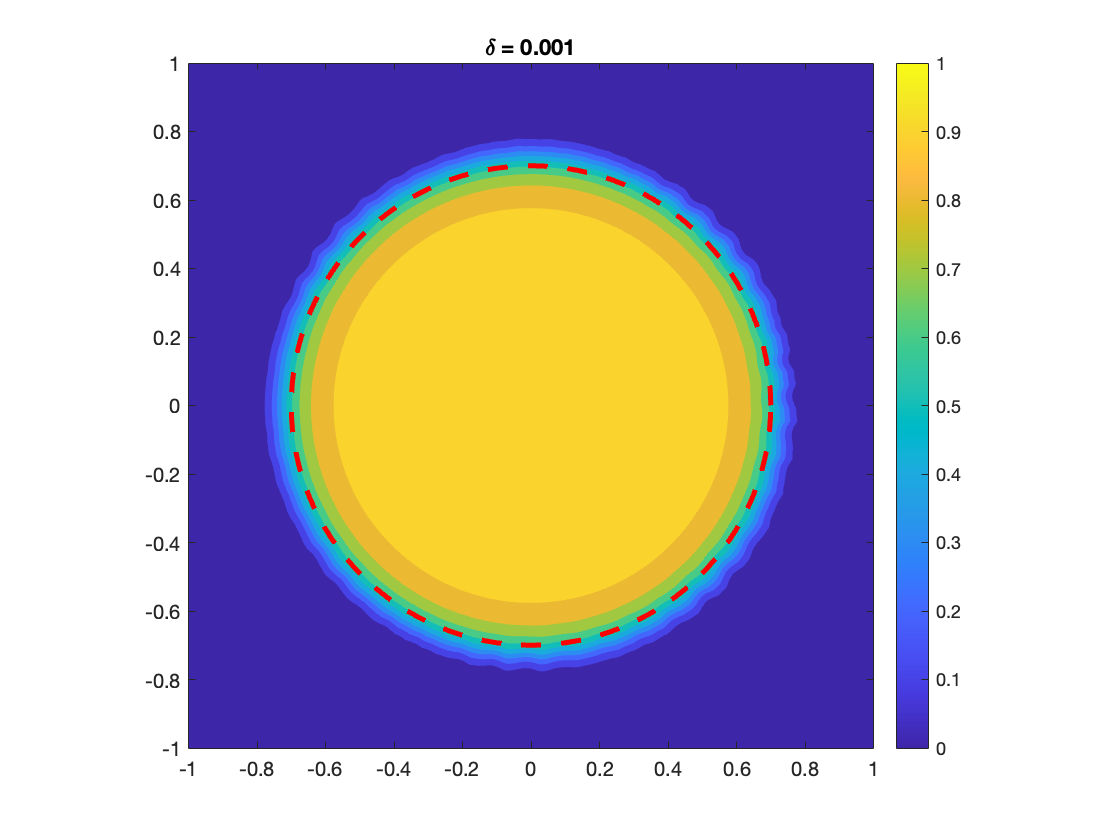}
\includegraphics[scale=0.15]{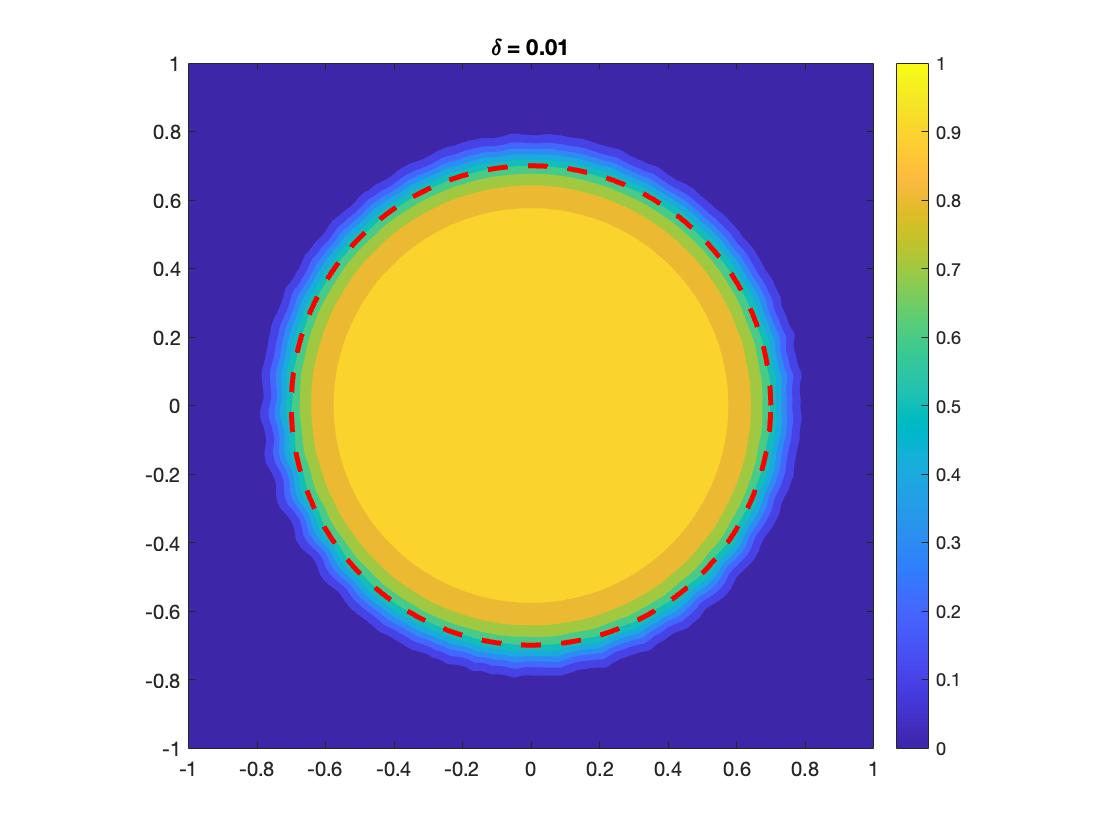}\hspace{-.25in}
\includegraphics[scale=0.15]{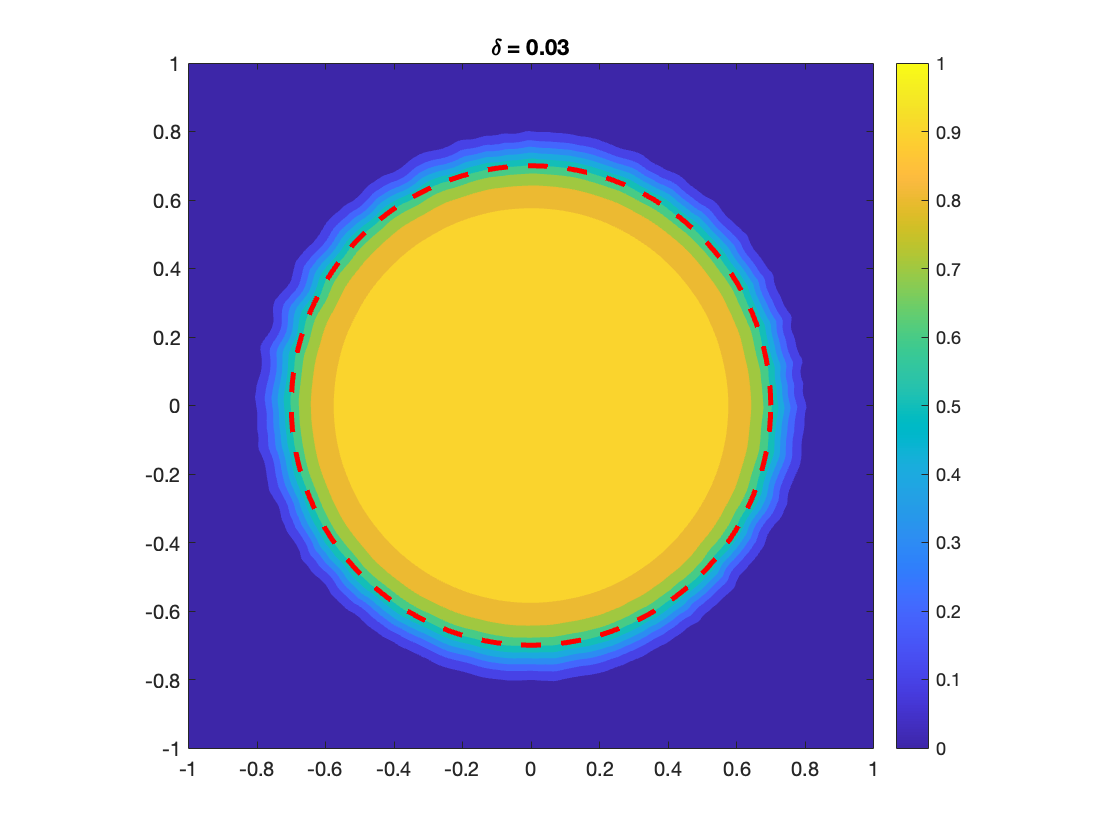}
\caption{Reconstruction of a circular region with $\rho=0.7$ via the regularized factorization method. Contour plot of $W_{\text{nor}}(z)$ plotted with $0.01 \%$ error on the top-left, $0.1 \%$ error on the top-right, $1\%$ error on the bottom-left, and $3 \%$ on the bottom-right.}
\label{num2}
\end{figure}

\noindent
\textbf{Example 3:}
In this example we study the effect of adding relative error $\delta$ to the data in conjunction with the regularization parameter
$\alpha$. Here we let $\rho = 0.55$, and presented in Figure \ref{num3} are contour plots of the imaging functional $W_{\text{nor}}(z)$. We provide a $2 \times 2$ array of plots where the rows correspond to the relative error $\delta = 0.001$ and $\delta = 0.03$, which represent $0.1 \%$ error and $3\%$ error, respectively. The columns in the plot array correspond to the Spectral cut-off regularization parameter $\alpha = 10^{-4}$ and $\alpha = 10^{-8}$. The boundary parameters remain at $\mu_s = 0.1$ and $\ell_s^2 = 0.0001$, where the ratio of the shear modulus $\mu = 2$.\\

\begin{figure}[h!]
\centering 
\includegraphics[scale=0.15]{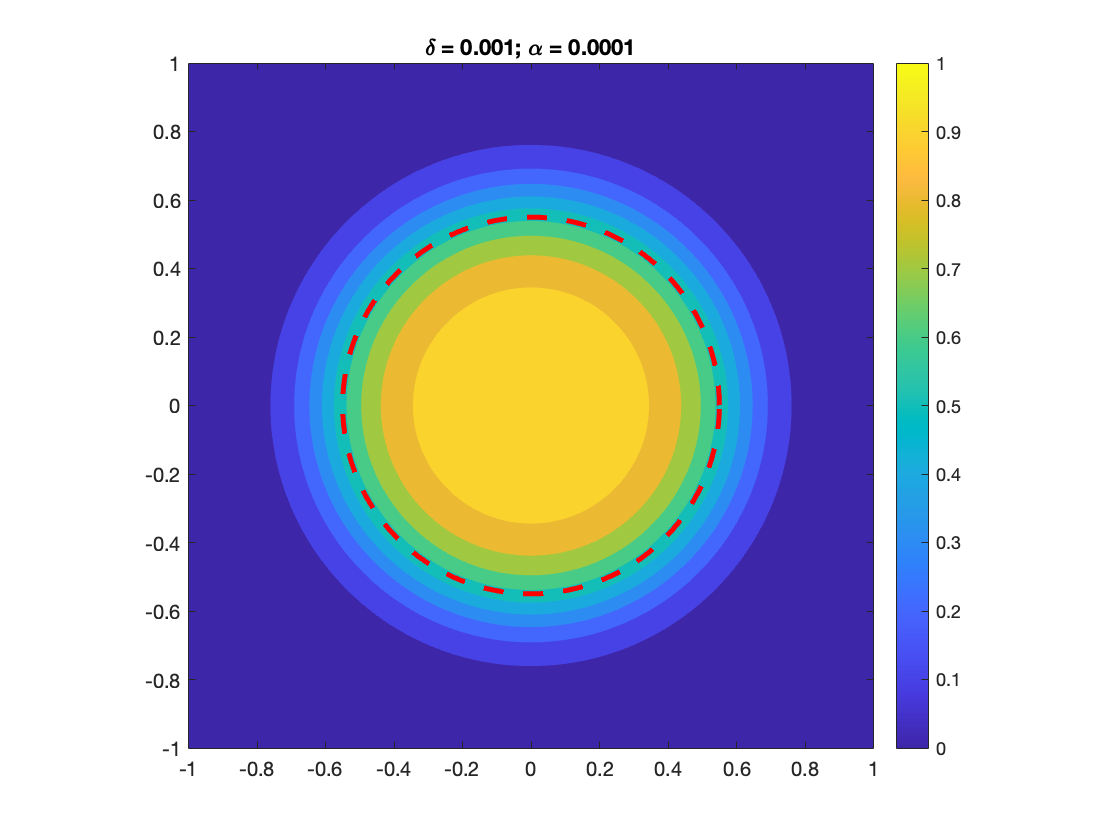}\hspace{-.25in}
\includegraphics[scale=0.15]{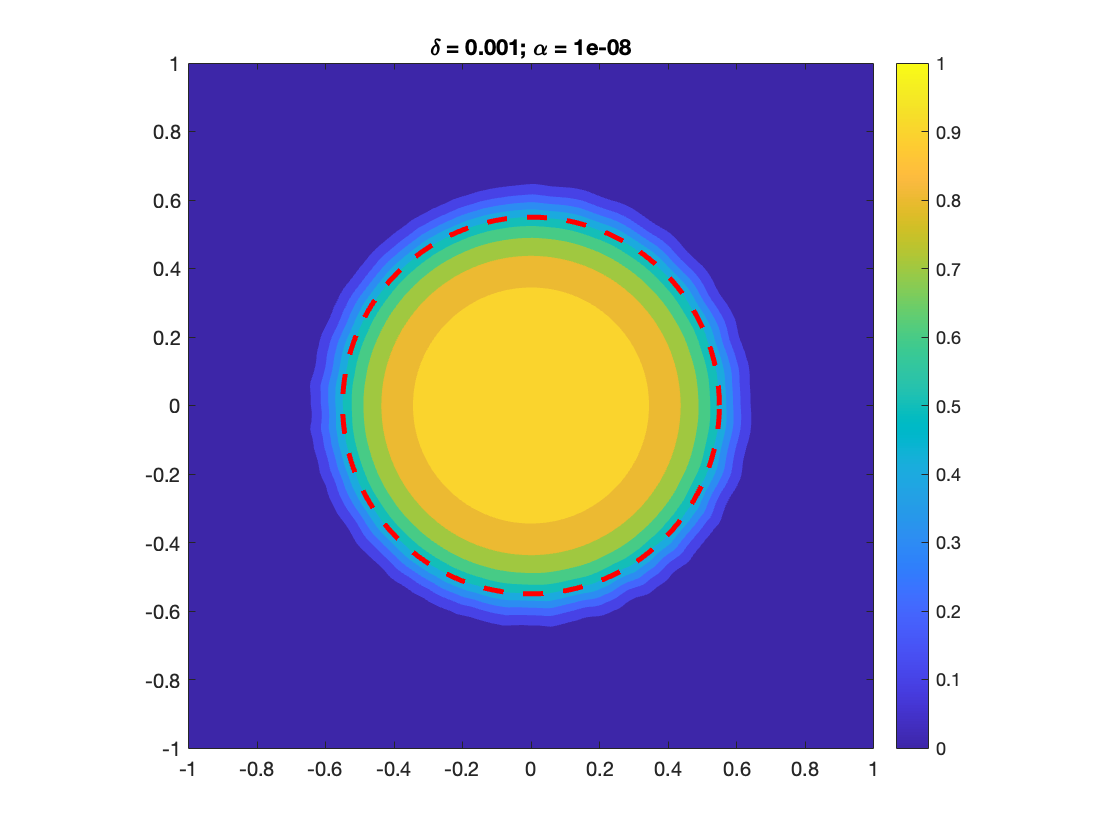}
\includegraphics[scale=0.15]{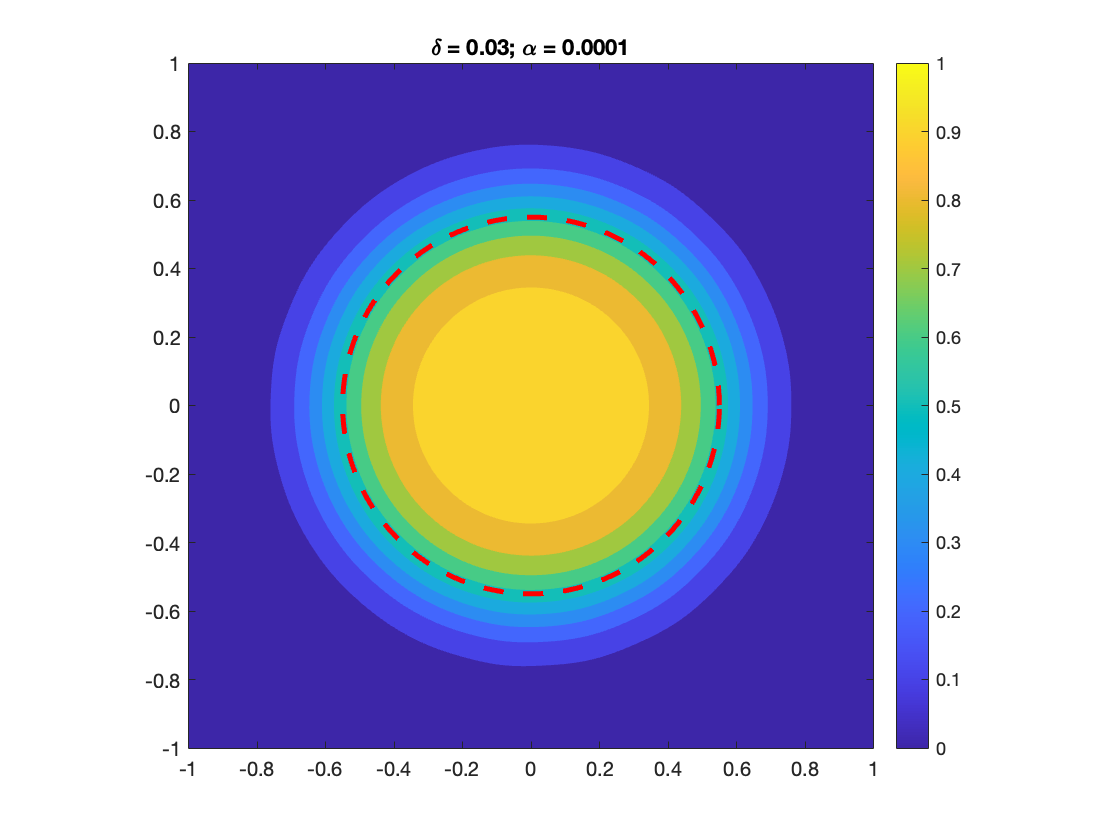} \hspace{-.25in}
\includegraphics[scale=0.15]{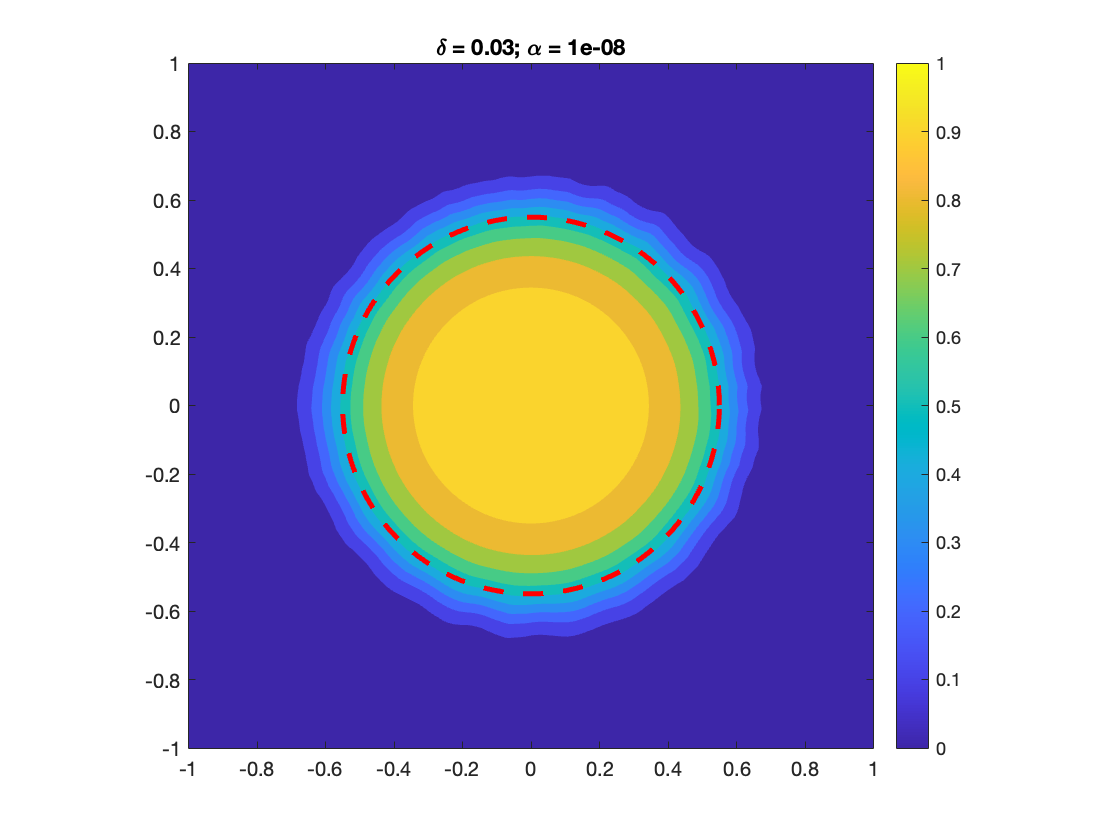}
\caption{Reconstruction of a circular region with $\rho=0.55$. Contour plot of $W_{\text{nor}}(z)$ plotted with $\mu = 0.55$ on the top-left, $\mu = 1$ on the top-right, $\mu = 10$ on the bottom-left, and $\mu = 100$ on the bottom-right.}
\label{num3}
\end{figure}

\noindent
\textbf{Example 4:}
Here we let $\rho = 0.8$, and in Figure \ref{num4} we present contour plots of the imaging functional $W_{\text{nor}}(z)$. We let $\delta= 0.005$ which corresponds to $0.5\%$ relative random noise added to the data. The boundary parameters are smaller from the previous examples where now $\mu_s = 0.01$ and $\ell_s^2 = 0.00001$. Here we vary the ratio of the shear modulus across four values. We begin with the special case where ratio $\mu = 0.8$. This case is meaningful as it is consistent with the hypothesis of Theorem \ref{t-coercive}, which highlights that our reconstruction algorithm is valid for certain values of $\mu$ less than 1. We then proceed to consider $\mu = 1$, which is exactly the case where the shear modulus in the regions $\Omega \setminus D$ and $D$ are equal. We finish this example by varying $\mu$ across two orders of magnitude, specifically $\mu = 10$ and $\mu = 100$. Consistent with our previous examples, the regularization parameter for the Spectral cut-off method is taken to be $\alpha = 10^{-16}$. \\
\begin{figure}[h!]
\centering 
\includegraphics[scale=0.15]{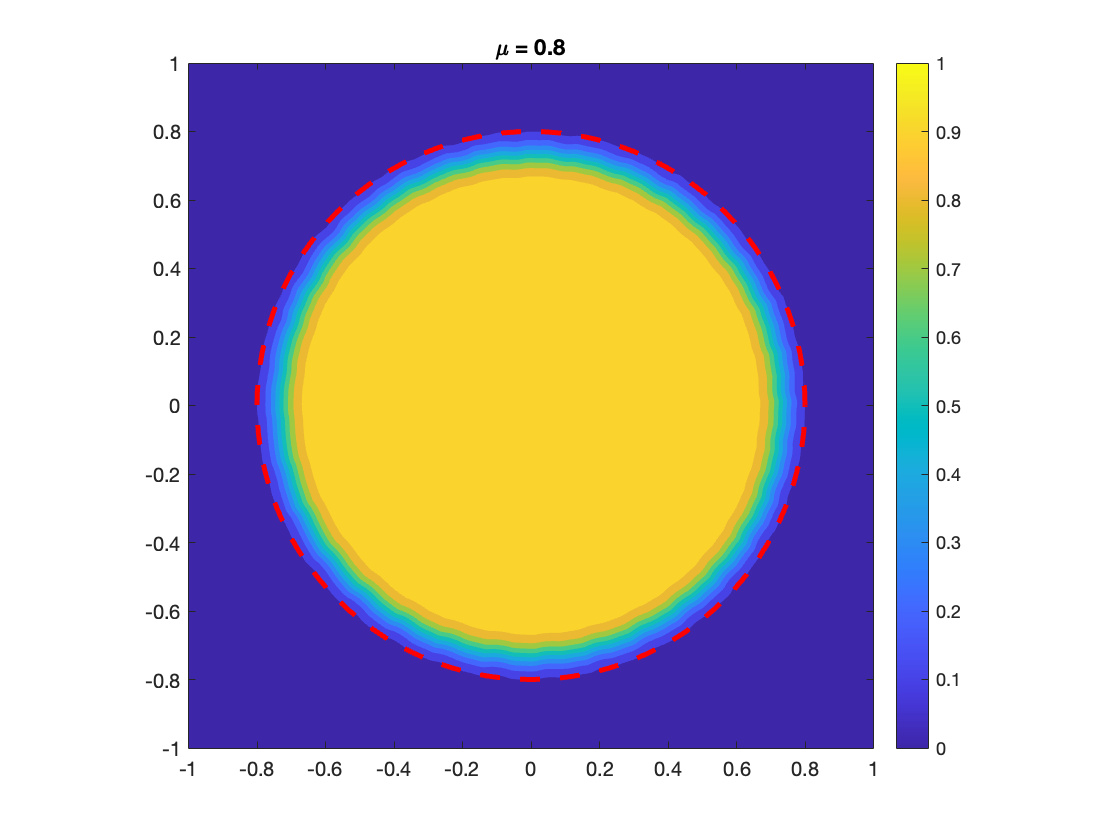}\hspace{-.25in}
\includegraphics[scale=0.15]{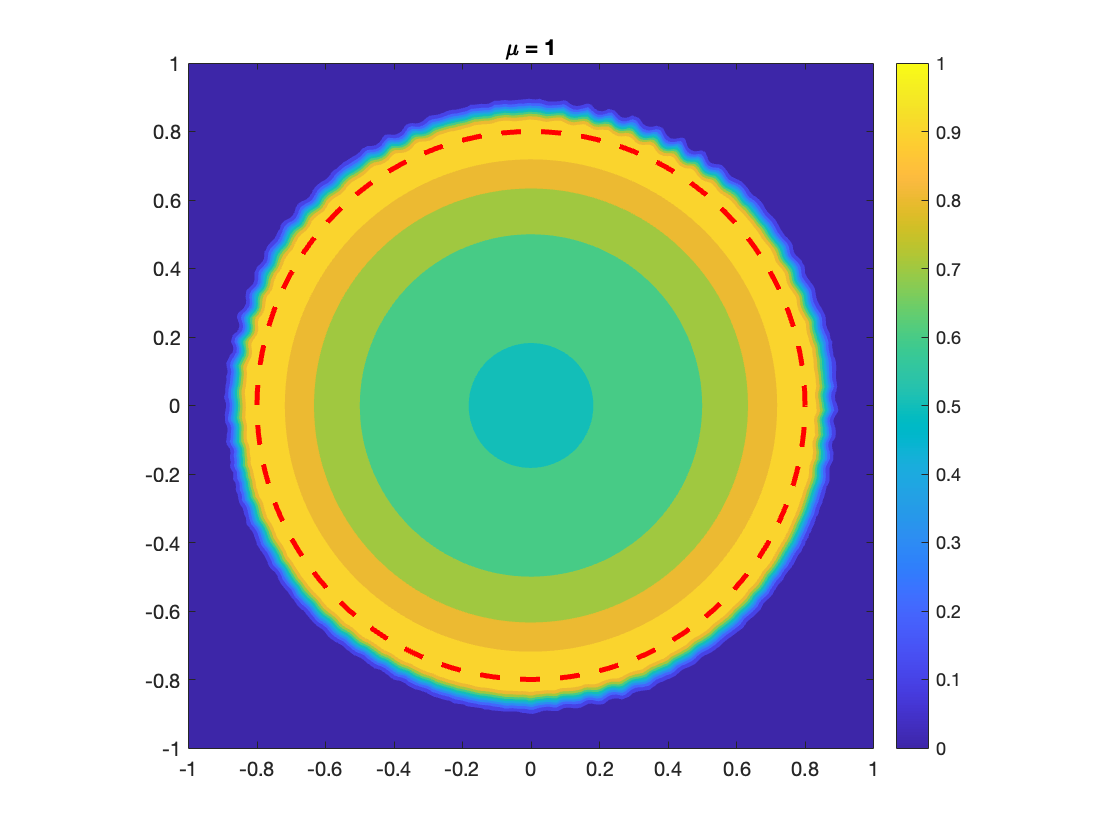}
\includegraphics[scale=0.15]{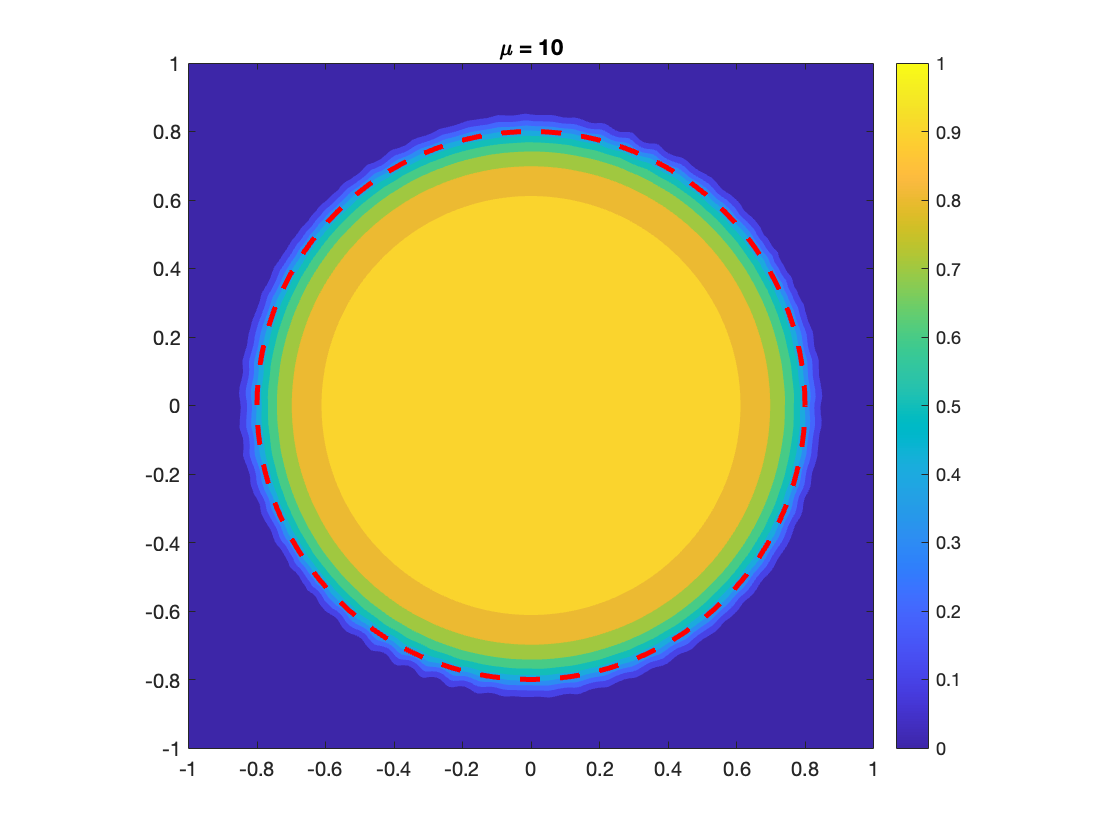} \hspace{-.25in}
\includegraphics[scale=0.15]{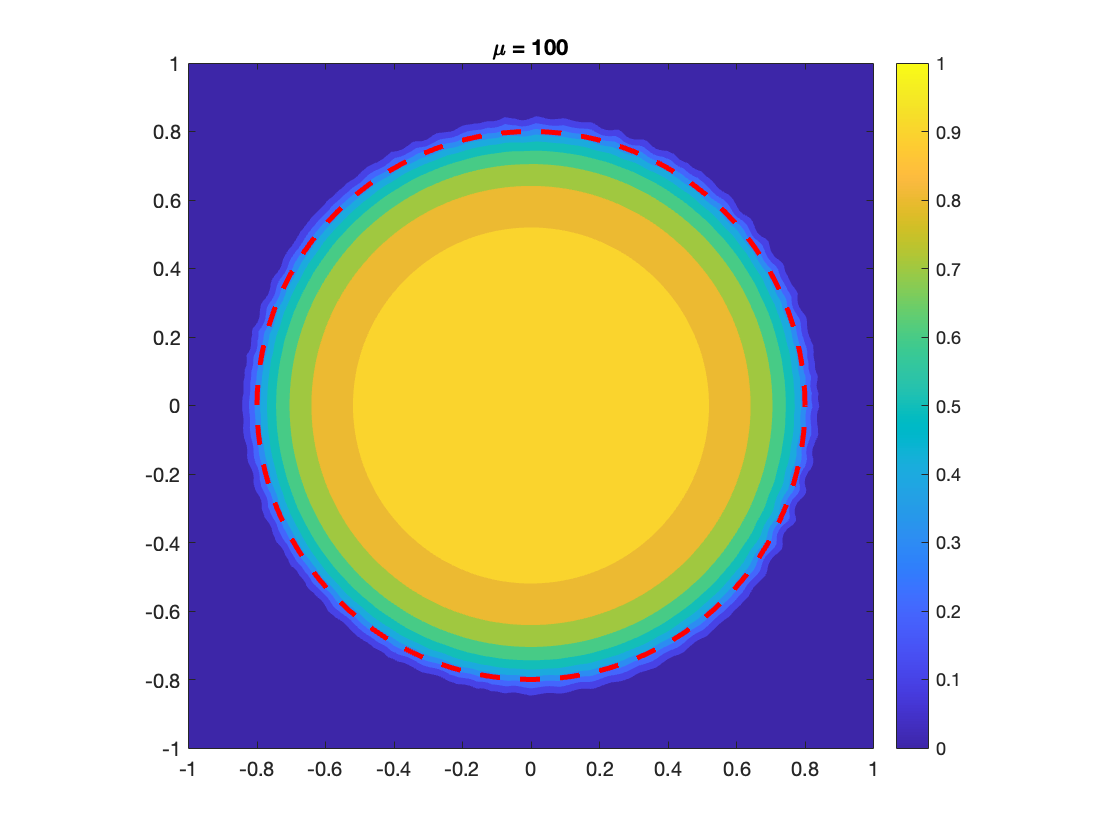}
\caption{Reconstruction of a circular region with $\rho=0.8$. Contour plot of $W_{\text{nor}}(z)$ plotted with $\mu = 0.8$ on the top-left, $\mu = 1$ on the top-right, $\mu = 10$ on the bottom-left, and $\mu = 100$ on the bottom-right.}
\label{num4}
\end{figure}

\noindent
\textbf{Example 5:}
In this example, we analyze the case where the interior region is smaller than in our previous cases. To introduce this case, we do not add error to the data, i.e. $\delta = 0$ across all plots. Presented in Figure \ref{num5} is a contour plot of the imaging functional $W_{\text{nor}}(z)$, where the boundary parameters are fixed at $\mu_s = 0.001$ and $\ell_s^2 = 0.00000001$. Like in our previous example, we vary the interior $\mu$ across key values. The regularization parameter for the Spectral cut-off method is taken to be $\alpha = 10^{-16}$.\\

\begin{figure}[h!]
\centering 
\includegraphics[scale=0.15]{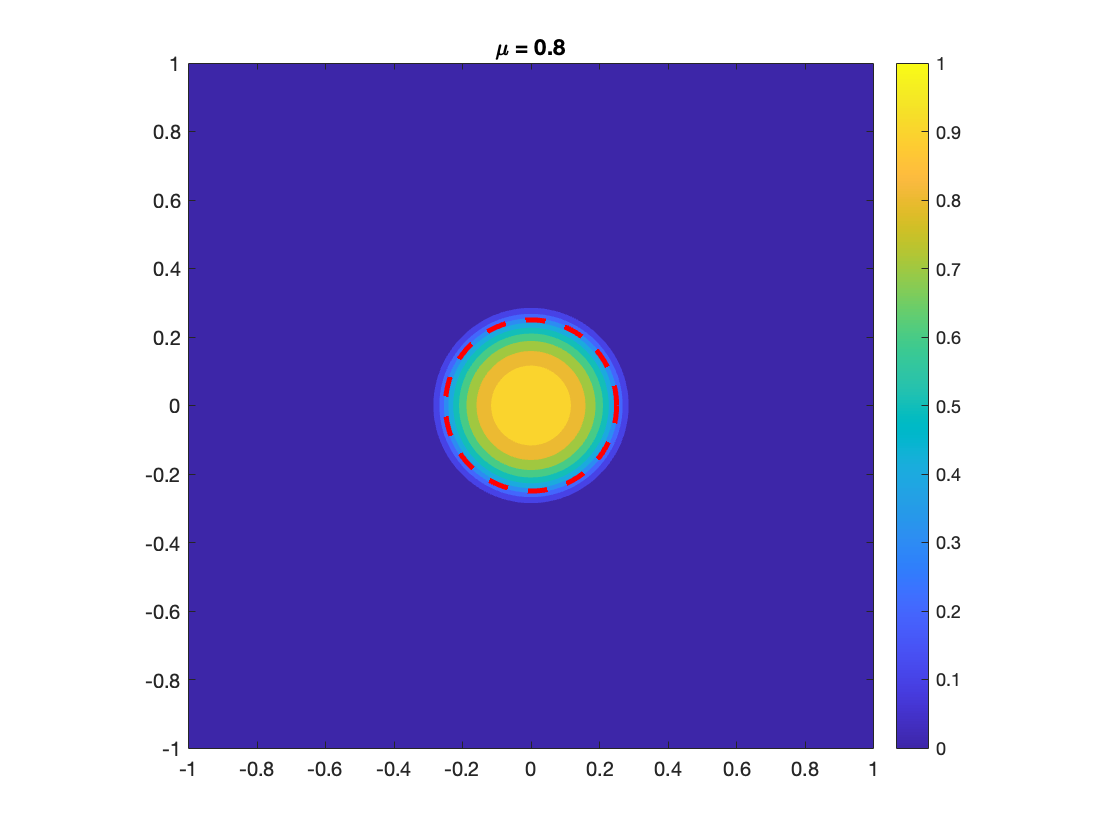}\hspace{-.25in}
\includegraphics[scale=0.15]{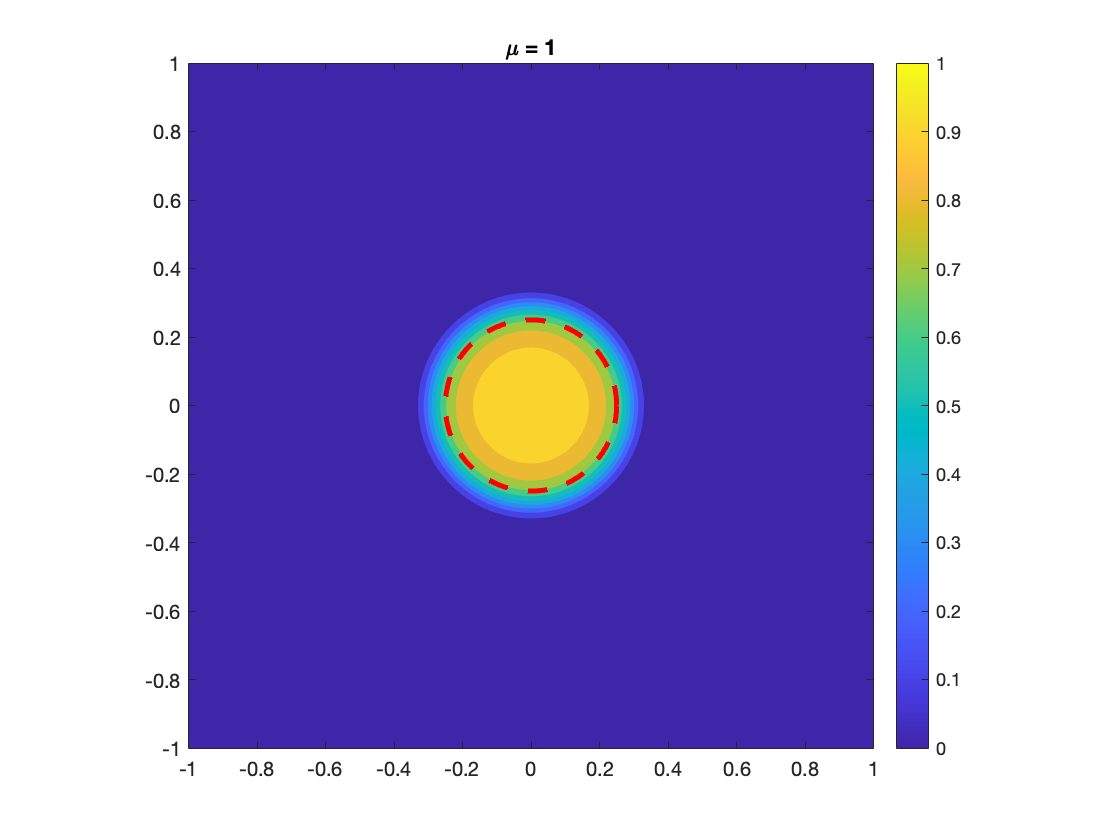}
\includegraphics[scale=0.15]{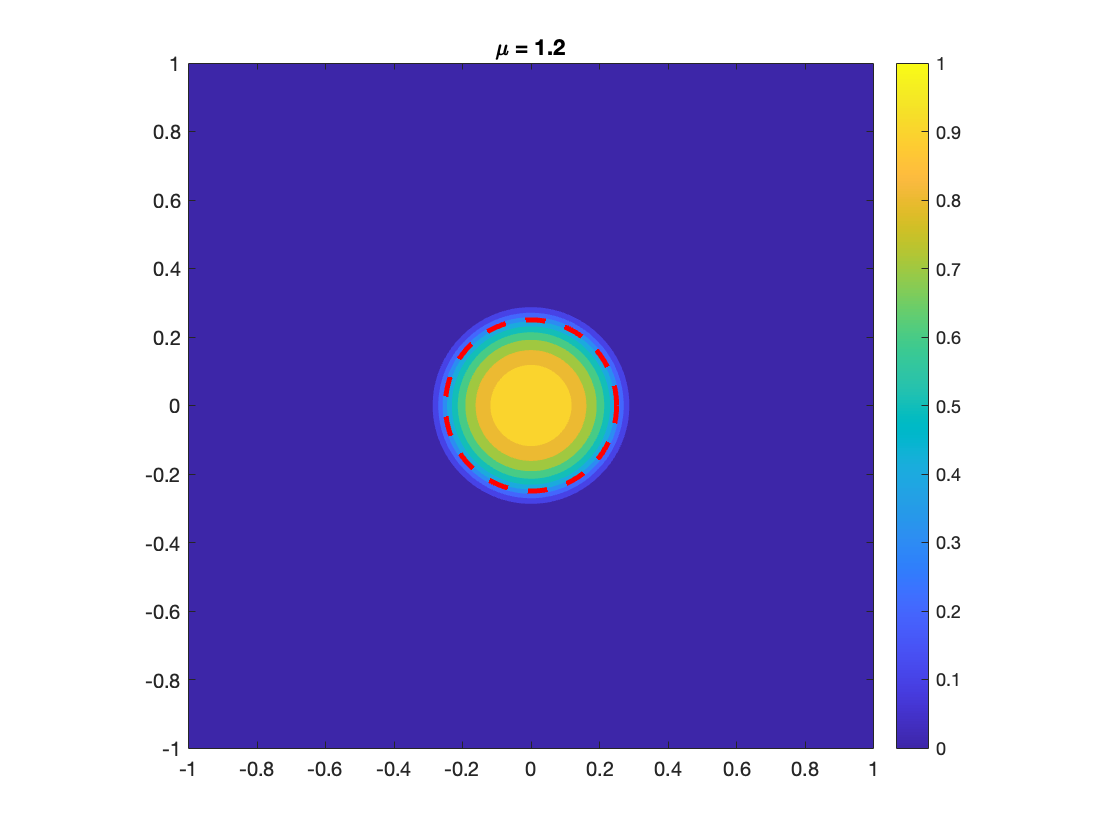} \hspace{-.25in}
\includegraphics[scale=0.15]{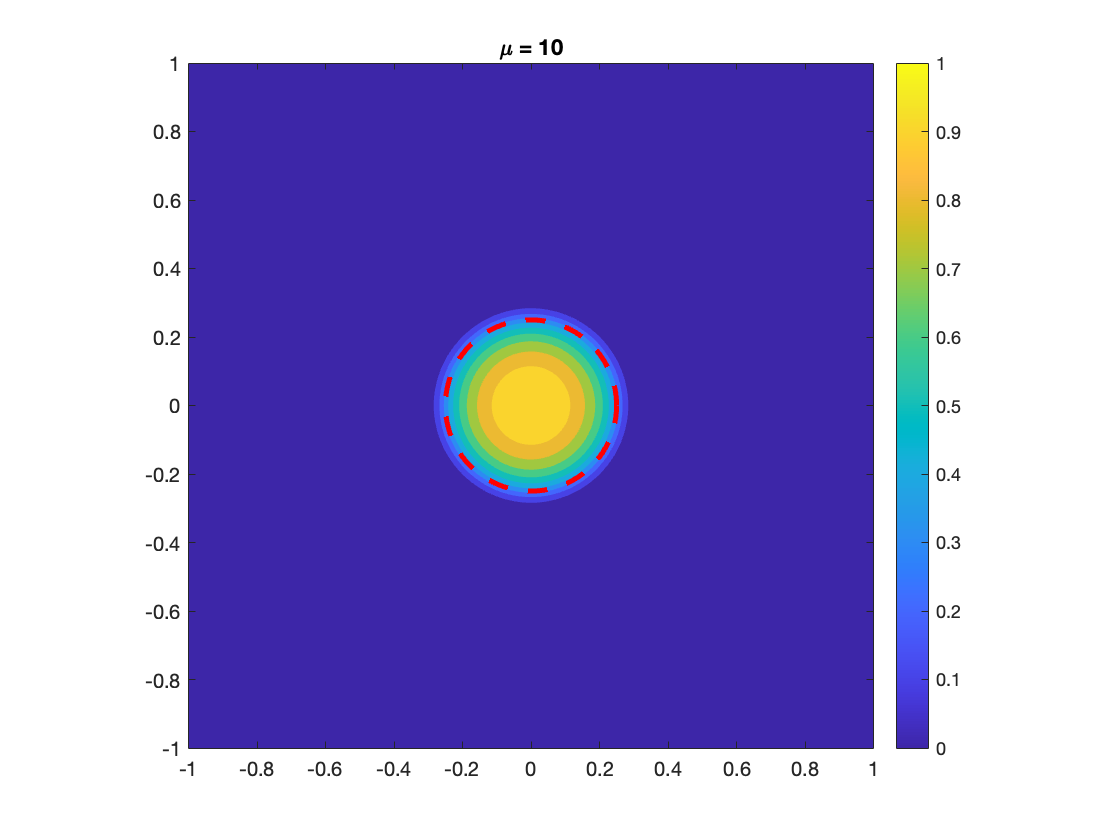}
\caption{Reconstruction of a circular region with $\rho=0.25$. Contour plot of $W_{\text{nor}}(z)$ plotted with $\mu = 0.8$ on the top-left, $\mu = 1$ on the top-right, $\mu = 1.2$ on the bottom-left, and $\mu = 10$ on the bottom-right.}
\label{num5}
\end{figure}

\section{Conclusion}
\label{sec:conclusion}
We have introduced a new form of inverse problem relating to discovering interior inhomogeneities in materials by measuring antiplane shear stress responses only on the outer boundary.  The model we present here is fully linear, though it would be natural to extend the findings here to nonlinear models of elasticity and stress.  The formulation we have derived bears resemblance to inverse problems in impedance and optical tomography, but the material forces at play at the interface between the two materials make the problem distinct and introduce key challenges.  However, we have successfully demonstrated in this linear model that we have uniqueness of measurements based upon material parameters, as well as a mechanism for reconstruction of the interior domain.  While we have focused on developing the machinery here in $2-$dimensional for the sake of clarity, we expect that many of the methods we have presented here can be extended to $3-$dimensional models. Throughout, we have focused on key ideas for implementation, without attempting to optimize certain assumptions, such as the regularity of $\partial \Omega$ and $\partial D$, which can almost certainly be weakened to assuming only Lipschitz.  Numerically, we demonstrated the results we have here on a simple setting of a disc material embedded in an outer disc, which is a convenient framework in which the Dirichlet-to-Neumann maps can be constructed explicitly. {\color{black}We have primarily focused on the single-sampling stage of the reconstruction algorithm, leaving the implementation of the initial two-point search and more general geometries to future work.} It would be interesting to work with various boundary integral solvers as in \cite{smigaj2015solving} to study the problem numerically in a larger set of domains without fixed symmetries.

\appendix
\section{Proofs of main decomposition theorems}\label{app:a}

Here we show the derivation of a spectral decomposition for the positive, compact operator $(\Lambda - \Lambda_0): H^{1/2}(\partial \Omega) \rightarrow H^{-1/2}(\partial \Omega)$ as similarly done in \cite{harris1,kirsch2005factorization,kirsch2007factorization}. The nuance here is that the data operator has a non-trivial kernel. Namely, ${\rm Null} (\Lambda - \Lambda_0) = {\rm span}\{1\}$. By the Riesz Representation Theorem, there exists a bijective isometry $J: H^{-1/2}(\partial \Omega) \rightarrow H^{1/2}(\partial \Omega)$ such that
\[J \ell = f_{\ell} \quad \text{where} \quad (f,f_{\ell})_{H^{1/2}(\partial \Omega)} = \langle f , \ell \rangle_{\partial \Omega} \quad \text{for all} \enspace f \in H^{1/2}(\partial \Omega),\]
where $(\cdot , \cdot)_{H^{1/2}(\partial \Omega)}$ denotes the inner product in $H^{1/2}(\partial \Omega)$. Since $\langle \cdot , \cdot \rangle_{\partial \Omega}$ is a sesquilinear dual-product, we have that $J$ is linear. We now consider the compact operator $J(\Lambda - \Lambda_0):H^{1/2}(\partial \Omega) \rightarrow H^{1/2}(\partial \Omega)$ where
\[(f , J(\Lambda - \Lambda_0)f)_{H^{1/2}(\partial \Omega)} = \langle f , (\Lambda - \Lambda_0)f \rangle_{\partial \Omega} > 0 \quad \text{for all} \enspace f \in H^{1/2}(\partial \Omega) \setminus \big({\rm Null}(\Lambda - \Lambda_0) \cup \{0\}\big)\]
since $(\Lambda - \Lambda_0)$ is non-negative. Thus, the composition $J(\Lambda - \Lambda_0)$ is a self-adjoint, compact operator. Therefore, by the Hilbert-Schmidt Theorem, it holds that there exists an eigenvalue decomposition
\begin{equation}\label{eqn:spectral-JA}
    \{\lambda_n ; f_n\}_{n\in\mathbb{N} \cup \{0\}} \in \mathbb{R}_{\geq 0} \times H^{1/2}(\partial \Omega) \quad \text{where} \quad (J(\Lambda - \Lambda_0))f = \sum_{n \in \mathbb{N} \cup \{0\}} \lambda_n (f,f_n)_{H^{1/2}(\partial \Omega)} f_n,
\end{equation}
where $\{\lambda_n\}_{n\geq 1}$ is a monotonically non-increasing sequence converging to zero and $f_n$ is an orthonormal basis of $H^{1/2}(\partial \Omega)$. Since ${\rm Null} (\Lambda - \Lambda_0) = {\rm span}\{1\}$, we have that $\lambda_0 = 0$ with corresponding eigenvector $f_0 = {\rm span}\{1\}$. Thus, for all $f\in H^{1/2}(\partial \Omega)$
\[ (J(\Lambda - \Lambda_0))f = \sum_{n=1} \lambda_n (f,f_n)_{H^{1/2}(\partial \Omega)} f_n.\]
We now define $\ell_n \in H^{-1/2}(\partial \Omega)$ to be the unique solution to $J \ell_n = f_n$ for any $n\in\mathbb{N} \cup \{0\}$. Furthermore, it holds that
\begin{equation}\label{eqn:basis-duality}
   \langle f_m , \ell_n \rangle_{\partial \Omega} = (f_m , f_n)_{H^{1/2}(\partial \Omega)} = \delta_{mn} \quad \text{for any} \enspace m,n \in \mathbb{N} \cup \{0\}. 
\end{equation}

Thus, $\ell_n$ are the corresponding dual-basis of $H^{-1/2}(\partial \Omega)$. We will show that $\{\ell_n\}_{n \in \mathbb{N} \cup \{0\}}$ is a complete orthonormal set for $H^{-1/2}(\partial \Omega)$ and construct a representation for any $\ell \in H^{-1/2}(\partial \Omega)$ and the operator $(\Lambda - \Lambda_0)$. Note that $H^{-1/2}(\partial \Omega)$ is a Hilbert space whose inner-product is defined by
\[(\ell , \varphi)_{H^{-1/2}(\partial \Omega)} = (f_{\ell} , f_{\varphi})_{H^{1/2}(\partial \Omega)} \quad \text{for all} \enspace \ell, \varphi \in H^{-1/2}(\partial \Omega),\] 
where $J \ell = f_{\ell}$ and $J \varphi = f_{\varphi}$. From this definition, we have that
\[(\ell_m , \ell_n )_{H^{-1/2}(\partial \Omega)} = (f_m , f_n)_{H^{1/2}(\partial \Omega)} = \delta_{mn} \quad \text{for all} \enspace m,n \in \mathbb{N} \cup \{ 0\}.\]
Thus, $\{\ell_n \}_{n\in\mathbb{N} \cup \{0\}}$ is an orthonormal set in $H^{-1/2}(\partial \Omega)$. To show it is complete, assume that $\ell \in H^{-1/2}(\partial \Omega)$ is orthogonal to the set $\{\ell_n\}_{n \in \mathbb{N} \cup \{0\}}$. So, for any $n \in \mathbb{N} \cup \{0\}$, it holds that
\[0=(\ell , \ell_n)_{H^{-1/2}(\partial \Omega)} = (f_{\ell} , f_n)_{H^{1/2}(\partial \Omega)},\]
which implies that $f_{\ell} = 0$. Furthermore, since $J$ is an isometry, $\ell = 0$, proving that $\{\ell_n\}_{n \in \mathbb{N} \cup \{0\}}$ is complete. Moreover, we now conclude that the sequence $\{\ell_n\}_{n \in \mathbb{N} \cup \{0\}}$ indeed forms an orthonormal basis of $H^{-1/2}(\partial \Omega)$. We now have the following representation for any $\ell \in H^{-1/2}(\partial \Omega)$
\[\ell = \sum_{n\in\mathbb{N} \cup \{0\}} (\ell , \ell_n)_{H^{-1/2}(\partial \Omega)} \quad \text{where} \enspace \langle f_n , \ell \rangle_{\partial \Omega} = (\ell , \ell_n)_{H^{-1/2}(\partial \Omega)}.\]
Thus, we have that
\begin{equation} \label{eqn:spectral-ell}
    \ell = \sum_{n\in\mathbb{N} \cup \{0\}} \langle f_n , \ell \rangle_{\partial \Omega} \,\ell_n \quad \text{for all} \enspace \ell \in H^{-1/2}(\partial \Omega).
\end{equation}
From the injectivity of the isometry $J$, the operator $(\Lambda - \Lambda_0)$ has the spectral decomposition
\begin{equation}\label{eqn:spectral-A}
    (\Lambda - \Lambda_0)f = \sum_{n\in\mathbb{N} \cup \{0\}} \lambda_n \langle f , \ell_n \rangle_{\partial \Omega} \, \ell_n \quad \text{for all} \enspace f \in H^{1/2}(\partial \Omega).
\end{equation}
In the following theorem, we characterize the range of an operator by the spectral decomposition of the compact operator $(\Lambda - \Lambda_0)$ given by \eqref{eqn:spectral-A}.
\begin{theorem}\label{thm:ps-range1}
    Let $P$ be the orthogonal projection onto ${\rm span}\{1\}^{\perp}$ and $(\Lambda - \Lambda_0) : H^{1/2}(\partial \Omega) \rightarrow H^{-1/2}(\partial \Omega)$ be a non-negative operator with factorization $(\Lambda - \Lambda_0) = (PS)^{*}TPS$ such that $S:H^{1/2}(\partial \Omega) \rightarrow H^{7/2}(\partial \Omega)$ and $T:H^{7/2}(\partial D) \rightarrow H^{-7/2}(\partial D)$ are bounded linear operators. Assume that $S$ is compact and injective, as well as $T$ coercive on ${\rm Range}(PS)$. Then we have that 
    \[ \ell \in {\rm Range}((PS)^{*}) \quad \text{if and only if} \quad \sum_{n=1} \frac{1}{\lambda_n} \big | \langle f_n, \ell \rangle_{\partial \Omega} \big |^2 < \infty,\]
    where $\{\lambda_n ; f_n \}_{n\in \mathbb{N} \cup \{0\}} \in \mathbb{R}_{\geq 0} \times H^{1/2}(\partial \Omega)$ are given by the spectral decomposition \eqref{eqn:spectral-A} of $(\Lambda - \Lambda_0)$.
\end{theorem}
\begin{proof}
    By the assumptions on $T$ and $S$, $(\Lambda - \Lambda_0)$ is a positive, compact operator. Thus, $(\Lambda - \Lambda_0)$ has the spectral decomposition as described in \eqref{eqn:spectral-A} where $\{f_n\}_{n\in\mathbb{N}\cup \{0\}}$ is an orthonormal basis of $H^{1/2}(\partial \Omega)$. We proceed by defining the bounded, linear operator $Q^{*}:L^{2}(\partial\Omega) \rightarrow H^{1/2}(\partial \Omega)$ such that 
    \[Q^{*}\phi = \sum_{n=1} \sqrt{\lambda_n} (\phi , \phi_n)_{L^{2}(\partial \Omega)} \ell_n \quad \text{for all}\enspace \phi \in L^{2}(\partial \Omega).\]
    Note that ${\rm Null}(Q^*) = {\rm span} \{1\}$. We define the adjoint $Q:H^{1/2}(\partial \Omega) \rightarrow L^{2}(\partial \Omega)$ by the following equality
    \[(Qf , \phi)_{L^{2}(\partial \Omega)} = \langle f , Q^{*}\phi \rangle_{\partial \Omega} \quad \text{for all} \enspace \phi \in L^{2}(\partial \Omega) \enspace \text{and} \enspace f \in H^{1/2}(\partial \Omega).\]
    Thus, for any $n \in \mathbb{N} \cup \{0\}$, we have that
    \[(Qf , \phi_n)_{L^{2}(\partial \Omega)} = \langle f , Q^{*}\phi_n \rangle_{\partial \Omega} = \sqrt{\lambda_n} \langle f , \ell_n \rangle_{\partial \Omega}, \]
    where we used the definition of $Q^{*}$ and the fact that $\{\phi_n\}_{n\in\mathbb{N}\cup\{0\}}$ is an orthonormal set in $L^{2}(\partial \Omega)$. It follows that 
    \[Qf = \sum_{n=1} \sqrt{\lambda_n} \langle f , \ell_n \rangle_{\partial \Omega} \, \phi_n \quad \text{for all} \enspace f \in H^{1/2}(\partial \Omega).\]
    Therefore, by \eqref{eqn:basis-duality}, it holds that $Q^{*}\phi_n = \sqrt{\lambda_n} \ell_n$ and $Qf_n = \sqrt{\lambda_n} \phi_n$ for any $n \in \mathbb{N} \cup \{0\}$. Thus, $Q^{*}Q f_n = \lambda_n \ell_n$ for all $n \in \mathbb{N} \cup \{0\}$. Notice that by \eqref{eqn:basis-duality} and \eqref{eqn:spectral-A}, we have that 
    \[(\Lambda - \Lambda_0)f_n = \lambda_n \ell_n = Q^{*}Qf_n \quad \text{for all} \enspace n \in \mathbb{N} \cup \{0\},\]
    which implies that $(\Lambda - \Lambda_0) = Q^{*}Q$ since they agree on a basis. Therefore, by Theorem \ref{range-adjs-q-s}, we have that ${\rm Range}((PS)^{*}) = {\rm Range}(Q^{*})$.
    
    Now that we have established this equality, suppose that $\ell \in {\rm Range} ((PS)^{*}) = {\rm Range} (Q^{*})$. This is equivalent to the existence of $\phi \in L^{2}(\partial \Omega)$ such that $Q^{*}\phi = \ell$. By the definition of $Q^{*}$ and \eqref{eqn:spectral-ell}, we have that 
    \[\langle f_n , \ell \rangle_{\partial \Omega} = \sqrt{\lambda_n} (\phi , \phi_n)_{L^{2}(\partial \Omega)} \quad \text{for all} \enspace n \in \mathbb{N} \cup \{0\}. \]
    Since $\{ \phi_n\}_{n\in\mathbb{N} \cup \{0\}}$ is an orthonormal basis of $L^{2}(\partial \Omega)$, it holds that $Q^{*}\phi = \ell$ is equivalent to
    \[\sum_{n\geq1} \frac{1}{\lambda_n} \rvert \langle f_n , \ell \rangle_{\partial \Omega} \rvert^2 \leq \| \phi\|^{2}_{L^{2}(\partial \Omega)} < \infty, \]
    proving the claim.
\end{proof}
We remark that if $\ell \in {\rm Range}(\Lambda - \Lambda_0)$, then we have that $(\Lambda - \Lambda_0)f=\ell$ for some $f \in H^{1/2}(\partial \Omega)$. Thus, by \eqref{eqn:spectral-ell} and \eqref{eqn:spectral-A}, we obtain that
\[f  = \langle f,\ell_0 \rangle_{\partial \Omega} \, f_0+ \sum_{n\geq1} \frac{1}{\lambda_n} \langle f_n, \ell \rangle_{\partial \Omega} \, f_n.\]
Since $(\Lambda - \Lambda_0)$ is positive and compact, we have that $\lambda_n >0$ for all $n \geq 1$ such that $\{\lambda_n\}_{n \in \mathbb{N}} \rightarrow 0$ as $n\rightarrow \infty$. Therefore, we define $f_{\alpha}$ to be the regularized solution of $(\Lambda - \Lambda_0)f = \ell$, which is given to be
\begin{equation}\label{eqn:regularized-xn}
f_{\alpha} = \langle f,\ell_0 \rangle_{\partial \Omega} \, f_0 + \sum_{n\geq1} \frac{h_{\alpha}(\lambda_n)}{\lambda_n} \langle f_n, \ell \rangle_{\partial \Omega} \, f_n.
\end{equation}
Here, $h_{\alpha}$ denotes the filter function associated with a specific regularization scheme. Also, we used the fact that $\{\lambda_n ; f_n; \ell_n\}_{n \in \mathbb{N}\cup\{0\}}\in\mathbb{R}_{\geq 0} \times H^{1/2}(\partial \Omega) \times H^{-1/2}(\partial \Omega)$ is the singular system for $(\Lambda - \Lambda_0)$. For all $\alpha >0$, the filter functions $h_{\alpha}:(0,\lambda_1] \rightarrow \mathbb{R}_{\geq 0}$ form a family of functions such that for all $0 < t \leq \lambda_1$
\[\lim_{\alpha \rightarrow 0} h_{\alpha}(t) = 1 \quad \text{and} \quad h_{\alpha} (t) \leq C_{\text{Reg}} \enspace \text{for all} \enspace \alpha >0\]
where $C_{\text{Reg}} > 0$ is a constant. Furthermore, note that $\lambda_1$ corresponds to the largest spectral valued defined in \eqref{eqn:spectral-JA}. The following result connects Range$((PS)^{*})$ to the sequence $\langle f_{\alpha} , (\Lambda - \Lambda_0)f_{\alpha} \rangle_{\partial \Omega}$.
\begin{theorem}\label{thm:ps-range2}
    Let $P$ be the orthogonal projection onto ${\rm span}\{1\}^{\perp}$ and $(\Lambda - \Lambda_0) : H^{1/2}(\partial \Omega) \rightarrow H^{-1/2}(\partial \Omega)$ be a positive operator with the factorization $(\Lambda - \Lambda_0)=(PS)^{*}TPS$ such that $S:H^{1/2}(\partial \Omega) \rightarrow H^{7/2}(\partial \Omega)$ and $T:H^{7/2}(\partial D) \rightarrow H^{-7/2}(\partial D)$ are bounded linear operators. Assume that $S$ is compact and injective as well as $T$ being coercive on ${\rm Range} (PS)$. Then we have that
    \[\ell \in {\rm Range}((PS)^{*}) \quad \text{if and only if} \quad \liminf_{\alpha \rightarrow 0} \langle f_{\alpha} , (\Lambda - \Lambda_0)f_{\alpha}\rangle_{\partial \Omega} < \infty \]
where $ f_{\alpha}$ is the regularized solution to $(\Lambda - \Lambda_0)f=\ell$.
\end{theorem}
\begin{proof}
    $(\Lambda - \Lambda_0)$ is positive and compact such that ${\rm Null}(\Lambda -\Lambda_0) = span \{1\}$. Then, $(\Lambda - \Lambda_0)$ has the spectral decomposition \eqref{eqn:spectral-A}, where $(\Lambda - \Lambda_0)f_n = \lambda_n \ell_n$ for all $n \in \mathbb{N} \cup \{0\}$. Suppose $f_{\alpha}$ is the regularized solution to $(\Lambda - \Lambda_0) f = \ell$. From \eqref{eqn:spectral-A} and \ref{eqn:regularized-xn}, we have that 
    \[(\Lambda - \Lambda_0) f_{\alpha} = \sum_{n \geq 1} h_{\alpha}(\lambda_n) \langle f_n , \ell \rangle_{\partial \Omega} \, \ell_n\,\]
    where we also used the fact that ${\rm Null}(\Lambda -\Lambda_0) = \text{span} \{1\}$. Thus, we have that
    \begin{align*}
        \langle f_{\alpha} , (\Lambda - \Lambda_0) f_{\alpha} \rangle_{\partial \Omega} &= \big\langle \langle f,\ell_0 \rangle_{\partial \Omega} \, f_0 + \sum_{n\geq 1} \frac{h_{\alpha}(\lambda_n)}{\lambda_n} \langle f_n, \ell \rangle_{\partial \Omega} \, f_n, \sum_{m\geq 1} h_{\alpha}(\lambda_m) \langle f_m , \ell \rangle_{\partial \Omega} \, \ell_m \big\rangle_{\partial \Omega} \\
        &= \Big\langle \sum_{n\geq 1} \frac{h_{\alpha}(\lambda_n)}{\lambda_n} \langle f_n, \ell \rangle_{\partial \Omega} \, f_n, \sum_{m\geq 1} h_{\alpha}(\lambda_m)\langle f_m , \ell \rangle_{\partial \Omega} \, \ell_m \Big\rangle_{\partial \Omega} \\
        &= \sum_{n\geq 1} \bigg[ \frac{h_{\alpha}(\lambda_n)}{\lambda_n} \langle f_n , \ell \rangle_{\partial \Omega} \sum_{m \geq 1} h_{\alpha} (\lambda_m) \langle f_n , \ell \rangle_{\partial \Omega}  \langle f_n , \ell_m \rangle_{\partial \Omega} \bigg]\\
        &= \sum_{n\geq 1} \frac{h^2_{\alpha}(\lambda_n)}{\lambda_n} \rvert \langle f_n , \ell \rangle_{\partial \Omega}\rvert^2 
    \end{align*}
    by the sesquilinear definition of the dual-product $\langle \cdot , \cdot \rangle_{\partial \Omega}$, as well as the duality of bases specified in \eqref{eqn:basis-duality}. From the above equality, we now provide limiting bounds to $\langle f_{\alpha},(\Lambda - \Lambda_0)f_{\alpha}) \rangle_{\partial \Omega}$, which will allow us to invoke Theorem \ref{thm:ps-range1}. To this end, recall that $\{h_{\alpha}\}_{\alpha>0}$ is a family of filter functions which is uniformly bounded by 1. Thus we obtain the upper bound
    \[\langle f_{\alpha},(\Lambda - \Lambda_0)f_{\alpha}) \rangle_{\partial \Omega} \leq C^2_{\text{Reg}} \sum_{n\geq 1} \frac{1}{\lambda_n} \rvert \langle f_n , \ell \rangle_{\partial \Omega}\rvert^2.  \]
    For a lower bound, note that for any $N \in \mathbb{N}$, it holds that
    \[\langle f_{\alpha},(\Lambda - \Lambda_0)f_{\alpha}) \rangle_{\partial \Omega} \geq \sum^{N}_{n=1} \frac{h^2_{\alpha}(\lambda_n)}{\lambda_n} \rvert \langle f_n , \ell \rangle_{\partial \Omega}\rvert^2.\]
    Note that for any $1 \leq n \leq N$, the filter functions satisfy $h_{\alpha}(\lambda_n) \rightarrow 1$ as $\alpha \rightarrow 0$. Therefore, we obtain the lower bound
    \[\liminf_{\alpha \rightarrow 0} \langle f_{\alpha},(\Lambda - \Lambda_0)f_{\alpha}) \rangle_{\partial \Omega} \geq \sum^{N}_{n=1} \frac{1}{\lambda_n} \rvert \langle f_n , \ell \rangle_{\partial \Omega}\rvert^2.\]
    By taking $N \rightarrow \infty$, we obtain the inequalities
    \begin{align*}
        \sum_{n\geq 1} \frac{1}{\lambda_n} \rvert \langle f_n , \ell \rangle_{\partial \Omega}\rvert^2 \leq \liminf_{\alpha \rightarrow 0} \langle f_{\alpha},(\Lambda - \Lambda_0)f_{\alpha}) \rangle_{\partial \Omega} \leq C^2_{\text{Reg}} \sum_{n\geq 1} \frac{1}{\lambda_n} \rvert \langle f_n , \ell \rangle_{\partial \Omega}\rvert^2.
    \end{align*}
    The above inequalities and Theorem \ref{thm:ps-range1} proves the claim.
\end{proof}

\section{Kernel function on a circular region}\label{appendix-kernel}

Here we provide the details on the calculation for the kernel function defined in \eqref{kernel-function}. In polar coordinates, $\partial D$ is given by $\rho (\text{cos} ( \theta ), \text{sin} (\theta))$ for some constant $\rho \in (0,1)$. As similarly demonstrated in \cite{granados1}, since $\Omega$ is taken to be the unit disk in $\mathbb{R}^2$, we make the ansatz that $u(r,\theta)$ has the following series representation
\begin{equation}\label{u-annulus}
u(r , \theta) = a_0 + b_0 \, \text{ln} \, r + \sum_{|n|=1}^{\infty} \left[ a_n r^{|n|} + b_n r^{-|n|} \right] \text{e}^{\text{i}n \theta} \quad \text{in} \quad \Omega \backslash D
\end{equation}
whereas
$$u(r , \theta) = c_0 + \sum_{|n|=1}^{\infty} c_n r^{|n|} \text{e}^{\text{i}n \theta} \quad \text{in} \quad D.$$
Note that $u(r, \theta )$ is harmonic in both the annular and circular regions which are separated by the interior boundary $\partial D$. Given that the coefficients $\mu , \mu_s ,$ and $\ell_s$ are constant, we are able to determine the Fourier coefficients $a_n$ and $b_n$ by using the boundary conditions at $r=1$ and $r = \rho$ given by 
\[u(1,\theta ) = f (\theta), \quad u^{+}(\rho , \theta) = u^{-}(\rho , \theta)\]
\[\text{and} \quad \partial_{r}u^{+}(\rho , \theta) - \mu \partial_{r}u^{-}(\rho , \theta) = \frac{\mu_s}{\rho^2} \bigg(\frac{\ell_s^2}{\rho^2} \frac{\partial^{4}}{\partial \theta^{4}} - \frac{\partial^2}{\partial\theta^2} \bigg) u(\rho , \theta).\]
We let $f_n$ for $n \in \mathbb{Z}$ denote the Fourier coefficients for the prescribed body displacement $f$. Note, that the Dirichlet boundary
condition at $r=1$ above gives that
\[a_0 = f_0 \quad \text{and} \quad a_n + b_n = f_n \enspace \text{for all} \enspace n \neq 0.\]
From the first boundary condition at $r = \rho$, we have that
\[ b_0 = \frac{c_0 - f_0}{\text{ln}\,r} \quad \text{and} \quad b_n = \rho^{2|n|} (c_n - a_n).\]
From this, we conclude that $b_0 = 0$, which implies that $c_0 = f_0$. For $n \neq 0$, by defining \[ a_n= {\tilde a}_n f_n, \quad b_n=  {\tilde b}_n  f_n, \quad \text{and} \quad c_n =  {\tilde c}_n f_n \] we can set up a system of equations for $a_n, b_n$, and $c_n$ as follows:
\begin{align}
    \label{eq:DtN1}
  &   {\tilde a}_n + {\tilde b}_n = 1, \\    \label{eq:DtN2}
  & \rho^{2 |n|} {\tilde a}_n + {\tilde b}_n - \rho^{2 |n|} {\tilde c}_n = 0, \\
      \label{eq:DtN3}
  & |n| \rho^{|n|-1} {\tilde a}_n - |n| \rho^{-|n| - 1} {\tilde b}_n - \left[  \mu |n| \rho^{|n|-1} + \rho^{|n|-2} \left( \mu_s |n|^2 + \frac{1}{\rho^2} \mu_s \ell_s^2 |n|^4 \right) \right] {\tilde c}_n = 0.
\end{align}
We define
\[
\sigma_n = {\tilde a}_n- {\tilde b}_n \enspace \text{for} \enspace n \neq 0.
\]
Interchanging summation with integration gives us that 
\[ (\Lambda - \Lambda_0) f = \frac{1}{2 \pi} \int_0^{2\pi} K(\theta , \phi) f(\phi) \, \text{d}\phi \quad \text{where} \quad K(\theta , \phi) = \sum_{|n|=1}^{\infty}|n|(\sigma_n - 1)e^{\text{i}n(\theta - \phi)}.\]
Thus, in \eqref{kernel-function}, $\kappa_0 = 0$ and $\kappa_n = |n| (\sigma_n -1)$ for all $n \in \mathbb{Z}\setminus\{0\}$.
From solving the system \eqref{eq:DtN1}--\eqref{eq:DtN3}, we determine the coefficients of the DtN kernel operator to be given by
\begin{align}
\kappa_n & = |n| \frac{2 \rho^{2|n|} (\rho^3 (\mu-1) |n| + \rho^2 \mu_s |n|^2 + \mu_s \ell_s^2 |n|^4)}{\rho^3 |n| (\mu +1) + \rho^2 \mu_s |n|^2 + \mu_s \ell_s^2 |n|^4 - \rho^{2 |n|} \left( \rho^3 (\mu-1) |n| + \rho^2 \mu_s |n|^2 + \mu_s \ell_s^2 |n|^4  \right)}. \label{eq:Kdef}
\end{align}

\vspace{5mm}
\noindent {\bf \large Data availability statement:} Data sharing not applicable to this article as no datasets were generated or analyzed for this current study.

\vspace{5mm}
\noindent {\bf \large Conflict of interest:} The authors have no competing interests to declare that are relevant to the content of this article. 

\bibliographystyle{abbrv}
\bibliography{refs, researchbibmech}

\end{document}